\documentclass[reqno,10pt]{amsart}

\usepackage{epsf,color}
\usepackage{epsfig}
\usepackage{amsmath}
\usepackage{amscd}
\usepackage{amssymb}
\usepackage{enumerate}
\usepackage{amsfonts}
\usepackage{graphicx}
\usepackage[all]{xy}
\usepackage{mathrsfs}
\usepackage{setspace}
\usepackage{mathtools}
\usepackage{euscript}

\numberwithin{equation}{section}
\newcommand{\bC}{{\mathbb C}}

\newcommand{\bE}{{\mathbb E}}
\newcommand{\bF}{{\mathbb F}}

\newcommand{\bK}{{\mathbb K}}
\newcommand{\bL}{{\mathbb L}}

\newcommand{\bP}{{\mathbb P}}
\newcommand{\bQ}{{\mathbb Q}}
\newcommand{\bR}{{\mathbb R}}

\newcommand{\bT}{{\mathbb T}}

\newcommand{\bZ}{{\mathbb Z}}

\newcommand{\Z}{\mathbb{Z}}
\newcommand{\C}{\mathbb{C}}

\newcommand{\scrC}{\EuScript{C}}
\newcommand{\scrY}{\EuScript{Y}}
\newcommand{\scrV}{\EuScript{V}}
\newcommand{\scrG}{\EuScript{G}}
\newcommand{\scrT}{\EuScript{T}}
\newcommand{\scrZ}{\EuScript{Z}}
\newcommand{\scrP}{\EuScript{P}}

\newcommand{\scrF}{\EuScript{F}}
\newcommand{\scrW}{\EuScript{W}}
\newcommand{\scrA}{\EuScript{A}}
\newcommand{\scrB}{\EuScript{B}}

\newcommand{\scrM}{\EuScript{M}}

\newcommand{\cF}{{\mathcal F}}

\newcommand{\cI}{{\mathcal I}}

\newcommand{\sA}{{\mathscr A}}
\newcommand{\sB}{{\mathscr B}}

\newcommand{\sF}{{\mathscr F}}

\newcommand{\Sym}{\operatorname{Sym}}

\newcommand{\End}{\operatorname{End}}
\newcommand{\Hom}{\operatorname{Hom}}

\newcommand{\id}{\operatorname{id}}

\newcommand{\rk}{\operatorname{rk}}

\newcommand{\Symp}{\operatorname{Symp}}
\newcommand{\Diff}{\operatorname{Diff}}
\newcommand{\Ham}{\operatorname{Ham}}
\newcommand{\Ob}{{\mathcal Ob}}

\newcommand{\Conf}{\operatorname{Conf}}
\newcommand{\im}{\mathrm{im}}

\newtheorem{thm}{Theorem}[section]

\newtheorem{Theorem}[thm]{Theorem}
\newtheorem{Corollary}[thm]{Corollary}
\newtheorem{Lemma}[thm]{Lemma}
\newtheorem{Proposition}[thm]{Proposition}
\newtheorem{Definition}[thm]{Definition}
\newtheorem{Conjecture}[thm]{Conjecture}
\newtheorem{Addendum}[thm]{Addendum}

\theoremstyle{remark}
\newtheorem{Remark}[thm]{Remark}

\newtheorem{Example}[thm]{Example}

\newtheorem*{Notation}{Notation}

\newcommand{\comment}[1]{}

\newcommand{\Acknowledgements}{{\em Acknowledgements.} }

\title[{Floer cohomology and pencils of quadrics}]{Floer cohomology and pencils of quadrics}

\author[Ivan Smith]{Ivan Smith} \date{October 2011.}

\begin{document}

\begin{abstract}
There is a classical relationship in algebraic geometry between a hyperelliptic curve and an associated pencil of quadric hypersurfaces.  We investigate symplectic aspects  of this relationship, with a view to applications in low-dimensional topology.
\end{abstract}

\maketitle
\begin{footnotesize}
\setcounter{tocdepth}{2}
\tableofcontents
\end{footnotesize}

\section{Introduction}

\subsection{Two contexts}  The main result of this paper, Theorem \ref{Thm:Embed},  belongs to categorical symplectic topology in the sense of Donaldson, Fukaya and Kontsevich.  It relates the derived Fukaya categories of a genus $g$ surface $\Sigma_g$ and of the complete intersection $Q_0 \cap Q_1$ of two smooth quadric hypersurfaces in $\bP^{2g+1}$.  The result is of interest in at least two different contexts.

\begin{itemize}
\item The intersection of two quadric 4-folds in $\bP^5$ is also a moduli space of solutions to the anti-self-dual Yang-Mills equations on the product $\Sigma_2 \times S^1$ \cite{Newstead, NS}.  Theorem \ref{Thm:Embed}, in the special case $g=2$, can be seen as an instance of the ``Seiberg-Witten equals Donaldson" philosophy for gauge-theory invariants of 3-manifolds.

\item Theorem \ref{Thm:Embed} can also be viewed as a symplectic analogue of a classical theorem of Bondal and Orlov \cite{BondalOrlov} in algebraic geometry, concerned with derived categories of sheaves on the same spaces.  The passage here from algebraic to symplectic geometry fits into the broader program of Kontsevich's Homological Mirror Symmetry conjecture.
 \end{itemize}
 After formulating a precise version of the main theorem, the rest of the Introduction will flesh out these two contexts, and then indicate the basic strategy of the proof,  which is itself motivated -- via mirror symmetry -- by several classical theorems on derived categories of sheaves.  

\subsection{The theorem}

Let $\Sigma_g$ be a closed surface of genus $g\geq 2$, equipped with a symplectic form.   This has a well-defined balanced Fukaya category $\scrF(\Sigma_g)$, linear over $\bC$.  Let $\{Q^{2g}_t\}_{t\in \bP^1} \subset \bP^{2g+1}$ denote a pencil of smooth $2g$-dimensional quadric hypersurfaces in $\bP^{2g+1}$, with smooth base locus $Q^{2g}_0 \cap Q^{2g}_1$.  By Moser's theorem, the symplectic manifold underlying the complete intersection $Q_0 \cap Q_1$ is  independent of the choice of generic pencil of quadrics.  It is a simply-connected Fano variety,  in particular a monotone symplectic manifold with a well-defined monotone Fukaya category $\scrF(Q_0\cap Q_1)$, again defined over $\bC$.   The latter category actually splits into a collection of mutually orthogonal $A_{\infty}$-subcategories, one for each eigenvalue of quantum multiplication by the first Chern class on the quantum cohomology of $Q_0\cap Q_1$.   We denote by $D^{\pi}\scrC$ the cohomological category $H(Tw^{\pi}\scrC)$ underlying the split-closure of the category of twisted complexes of an $A_{\infty}$-category $\scrC$.  Our main result is:

\begin{Theorem} \label{Thm:Embed}
There is a $\bC$-linear equivalence of $\bZ_2$-graded split-closed triangulated categories 
\[
D^{\pi}\scrF(\Sigma_g) \simeq D^{\pi}\scrF(Q_0^{2g} \cap Q_1^{2g};0)
\]
where $\scrF(\bullet;0)$ denotes the summand corresponding to the 0-eigenvalue of quantum cup-product by the first Chern class.
\end{Theorem}

The proof  of Theorem \ref{Thm:Embed} given here relies essentially on the  (typically non-geometric) passage to idempotent completion.  

If $\Sigma_g \rightarrow \bP^1$ is a hyperelliptic curve, branched over $\{\lambda_1, \ldots, \lambda_{2g+2}\} \subset \bC$, it determines a $(2,2)$-complete intersection
\[
Q_0 \cap Q_1 \ = \ \left ( \sum z_j^2 = 0 \right ) \, \cap \, \left ( \sum \lambda_j z_j^2 = 0 \right ) \ \subset \, \bP^{2g+1}.
\]
Varying the $\lambda_j$ in $\bC$, one obtains a natural action, by parallel transport, of the hyperelliptic mapping class group $\Gamma_{g,1}^{hyp}$ of a once-pointed curve on each of $\Sigma_g$ and $Q_0 \cap Q_1$.  (There is no universal hyperelliptic curve over configuration space $\Conf_{2g+2}(\bP^1)$  \cite{Mess}; we have constrained the $\lambda_j$ to lie in $\bC\subset \bP^1$, which accounts for the appearance of once-pointed curves.)

\begin{Addendum} \label{Thm:Action}
 The equivalence of Theorem \ref{Thm:Embed} is compatible with the weak action of $\Gamma_{g,1}^{hyp}$. 
 \end{Addendum}
 
 Considering the situation when branch points go to infinity more carefully, there is a non-split finite extension
\[
1 \rightarrow \bZ_2^{2g} \longrightarrow \tilde{\Gamma}^{hyp}_g \longrightarrow \Gamma^{hyp}_g \rightarrow 1
\]
of the classical hyperelliptic mapping class group which acts, via parallel transport, by symplectomorphisms of $Q_0 \cap Q_1$, cf. Remark \ref{Rem:BranchingInfinity}.  A fairly direct consequence of Theorem \ref{Thm:Embed} is:

\begin{Corollary} \label{Cor:Faithful2}
The natural representation $\tilde{\Gamma}^{hyp}_g \rightarrow \pi_0\Symp(Q_0 \cap Q_1)$ is faithful.
\end{Corollary}

Classical surgery theory \cite[Theorem 13.5]{Sullivan} shows  $\tilde\Gamma_g^{hyp} \rightarrow \pi_0\Diff(Q_0 \cap Q_1)$ has infinite kernel.   More surprisingly, the $\Gamma_{g,1}^{hyp}$-action extends to a faithful weak action of the full mapping class group $\Gamma_g$ by autoequivalences of $D^{\pi}\scrF(Q_0\cap Q_1)$, which has no obvious direct geometric construction.  The construction of mapping class group actions on triangulated categories is a well-known problem in that part of representation theory concerned with categorification.

\subsection{Entropy}

Corollary \ref{Cor:Faithful2} is essentially equivalent to a related dynamical statement. For a symplectic manifold $X$  the (conjugation-invariant) \emph{Floer-theoretic entropy} of  a mapping class $\phi \in \Symp(X)/\Ham(X)$ is
\[ 
h_{Floer}(\phi) \ = \ \limsup \frac{1}{n} \log \rk\, HF(\phi^n).
\]
This is a kind of robust version of the periodic entropy, robust in the sense that it depends on a symplectic diffeomorphism only through its mapping class; by contrast topological and periodic entropy are typically very sensitive to perturbation.  
For area-preserving diffeomorphisms of a surface $\Sigma$, the Floer-theoretic entropy (of the action on $\Sigma$ itself) co-incides with the minimum of topological or periodic entropy amongst representatives of the mapping class  \cite{Cotton-Clay,FLP, Felshtyn}; moreover, $h_{Floer}(\phi)>0 \Leftrightarrow \phi$ has a pseudo-Anosov component.   In this low dimension, these phenomena are basically detected by the fundamental group. 

\begin{Theorem} \label{Thm:Entropy}
A diffeomorphism $\phi \in \tilde{\Gamma}_g^{hyp}$  has a  pseudo-Anosov component if and only if the induced map on $Q_0 \cap Q_1$ has positive Floer-theoretic entropy.
\end{Theorem}

This seems to be one of the few computations of non-zero Floer-theoretic entropy in a closed symplectic manifold which is  not detected by classical topology.   Analogous results for symplectomorphisms of certain $K3$ surfaces can be derived from combining \cite{Seidel:graded} and \cite{KhovanovSeidel}.

\subsection{Representation varieties}\label{Sec:RepVar}

Fix a point $p\in \Sigma_g$ and  denote by $\scrM(\Sigma_g)$ the moduli space of rank two, fixed odd determinant stable bundles on a complex curve of genus $g$, equivalently the space of conjugacy classes of $SU(2)$-representations of $\pi_1(\Sigma\backslash\{p\})$ which have holonomy $-I$ at the puncture.  $\scrM(\Sigma_g)$ admits a natural symplectic structure; the symplectic volume of $\scrM(\Sigma_g)$ was computed by Witten  and Jeffrey-Weitsman \cite{Witten, JW}, and techniques from number theory, gauge theory and birational geometry have all been brought to bear on understanding its cohomology  \cite{NS, AtiyahBott, Thaddeus}.  The connection to pencils of quadrics goes back to Newstead \cite{Newstead}, who constructed an isomorphism  $\scrM(\Sigma_2) \cong Q^{4}_0\cap Q^{4}_1 \subset \bP^5$.  Theorem \ref{Thm:Embed} therefore computes part of the Fukaya category of $\scrM(\Sigma_2)$.  

In the topological setting, there is an extension $\widehat\Gamma(\Sigma_g)$  of the (full, not hyperelliptic) classical mapping class group $\Gamma_g=\pi_0\Diff^+(\Sigma_g)$ by $\bZ_2^{2g} = H^1(\Sigma_g;\bZ_2)$ which acts  on the moduli space via  a homomorphism $\widehat\Gamma(\Sigma_g) \rightarrow \Symp \scrM(\Sigma_g)$.       Goldman \cite{Goldman}  proved $\widehat{\Gamma}(\Sigma_g)$ acts ergodically.  Let $\hat\rho: \widehat\Gamma(\Sigma_g) \rightarrow \pi_0\Symp\scrM(\Sigma_g)$ denote the associated representation on groups of components.

\begin{Theorem}  \label{Thm:NotSymplectic}
For $g\geq 2$,  $\hat{\rho}: \widehat\Gamma(\Sigma_g) \longrightarrow \pi_0\Symp \scrM(\Sigma_g)$ does not factor through the symplectic group (i.e. $\hat\rho$ is non-trivial on the Torelli group).  When $g=2$, $\hat\rho$ is faithful.
\end{Theorem}

The second statement is just Corollary \ref{Cor:Faithful2}, and it implies the first statement by an argument of Wehrheim and Woodward which we recall in Corollary \ref{Cor:TorelliNontrivial}. By contast, the action of $\widehat{\Gamma}(\Sigma_g)$ on $H^*(\scrM(\Sigma_g))$ factors through the symplectic group. Theorem \ref{Thm:NotSymplectic} answers a question of Dostoglou and Salamon \cite[Remark 5.6]{DS2}. Donaldson's former student Michael Callahan proved (unfinished Oxford D. Phil thesis, circa 1993) that $\hat\rho$ distinguished the Dehn twist on a separating curve $\sigma \subset \Sigma_2$ from the identity. Callahan apparently used the gauge theoretic methods recalled in Section \ref{Subsec:3manifolds}, which seem less well suited to treat the general case.   A related faithfulness result, for the action of a cousin of the 5-strand spherical braid group on a four-dimensional moduli space $\bP^2\# 5\overline{\bP}^2$ of parabolic bundles on $S^2$,  was established by Seidel \cite[Example 2.13]{Seidel:4dim} using Gromov-Witten invariants and positivity of intersections of closed holomorphic curves in 4-manifolds.  

 There is also a fixed-point theorem valid at any genus: this is something of a digression from the main theme of the paper, but has some relevance in light of the conjectural description of $\scrF(\scrM(\Sigma_g))$ discussed in Section \ref{Sec:Speculative} below, and is needed for Corollary \ref{Cor:AllSummands}.

\begin{Theorem} \label{Thm:FixPoint}
If $\phi \in \Symp \scrM(\Sigma_g)$ represents a class in $im(\hat{\rho})$ then $\phi$ fixes a point of $\scrM(\Sigma_g)$. 
\end{Theorem}

If $\psi \in \widehat\Gamma(\Sigma_g)$ lies in the Torelli group, then the Lefschetz number of $\hat\rho(\psi)$ is zero, and $\hat\rho(\psi)$ is isotopic to a diffeomorphism of $\scrM(\Sigma_g)$ without fixed points.  Thus Theorem \ref{Cor:Faithful2} and \ref{Thm:FixPoint}  both detect essentially symplectic phenomena.   There is an obvious correspondence between fixed points of the action of $\phi$ on $\scrM(\Sigma_g)$, and representations of the fundamental group of the mapping torus of $\phi$.  This gives a purely topological application of the main theorems.

\begin{Corollary} \label{Cor:RepnGrowth}
Let $Y \rightarrow S^1$ be a closed 3-manifold fibring over the circle with fibre $\Sigma_g$.
\begin{itemize} 
\item There is a non-abelian $SO(3)$-representation of $\pi_1(Y)$.
\item Suppose $g=2$ and $Y$ is hyperbolic.  There are cyclic degree $d$ covers $Y_d\rightarrow Y$ for which the number of conjugacy classes of non-abelian $SO(3)$-representations of $\pi_1(Y_d)$  grows exponentially in $d$.
\end{itemize}
\end{Corollary}

The first result was proved by Kronheimer-Mrowka \cite{KM2} using gauge theory; the second (which should generalise to higher genus fibres) appears to be new though closely related results are in the literature, cf. Section \ref{Subsec:3manifolds}. In any case, it's interesting that these are amenable to techniques of symplectic geometry.

\subsection{Instanton Floer homology}\label{Sec:InstFloer}

 Let $f: Y_h \rightarrow S^1$ be a fibred 3-manifold with a distinguished section, defined by a monodromy map $h \in \Gamma(\Sigma_{g,1})$ in the mapping class group of a once-marked surface $f^{-1}(pt)$.  Up to isomorphism there is a unique non-trivial $SO(3)$-bundle $E\rightarrow f^{-1}(pt)$, with $\langle w_2(E),  f^{-1}(pt)\rangle \neq 0$; the section of $Y_h \rightarrow S^1$ defines a distinguished extension of $E$ to an $SO(3)$ bundle (still denoted) $E\rightarrow Y_h$.  If $Y_h$ is a homology $S^1 \times S^2$, this is the unique non-trivial $SO(3)$-bundle over $Y_h$ up to isomorphism.  Any such $E$  admits no reducible flat connexion, hence is ``admissible" in the sense of Donaldson \cite{Donaldson:Floer}, so there is an associated $\bZ_4$-graded instanton Floer homology group  $HF_{inst}(Y_h; E)$ computed from the Chern-Simons functional on connexions in $E$. Now $h$  acts canonically on $\scrM(f^{-1}(pt))$ via an element $\hat{h} \in \widehat{\Gamma}(f^{-1}(pt))$, and the famous theorem of Dostoglou and Salamon \cite{DS} gives an isomorphism $HF_{inst}(Y_h; E) \cong HF(\hat{\rho}(\hat{h}))$.   Appealing in addition to Theorems \ref{Thm:Embed} and (the proof of) \ref{Thm:FixPoint} yields:

\begin{Corollary} \label{Cor:AllSummands}
Fix a  3-manifold $Y_{h} \rightarrow S^1$ fibred by genus 2 curves, which in addition is a homology $S^1 \times S^2$, and let $E \rightarrow Y_h$ be the non-trivial $SO(3)$-bundle.  There is an isomorphism of $\bZ_2$-graded $\bC$-vector spaces
\begin{equation} \label{Equation:HFinstanton}
HF_{inst}(Y_{h}; E) \ \cong \ \bC \oplus HF(h) \oplus \bC.
\end{equation}
\end{Corollary}

By work of Cotton-Clay \cite{Cotton-Clay} the central summand $HF(h)$ on the right side, which is computed on the curve $\Sigma_2$ itself rather than on any associated moduli space, can be algorithmically computed from  a presentation of $h$ as a word in positive and negative Dehn twists, or from a suitable traintrack.   In general, if we make no assumption on the action of $h$ on $H_1(\Sigma_2; \bZ)$, the instanton Floer homologies $HF_{inst}(Y_h; E)$ of the different principal bundles $E \rightarrow Y_h$ arising as mapping tori of $h$ are given by Hochschild cohomologies $\bC \oplus HH^*(h_{\xi}) \oplus \bC$, with $h_{\xi}$ the various autoequivalences of $\scrF(\Sigma_2)$  obtained by combining the action of $h$ and tensoring by some flat line bundle $\xi \in H^1(\Sigma_2;\bZ_2)$.  We should emphasise that it seems to be very hard to  compute these instanton Floer homology groups  by direct computations on $Y_h$, or even to write down the underlying chain groups.

\subsection{Witten's conjecture} \label{Sec:Speculative}

We include a brief speculative discussion, which should nonetheless set these topological results in a helpful context.  The space $\scrM(\Sigma_g)$ is always a simply-connected Fano variety, symplectomorphic to the variety of $(g-2)$-dimensional linear subspaces of $Q_0\cap Q_1 \subset \bP^{2g+1}$, hence there is an evaluation from the total space of a projective bundle
\[
\bP^{g-2} \, \widetilde{\times} \, \scrM(\Sigma_g) \longrightarrow Q_0\cap Q_1.
\]
Probably deeper analysis of the birational structure of this map would help relate the Fukaya category of the intersection of quadrics to that of the moduli space.  Explicitly:

\begin{Conjecture} \label{Conj:SecondHighest}
There is a $\bC$-linear equivalence of $\bZ_2$-graded triangulated categories
\[
D^{\pi} \scrF(\Sigma_g) \ \simeq \ D^{\pi} \scrF(\scrM(\Sigma_g); 4(g-2)).
\]
\end{Conjecture}

We give a slightly more precise statement in Remark \ref{Rmk:PreciseConj}. Conjecture \ref{Conj:SecondHighest} should be seen in light of Witten's conjecture relating Seiberg-Witten and Donaldson invariants in low-dimensional topology.  The spaces $\scrM(\Sigma_g)$ arise as moduli spaces of instantons on $\Sigma_g \times S^1$.  The corresponding moduli spaces of Seiberg-Witten solutions on $\Sigma_g \times S^1$ are given by symmetric products of $\Sigma_g$.    A theorem of Munoz \cite{Munoz} implies the spectrum of quantum multiplication by the first Chern class on $\scrM(\Sigma_g)$ is given (dropping some factors of $i$) by 
\[
\{-4(g-1), -4(g-2), \ldots, -4, 0, 4, \ldots, 4(g-2),  4(g-1)\}.
\]
The Fukaya category $\scrF(\scrM(\Sigma_g))$  breaks into summands indexed by these values.  Moreover, the summands for $\pm \lambda$ should be equivalent.  Naively, one would expect the Fukaya category summand $\scrF(\scrM(\Sigma_g); 4(g-k-1))$ to be built out of the Fukaya category of the symmetric product $\Sym^k(\Sigma_g)$, when $|k| < g$ (this is an intentionally vague statement, not least because we have not specified a symplectic form on the symmetric product).  In particular, the summands corresponding to the outermost eigenvalues $\pm 4(g-1)$ are expected to be semisimple (equivalent to the Fukaya category of a point).  Whilst we do not prove this, Theorem \ref{Thm:FixPoint} provides rather strong evidence. More precisely, that theorem is proved by showing that the symplectic Floer cohomology $HF(\phi)$ has a distinguished rank one summand, given by its generalized eigenspace for the eigenvalue $4(g-1)$.  This should be compared to similar rank one pieces of monopole or Heegaard Floer homology \cite{OzSz,KronheimerMrowka:Sutured}, in that context associated to the ``top" $Spin^c$-structure.  In this vein, Conjecture \ref{Conj:SecondHighest} -- or Theorem \ref{Thm:Embed}, when $g=2$ -- can be viewed as an instance of the general ``Seiberg-Witten = Donaldson" philosophy, for invariants of mapping tori coming from the next highest $Spin^c$-structure.  

Practically, one might hope -- by analogy with derived categories of sheaves, cf. Section \ref{Sec:HMS} below -- that Fukaya categories behave well under both blowing up and down, and under passing to the total space of a projective bundle.  Then 
Thaddeus' work on flip diagrams \cite{Thaddeus}, in which flips along projective bundles over symmetric products explicitly relate a projective bundle over $\scrM(\Sigma_g)$ to a projective space, would seem to give a cut-and-paste route to assembling $\scrF(\scrM(\Sigma_g))$ from the $\scrF(\Sym^k(\Sigma_g))$.  In reality, it seems hard to give a complete argument on these lines using current technology. The intermediate spaces in Thaddeus' work have less good monotonicity properties than $\scrM(\Sigma_g)$ itself, whilst there is some delicacy in the choice of symplectic form on $\Sym^k(\Sigma_g)$. Note that the latter delicacy is absent precisely for the summands corresponding to $\Sym^k(\Sigma_g)$ for $k=0,1$, which are those treated in this paper.

\subsection{Homological Mirror Symmetry}\label{Sec:HMS}

The proof of Theorem \ref{Thm:Embed} involves establishing Fukaya category analogues of three well-known results concerning derived categories of sheaves.

\begin{itemize}
\item (Bondal-Orlov) The derived category of sheaves on the intersection of two even-dimensional quadrics has a semi-orthogonal decomposition  in which one factor is the  derived category of sheaves on an underlying hyperelliptic curve \cite{BondalOrlov}.
\item (Kapranov) The derived category of sheaves on an even-dimensional quadric has a semi-orthogonal decomposition in which one factor is the derived category of sheaves on a pair of points \cite{Kapranov}.
\item (Bondal-Orlov) If $X \rightarrow Y$ is the blow-up along a smooth centre $B\subset Y$ of codimension at least 2, the derived category of sheaves on $B$ embeds in the derived category of sheaves on $X$ \cite{BondalOrlov}.
\end{itemize}

Obviously, Theorem \ref{Thm:Embed} is exactly the symplectic analogue of the first of these results.  Our proof of Theorem \ref{Thm:Embed} relates both categories appearing in its statement with the Fukaya category of the \emph{relative quadric} $Z$, given by blowing up $\bP^{2g+1}$ along $Q_0\cap Q_1$. We construct  equivalences with quasi-isomorphic images
\[
D^{\pi}\scrF(\Sigma_g) \ \hookrightarrow \ D^{\pi}\scrF(Z) \ \hookleftarrow  \ D^{\pi}\scrF(Q_0\cap Q_1;0).
\]
Although it occupies only a small part of the final proof, the geometric heart of $\hookrightarrow$ and arguably of Theorem \ref{Thm:Embed} is Section \ref{Section:Quadric}, which gives a derived equivalence between the nilpotent summand of the Fukaya category of a quadric $Q \subset \bP^{2g+1}$ and the Fukaya category of a pair of points $S^0$.     This is precisely the symplectic analogue of Kapranov's theorem.  From that perspective, a hyperelliptic curve and the total space of the associated pencil of quadrics behave, categorically, like fibrations with the same fibre and monodromy.   To pass between $D^{\pi}\scrF(Z)$ and $D^{\pi} \scrF(Q_0 \cap Q_1)$ involves relating the derived Fukaya category of a base locus and of a blow-up, akin to the second mentioned theorem of Bondal-Orlov (in this case, however, the symplectic story is carried out in more restricted circumstances, for blow-ups of codimension two complete intersections satisfying a raft of convenient hypotheses). In light of these analogies, one is led to  compare the chains of spaces
\begin{equation} \label{Eqn:Kuznetsov}
\begin{array}{ccccccc} 
\emptyset & \subset & 2 \ \textrm{points} & \subset & \Sigma_g & \subset & X \\
\updownarrow &  &\updownarrow & &\updownarrow & & \updownarrow \\ 
\bP^{2g+1} & \supset & Q & \supset & Q_0\cap Q_1 & \supset & X^{\vee} 
\end{array}
\end{equation}
where the top line corresponds to double branched covers of projective space of dimensions $d=-1,0,1,2$ over divisors of degree $2g+2$, and the lower line the base locus of a $d$-dimensional family of quadrics in $\bP^{2g+1}$.   There are equivalences between the Fukaya category of a space in the top line, and the nilpotent summand of the Fukaya category of the space below it, in each of the first 3 columns, mirror to a  small part of Kuznetsov's ``homological projective duality" \cite{Kuznetsov}.  The situation in the final column is less clear. For instance, when $g=2$, both $X$ and $X^{\vee}$ are $K3$ surfaces, but in general they don't have equivalent derived categories of coherent sheaves unless one introduces a twist by a class in the Brauer group.  It would be interesting to know whether this has any symplectic analogue. The hypersurface of degenerate quadrics is singular in codimension 3, so technical difficulties arise in pushing this any further.

Given the proposed  Landau-Ginzburg mirror $w:Y_{Q_0 \cap Q_1} \rightarrow \bC$ of  $Q_0 \cap Q_1$ \cite{Katzarkov, Przyj}, and the proofs by Seidel and Efimov of mirror symmetry for curves of genus $\geq 2$,  \cite{Seidel:HMSgenus2, Efimov}, Theorem \ref{Thm:Embed}  essentially proves one direction of Homological Mirror Symmetry for the  Fano variety which is the complete intersection of two quadrics:
\[ D^{\pi}\scrF(Q_0 \cap Q_1;0) \ \simeq \ D^b_{sing}(Y_{Q_0 \cap Q_1},w|_{\text{open subset}}) \]
(restricting $w$ throws out certain ``massive modes").  In particular, this gives the first proof of HMS for a moduli space of bundles on a curve of higher genus.

\begin{Remark}
Homological Mirror Symmetry typically relates symplectic and algebraic geometry.   The relation of Theorem \ref{Thm:Embed} to HMS is roughly as follows. The mirror of a Fano variety $X$ is a Landau-Ginzburg model $w: Y \rightarrow \bC$; blowing up the Fano $X' \rightarrow X$ is expected to introduce new singular fibres $Y \subset Y' \stackrel{w'}{\rightarrow} \bC$ into the LG-model.  Blow-ups induce semi-orthogonal decompositions of the derived category $D^b(X') = \langle D^b(X), \mathcal{C}\rangle$, hence corresponding semi-orthogonal decompositions relating the Fukaya-Seidel categories of the mirrors.  On the other hand, the new singular fibre of $(Y',w')$  leads to a new summand to its derived category of singularities, which should correspond to a new summand in the Fukaya category of  $X'$ (and perhaps some bulk deformation of the summands coming from $X$).   For more on the mirror symmetric background and motivation for Theorem \ref{Thm:Embed}, see \cite{KKOY}.  

\end{Remark}

\subsection{Another birational motivation} \label{Sec:BirModel}
There is another description of $\scrM(\Sigma_2) = Q_0 \cap Q_1 \subset \bP^5$ which, whilst not used in this paper, is undoubtedly related to the $g=2$ case of Theorem \ref{Thm:Embed}, and provided motivation for the result.
 Let's say that a line $l \subset Q_0\cap Q_1$ is \emph{generic} if its normal bundle is holomorphically trivial $\mathcal{O} \oplus \mathcal{O}$ (rather than $\mathcal{O}(i) \oplus \mathcal{O}(-i)$.  Recall that a genus 2 curve, whilst not a complete intersection, admits an embedding as a $(2,3)$-curve on a quadric hypersurface $\bP^1\times \bP^1 \subset \bP^3$.  

\begin{Lemma}
The blow-up $W \rightarrow \scrM(\Sigma_2)$ at a generic line inside $Q_0\cap Q_1$ is isomorphic to the blow-up $W' \rightarrow \bP^3$ along an embedded genus 2 curve of degree 5.
\end{Lemma}

A proof is given in \cite{GH}.  The diagram $\bP^3 \leftarrow Bl_C(\bP^3) = Bl_l(Q_0\cap Q_1) \rightarrow Q_0\cap Q_1$ is also the simplest of the ``flip-diagrams" constructed using stable pairs by Thaddeus in \cite{Thaddeus} (in this special case there are actually no flips, see \emph{op. cit.} section 3.19; the map to $Q_0\cap Q_1$ is his non-abelian Abel-Jacobi map). 

If one believes that Fukaya categories should behave well under birational transformations -- a view espoused by Katzarkov and his collaborators, see for instance \cite{KKOY}, and encouraged by the particular cases treated in this paper -- then one expects $\scrF(W)$ to be built out of $\scrF(\bP^1)$, which has two semisimple summands,  and of $\scrF(\scrM(\Sigma_2))$, whilst $\scrF(W')$ should be assembled from $\scrF(\bP^3)$ (four semisimple summands) and $\scrF(\Sigma_2)$.  Equating $\scrF(W)$ and $\scrF(W')$ would quickly give Theorem \ref{Thm:Embed} in the case $g=2$. Whilst the argument we give is more roundabout, it seems technically simpler (because of its appeal to the special geometric features of Lefschetz fibrations), and also applies in greater generality ($g \geq 2$).

\subsection{Outline of proof}

As mentioned previously, 
our proof of Theorem \ref{Thm:Embed} constructs  equivalences with quasi-isomorphic images
\[
D^{\pi}\scrF(\Sigma) \ \hookrightarrow \ D^{\pi}\scrF(Z) \ \hookleftarrow  \ D^{\pi}\scrF(Q_0\cap Q_1;0).
\]
Both $\scrF(\Sigma_g)$ and $\scrF(Q_0 \cap Q_1)$ are naturally $\bZ_{4g-4}$-graded, but only the underlying $\bZ_2$-graded categories are equivalent; from this perspective, that is because the intermediate category $\scrF(Z)$ only admits a $\bZ_2$-grading.   At a technical level, the proof of Theorem \ref{Thm:Embed} combines five principal ingredients.

\begin{enumerate}
\item \emph{Lefschetz fibrations}  provide the basic geometric setting in which all the investigations take place.  Thus, a hyperelliptic curve and the relative quadric $Z$ are viewed as Lefschetz fibrations with categorically-equivalent fibres.
\item \emph{Eigenvalue splittings} for Fukaya categories; in particular, we derive sufficient conditions for a collection of Lagrangian spheres to split-generate a particular summand.
\item \emph{Quilt theory}, in the sense of Mau-Wehrheim-Woodward; specifically, we appeal to their work to construct embeddings of categories associated to (idempotent summands of)  Lagrangian correspondences which arise from blowing up certain codimension two complete intersections.  
\item \emph{Surgery} and \emph{$\bZ$-gradings}, whereby certain holomorphic polygons are constrained by knowledge of related polygons for a Lagrange surgery at an isolated intersection point, or by their intersection numbers with divisors of poles of holomorphic volume forms.
\item \emph{Finite determinacy}, introduced in this setting in Seidel's beautiful paper \cite{Seidel:HMSgenus2}, which relies on Kontsevich's formality theorem to describe certain $A_{\infty}$-structures in terms of polyvector fields.  Eventually, the Fukaya categories of the curve and the relative quadric are identified by singling out a quasi-isomorphism class of $A_{\infty}$-structures on an exterior algebra with the desired Hochschild cohomology. 
\end{enumerate}

These ingredients are assembled as general tools in Section \ref{Section:Fukaya}, and deployed in the subsequent parts of the paper.  The connection $\hookleftarrow$ between $\scrF(Z)$ and  the Fukaya category of the base locus seems to be part of a wider phenomenon for blow-ups; it also makes a connection to a well-known ``spinning" construction of Lagrangian spheres, which provides a link between algebra and geometry useful for establishing Addendum \ref{Thm:Action}.   The material on Fukaya categories of blow-ups should be compared to the work of Abouzaid, Auroux and Katzarkov \cite{AAK}, who derive closely related results from the ``Strominger-Yau-Zaslow" perspective of Lagrangian torus fibrations; similar techniques are also applied in Seidel's recent \cite{Seidel:Flux}.  These results might be viewed as an open-string counterpart to the symplectic birational geometry programme of Yongbin Ruan and his collaborators.
  
\begin{Remark} \label{Rem:OurSpacesAreMonotone}
The symplectic manifolds which occur in the paper are real surfaces or Fano varieties, hence are monotone with minimal Chern number $\geq 1$, and the Lagrangian submanifolds we consider are monotone of minimal Maslov number $\geq 2$.  These hypotheses considerably simplify the definition of $\scrF(\bullet; \lambda)$, which can then be constructed using essentially classical tools, and allow us to make systematic use of the quilted Floer theory of Mau, Wehrheim and Woodward \cite{MWW:Announce}. We emphasise that Fukaya categories  enter the main argument in a rather formal manner, and the properties we require would hold independent of the finer details of the construction. That underlying construction is nonetheless a  very substantial undertaking, and we are accordingly indebted to the foundational work of Fukaya-Oh-Ohta-Ono, of Seidel, and of Mau-Wehrheim-Woodward.  \end{Remark}

\begin{small}
\noindent \emph{Organisation of the paper}.  Section \ref{Sec:Top} gives background on the spaces $\scrM(\Sigma_g)$ and derives  the results of Sections \ref{Sec:RepVar} and \ref{Sec:InstFloer}, assuming Theorem \ref{Thm:Embed}.   This section is largely self-contained, though at one or two points we refer ahead for general Floer-theoretical results.  Such results are collected in  Section \ref{Section:Fukaya}.  The proof of Theorem \ref{Thm:Embed} itself takes up the remainder of the paper, Sections \ref{Sec:BaseToQuadric} and \ref{Sec:CurveToQuadric}.  A low-dimensional topologist happy to take the proof of Theorem \ref{Thm:Embed} on faith can stop reading at the end of Section \ref{Sec:Top}. A symplectic topologist interested in Theorem \ref{Thm:Embed} but unconcerned with 3-manifolds can start reading from Section \ref{Section:Fukaya}.
\end{small}

\vspace{0.5cm}

\noindent \Acknowledgements  I am extremely indebted to Mohammed Abouzaid, Denis Auroux and Paul Seidel for all their help and their numerous patient explanations. This paper is written very much in the shadow of Seidel's \cite{Seidel:HMSgenus2}.  Conversations with Tom Bridgeland, Alex Dimca, Tobias Dyckerhoff, Sasha Efimov, Tobias Ekholm, Rahul Pandharipande, Tim Perutz, Catharina Stroppel and Burt Totaro were also influential.  Detailed and insightful comments from the anonymous referees have dramatically improved the paper,  both clarifying and simplifying the exposition. \newline 

\noindent The author is partially supported by grant ERC-2007-StG-205349 from the European Research Council. Part of this work was undertaken whilst the author held a Research Professorship at the Mathematical Sciences Research Institute, Berkeley; thanks to MSRI for hospitality.
%%%%%%%%%%%%%%%%%%%%%%%%%%%%%%%%%%%%%%%%%%%%
%%%%%%%%%%%%%%%%%%%%%%%%%%%%%%%%%%%%%%%%%%%%

\section{The representation variety}\label{Sec:Top}

\subsection{Topology\label{Sec:Topology}}
Recall that $\scrM(\Sigma_g)$ denotes the variety of twisted representations of the fundamental group of a punctured surface $\Sigma_g\backslash \{p\}$:
\begin{equation}\label{Eqn:RepVar}
\scrM(\Sigma_g) =  \left\{ (A_1,B_1,\ldots, A_g,B_g) \in SU(2)^{2g} \ \big| \ \prod_i [A_i,B_i] = -I \right\} \Big/ SO(3)
\end{equation}
where $SO(3) = SU(2)/\{\pm I\}$ acts diagonally by conjugation.   There are no reducible representations (since an abelian representation couldn't have holonomy $-I$ around the puncture), and the moduli space is a smooth compact manifold of real dimension $6g-6$.   It has a natural symplectic structure, which comes from a skew form on the tangent space $H^1(ad \,E)$ given by combining the Killing form in the Lie algebra $\frak{su}_2$ with wedge product.  The representation variety is symplectomorphic to the moduli space of stable bundles of rank two and fixed odd determinant on a Riemann surface of genus $g$ \cite{NS}, which in turn carries a canonical K\"ahler form arising from its interpretation as an infinite-dimensional quotient of the space of connexions on $\Sigma_g$:  gauge theoretically,  the symplectic structure is inherited from that on adjoint-bundle-valued 1-forms given by $\langle a,b \rangle = \int_{\Sigma} a\wedge \ast b$.  
Let $\Gamma(\Sigma)$ denote the mapping class group of isotopy classes of diffeomorphisms of $\Sigma$ which preserve  $p$; let $\cI(\Gamma)$ denote the Torelli subgroup of $\Gamma(\Sigma)$ of elements which act trivially on homology.   There is a short exact sequence 
\begin{equation} \label{Eqn:ExSeqNotSplit}
1 \rightarrow \pi_1(\Sigma) \rightarrow \Gamma(\Sigma) \rightarrow \Gamma_g \rightarrow 1,
\end{equation}
where the classical mapping class group $\Gamma_g = \pi_0\Diff^+(\Sigma) = \textrm{Out}(\pi_1(\Sigma))$ acts on the variety $\Hom(\pi_1(\Sigma), SO(3))/SO(3)$, which is a symplectic orbifold.  This latter space has two connected components, corresponding to flat connexions in the trivial respectively non-trivial $SO(3)$-bundle over $\Sigma$.  The latter component  is a quotient of $\scrM(\Sigma)$ by the finite group $\bZ_2^{2g} = H^1(\Sigma;\bZ_2) = \Hom(\pi_1(\Sigma),\bZ_2)$. This acts via  $A_i \mapsto \pm A_i, B_j \mapsto \pm B_j$ in the explicit description of the moduli space given above, or by tensoring by order two elements of the Jacobian torus of degree zero line bundles if one regards $\scrM(\Sigma)$ as a moduli space of stable bundles.  It follows that the representation $\rho: \Gamma(\Sigma) \rightarrow \pi_0\Symp(\scrM(\Sigma))$ actually factors through a representation
\[
\hat{\rho}: \widehat{\Gamma}(\Sigma) \rightarrow \pi_0\Symp(\scrM(\Sigma))
\]
where $\widehat{\Gamma}(\Sigma)$ is an extension of $\Gamma_g$ by $\bZ_2^{2g}$. 

\begin{Remark} \label{Remark:GpActingOnModuliSpace}
Finite subgroups of $\Gamma(\Sigma) \cong \Gamma_{g,1}$ are cyclic, which implies that the sequence of Equation \ref{Eqn:ExSeqNotSplit} does not split \cite{FarbMargalit}.  Results of Morita \cite{Morita, Morita2} imply that  for $g\geq 2$
\[
\bZ_2 \ \hookrightarrow \ H^2(\Gamma_g; H^1(\Sigma;\bZ_2)), 
\]  the map being an isomorphism for large $g$.  It follows that there are \emph{two} distinct extensions of the mapping class group by the finite group $H^1(\Sigma;\bZ_2)$, and $\widehat{\Gamma}$ is isomorphic to the \emph{non-trivial} extension (cf. \cite[Section 9]{Earle} and \cite[Proposition 4]{Morita2}).  That is, the quotient sequence 
\begin{equation} \label{Eqn:DoesSplit}
1 \rightarrow H^1(\Sigma;\bZ_2) \rightarrow \widehat{\Gamma}(\Sigma) \rightarrow \Gamma_g \rightarrow 1
\end{equation}
does \emph{not} split; compare to Remark \ref{Remark:GpActingOnCurve}.  However, the pullback of the extension to the mapping class group $\Gamma_{g,1}$  \emph{does} split, since Morita computes that $H_1(\Gamma_{g,1}; H_1(\Sigma;\bZ)) \cong \bZ$, which implies $H^2(\Gamma_{g,1}; H^1(\Sigma; \bZ_2)) =  0.$ 
It therefore makes good sense to compare the (weak) actions of the split extension of $\Gamma_{g,1}$ by $H^1(\Sigma;\bZ_2)$ on $\scrF(\Sigma_g)$ and $\scrF(\scrM(\Sigma_g))$. 
\end{Remark}

 There is a universal rank 2 bundle $\bE \rightarrow \Sigma\times\scrM(\Sigma)$; this is not uniquely defined, but its endomorphism bundle $\End(\bE)$ is unique.  By decomposing the second Chern class $c_2(\End(\bE))$ into K\"unneth components (i.e. using slant product) one obtains a map
\begin{equation}\label{Equation:mu}
\mu: H_*(\Sigma;\bZ) \longrightarrow H^{4-*}(\scrM(\Sigma);\bZ)
\end{equation}
which we refer to as the $\mu$-map.   We recall a number of well-known facts (see \cite{Thaddeus:Notes, Newstead, Donaldson, GH} for proofs; many other references are also appropriate).

\begin{enumerate}
\item  The space $\scrM(\Sigma_g)$ is simply connected (so there is no distinction between Hamiltonian and symplectic isotopy).  
\item $\scrM(\Sigma_g)$ has  Euler characteristic zero.  It admits a perfect Morse-Bott function with critical submanifolds (i) two $S^3$-bundles over $\scrM(\Sigma_{g-1})$, one the absolute minimum and one the absolute maximum; and (ii) a torus $T^{2g-2}$ of middle index.  
\item The cohomology ring $H^*(\scrM(\Sigma);\bZ)$ is generated by the image of the $\mu$-map. The action $\widehat\Gamma(\Sigma) \rightarrow Aut(H^*(\scrM(\Sigma);\bR)$ factors through the symplectic group $Sp_{2g}(\bZ)$, in particular $H^1(\Sigma;\bZ_2)$ acts trivially. 
\item $\scrM(\Sigma)$ is a smooth Fano variety with $H^2(\scrM(\Sigma);\bZ) \cong \bZ$. The first Chern class is twice the generator. 
\end{enumerate}

\begin{Example} \label{Ex:RingIsoQH}
We noted in the Introduction that $\scrM(\Sigma_2) \cong Q_0\cap Q_1 \subset \bP^5$.   The map $\mu: H_1(\Sigma;\bZ) \rightarrow H^3(\scrM(\Sigma);\bZ)$ is an isomorphism,  equivariant for the action of the symplectic group $Sp_4(\bZ)$. 
\end{Example}

The quantum cohomology of any closed symplectic manifold splits into generalised eigenspaces for the action of quantum multiplication by the first Chern class, cf. Section \ref{SubSec:Floer} for a fuller discussion.  For $\scrM(\Sigma_2)$ one can compute this explicitly, starting from a presentation of quantum cohomology derived by Donaldson.

\begin{Lemma} \label{Lem:QHmodulispace}
There is a ring isomorphism 
\[
QH^*(\scrM(\Sigma_2)) \ \cong \ H^*(pt) \oplus H^*(\Sigma_2) \oplus H^*(pt)
\]
with summands the generalised eigenspaces for quantum cup with $c_1(\scrM(\Sigma_2))$.
\end{Lemma}

\begin{proof}
Let $h_j$ denote the generator of $H^j(\scrM(\Sigma_2))$, for $j=2,4,6$.  Let $\gamma_i$ denote classes in $H_1(\Sigma)$, and $\delta(\cdot,\cdot)$ the intersection pairing on $\Sigma$.  Starting from the classical facts that
\begin{itemize}
\item there are four lines through the generic point of $\scrM(\Sigma_2)$;
\item given two generic lines $l_1,l_2$, there are  two other lines which meet both $l_i$;
\end{itemize}
Donaldson \cite{Donaldson} proved that $h_2\ast h_2 = 4h_4+1$ and $h_2\ast h_4 = h_6 + 2h_2$, whilst  $h_j \ast \mu(\gamma_i)=0$ for degree reasons.  Write $h=h_2$, to simplify notation;  one then obtains that $QH^*(\scrM(\Sigma_2))$ is the ring with generators and relations
\[
 \langle h, \mu(\gamma_i) \ | \ h^4=16h^2, \, h\mu(\gamma_i)=0, \, \mu(\gamma_i)\mu(\gamma_j) = \delta(\gamma_i,\gamma_j)(h^3/4-4h)\rangle
\]
Since $h=c_1/2$, we are interested in the generalised eigenspaces for $\ast h$.  Now take the three summands to be respectively generated  by 
\[
\langle h^2/16\rangle, \quad \langle 1-h^2/16,  \ \mu(H_1(\Sigma)), \ h^3-16h\rangle, \quad \langle (h^3+4h^2)/128\rangle.
\]   and observe that the summands are mutually orthogonal and closed under quantum product.
\end{proof}

\subsection{Monodromy and degenerations\label{Subsec:Monodromy}}

There are several easy similarities between Floer cohomology computations on $\Sigma_2$ and on $\scrM(\Sigma_2)$ which helped to motivate  Theorem \ref{Thm:Embed}.
 Let $\gamma\subset \Sigma$ denote a homologically essential simple closed curve.   Taking the subspace of representations trivial along $\gamma$ defines a co-isotropic submanifold $V_{\gamma} \subset \scrM(\Sigma)$ which is an $SU(2)$-bundle over $\scrM(\Sigma_{g-1})$, where the lower genus surface is essentially the normalisation of the nodal surface $\Sigma/\{\gamma\}$. When $g=2$, $\scrM(\Sigma_{g-1})$ is a single point, and $V_{\gamma} \cong SU(2)$ is a Lagrangian 3-sphere.

\begin{Lemma}\label{Lem:HolonomySpheresMeetRight}
Fix $g=2$.  Let $\gamma$ and $\gamma'$ be homologically independent simple closed curves in $\Sigma_2$. 
\begin{itemize}
\item If $\gamma \cap \gamma' = \emptyset$ then $V_{\gamma} \cap V_{\gamma'} = \emptyset$.  
\item If $\gamma \pitchfork \gamma' = \{pt\}$ then $V_{\gamma} \pitchfork V_{\gamma'} = \{pt\}$.
\end{itemize}
\end{Lemma}

\begin{proof}
In the first case, a representation in $V_{\gamma} \cap V_{\gamma'}$ defines a flat connexion on the sphere given by cutting $\Sigma_2$ open along $\gamma \cup \gamma'$, which however has non-trivial holonomy at the puncture.  No such representation can exist.  Similarly, if $\gamma \pitchfork \gamma' = \{pt\}$, any element of $V_{\gamma} \cap V_{\gamma'}$ defines a flat connexion on the torus given by cutting along $\gamma \cup \gamma'$, with non-trivial monodromy at the puncture. Such a connexion is defined by a pair of matrices $(A,B) \in SU(2)^2$ with $[A,B]=-I$, and it is a straightforward exercise to see that there is a unique such pair up to simultaneous conjugation.  We must check that the corresponding intersection point of $V_{\gamma}$ and $V_{\gamma'}$ is transverse, a result  easily extracted from either of  \cite{Donaldson, Seidel:Notes}.
 The map $f:(A_2,B_2) \mapsto [A_2,B_2]$ from $SU(2)^2 \rightarrow SU(2)$ has $-I$ as a regular value, and there is an embedding
\[
\eta: f^{-1}(-I) \times B_{I}(\delta) \hookrightarrow SU(2) \times SU(2)
\] 
for which $f\circ \eta(x,u) = -u$, with $B_{I}(\delta)$ a small open neighbourhood of $I \in SU(2)$.  We can therefore embed an open neighbourhood of the zero-section $V_{\gamma} = SU(2) \subset T^*SU(2) \cong \frak{su}_2 \times SU(2) \hookrightarrow \scrM(\Sigma)$ by a map
\[
g: (\frak{h}, B_2) \ \mapsto \ (-e^{-\frak{h}}, B_2, \eta(A_2, B_2, [e^{-\frak{h}},B_2]^{-1}).
\]
This shows that locally near $V_{\gamma}$, the fibre direction $\frak{su}_2$ is parametrised by the holonomy of the transverse curve $\gamma'$, which implies that $V_{\gamma}$ and $V_{\gamma'}$ meet transversely. 
 \end{proof}

\begin{Remark} \label{Rem:TwoHolChoices}
An essential simple closed curve $\gamma \subset \Sigma_2$ defines two Lagrangian 3-spheres, namely $hol_{\gamma}^{-1}(\pm I)$. Call these $V_{\gamma}$ and $V_{\bar\gamma}$. Since these spheres are disjoint, one obviously has
\begin{equation} \label{Eqn:TwoHolChoices}
HF(V_{\gamma},V_{\gamma}) \cong H^*(S^3) \cong HF(V_{\bar\gamma},V_{\bar\gamma}); \quad HF(V_{\gamma}, V_{\bar\gamma}) = 0.
\end{equation}
The spheres $V_{\gamma}$ and $V_{\bar\gamma}$ lie in the same homology class and the same orbit under the action of $H^1(\Sigma;\bZ_2)$. This should be compared with Remark \ref{Remark:2objects-1}.
\end{Remark}

The action on $\scrM(\Sigma)$ of a Dehn twist on $\Sigma$ has been studied by Callahan and Seidel (both unpublished, now discussed in detail in \cite{WW:triangle}). 
Let $\sigma \subset \Sigma$ denote a nullhomologous but not nullhomotopic simple closed curve, dividing $\Sigma$ into subsurfaces of genus $1$ and $g-1$ say (see Figure \ref{Fig:genus2pencil}).
There is a unique conjugation-invariant function
\[
SU_2 \longrightarrow [0,1] \quad \textrm{for which} \quad  log\left( \begin{array}{cc} e^{i\pi t} & 0 \\ 0 & e^{-i\pi t}  \end{array}\right) \, \mapsto \, t, \quad 0\leq t \leq 1.
\]
This induces a map $f: \scrM(\Sigma) \rightarrow [0,1]$, by taking the conjugacy class of the holonomy around  $\sigma$.  A theorem of Goldman \cite{Goldman} implies that, on the open subset $f^{-1}(0,1)$ where $f$ is smooth, it is the Hamiltonian function of a circle action. Explicitly, if $\sigma$ corresponds to the curve defining the matrix $A_g$ in the notation above, and if $A_g \neq \pm I$, there is a unique homomorphism 
\[
\phi: U(1) \rightarrow SU_2 \quad \textrm{with} \quad A_g \in \phi(\{ \Im(z) > 0 \} )
\] and the circle acts via
\[
\lambda \cdot (A_i, B_i) \ = \ (A_1, B_1, \ldots, A_g, \phi(\lambda)B_g).
\]  
Let $V_{\sigma}=f^{-1}(\epsilon)$ for a small $\epsilon > 0$; $V\subset \scrM(\Sigma)$ is a separating hypersurface which is preserved and acted upon freely by this  $S^1$-action.  This manifests it as a fibred co-isotropic submanifold, with base the moduli space corresponding to  the disconnected curve given by normalising  the nodal surface in which $\sigma$ is collapsed to a point.

\begin{Proposition}[Seidel, Callahan] \label{Prop:SeidelCallahan}
Let $t_{c}$ denote the positive Dehn twist along a curve $c\subset \Sigma$ and let $\tau_c$ denote the induced action of the twist on $\scrM(\Sigma)$.  
\begin{itemize}
\item $\tau_{\gamma}$ is a rank 3 fibred positive Dehn twist in the $S^3$-fibred co-isotropic $V_{\bar\gamma}$.
\item $\tau_{\sigma}$  is a rank 1 fibred positive Dehn twist in the $S^1$-fibred co-isotropic  $V_{\sigma}$.
\end{itemize}
\end{Proposition}

Both results are proved by considering the monodromy of suitable Morse-Bott degenerations, cf. related discussions in \cite{Perutz, Seidel:Notes, WW:triangle} (the last reference contains proofs of both cases).   
  If there is \emph{any} equivalence of categories $D^{\pi}\scrF(\Sigma_2) \simeq D^{\pi}\scrF(\scrM(\Sigma_2);0)$ which respects the actions of the mapping class group, then necessarily $[\gamma] \mapsto [V_{\bar\gamma}]$, by combining Proposition \ref{Prop:SeidelCallahan} and Corollary \ref{Cor:TwistDeterminesSphere} (which says that Lagrangian spheres are essentially determined by their associated twist functors).  Unfortunately, we do not know of any (smooth and embedded) Lagrangian correspondence $\Gamma \subset \Sigma_2 \times \scrM(\Sigma_2)$ which might define such a functor.  The upshot of the following sections will be to show that a certain generalised Lagrangian correspondence
\[
\Sigma_2 \rightsquigarrow Z \rightsquigarrow  \scrM(\Sigma_2) 
\]
does define an equivalence on suitable subcategories, where $Z$ denotes the relative quadric $Bl_{\scrM(\Sigma)} (\bP^5)$.  Even here, two subtleties arise: the first correspondence (from the curve to the relative quadric) would be singular, and not amenable to quilt theory (so we will avoid it and use more algebraic arguments); and the second correspondence will not itself be fully faithful, but will have an idempotent summand which induces the desired equivalence.

\subsection{Fibred 3-manifolds\label{Subsec:3manifolds}}

Recall that an $SO(3)$-bundle on a closed oriented 3-manifold $Y$ is determined up to isomorphism by its second Stiefel-Whitney class $w_2(E) \in H^2(Y;\bZ_2)$. 
Let $f : Y\rightarrow S^1$ be a 3-manifold which fibres smoothly over the circle.  The fibre $f^{-1}(pt)$ is necessarily homologically essential and primitive, and we fix an $SO(3)$-bundle $E$ over $Y$ for which $\langle w_2(E), [f^{-1}(pt)] \rangle \neq 0$.   Indeed, we take $E$ to be the mapping torus of a lift of the monodromy of $Y$ to the unique non-trivial $SO(3)$-bundle over a preferred fibre $f^{-1}(pt)$, with  $w_2(E)$ Poincar\'e dual to a section of $f$. The Chern-Simons functional
\[
a \mapsto \int_Y \, a \wedge da + a\wedge a \wedge a
\]
on the affine space $\Omega^1(Y)$ of connexions $A = A_0 + a$ on $E$ has critical points the flat connexions, which give rise to irreducible representations of $\pi_1(Y)$ into $SO(3)$ whose restriction to the fibre $\Sigma = f^{-1}(pt)$ defines a flat connexion in the original bundle $E|_{\Sigma}$.  In other words, the relevant representations of $\pi_1(Y)$ are exactly those arising from fixed-points of the natural action of the monodromy $\phi$ of $f$ (lifted to $E$) on $\scrM(\Sigma)$.  All of these representations are non-abelian, since they restrict to non-abelian representations of $\pi_1(\Sigma)$.  By  counting solutions to the anti-self-dual Yang-Mills equations on $Y\times \bR$, one can define a Morse-Floer theory for the Chern-Simons functional and hence an \emph{instanton Floer homology} group $HF_{inst}(Y; E)$. The classical theorem of Dostoglou and Salamon \cite{DS} asserts  
\begin{equation} \label{Eqn:DostSal}
HF_{inst}(Y;E) \ \cong \ HF(\hat{\rho}(\phi)).
\end{equation}
Since the fixed points of mapping classes on $\scrM(\Sigma)$ define representations of $\pi_1(Y)$, the latter can be studied via symplectic Floer homology; this will be the approach taken in the proof of Corollary \ref{Cor:RepnGrowth} later. 

\begin{Remark}
If $Y_h \rightarrow S^1$ has monodromy $h$ with $im(1-h_*) = H_1(\Sigma;\bZ)$ for the induced action on  $H_1(\Sigma;\bZ)$, the mapping torus is a homology $S^1 \times S^2$, and there is a unique $SO(3)$-bundle over $Y_h$. In particular, the left hand side of Equation (\ref{Eqn:DostSal}) is independent of the lift of $h$ to the principal bundle $E \rightarrow \Sigma$.  It follows that the Floer homologies on the representation variety of all lifts of $h$ to $\widehat{\Gamma}(\Sigma)$ are actually isomorphic, which doesn't seem obvious directly.  \end{Remark}

 Finally, one can consider the set of non-abelian $SO(3)$-representations of $\pi_1(Y)$ and its behaviour under coverings, analogously to recent interest in the growth of Betti numbers or Heegaard genus of covers.  Long and Reid \cite{LongReid} proved that any hyperbolic 3-manifold has finite covers with arbitrarily many representations into the dihedral group $A_5$ and hence non-abelian $SO(3)$-representations. However, these covers are not cyclic.  Ian Agol pointed out to the author that if a 3-manifold $Y$ has a non-trivial JSJ-decomposition -- for instance, the mapping torus of a reducible surface diffeomorphism --  some finite cover has fundamental group mapping onto a rank 2 free group, hence there are covers with infinitely many non-abelian $SO(3)$-representations. However, if $Y$ is actually hyperbolic, properties of $SO(3)$-representations under cyclic covers seem mysterious, and the second statement in Corollary \ref{Cor:RepnGrowth} seems to be new.

\subsection{Faithfulness}\label{Sec:Faithful}

For a symplectic manifold $X$  consider the conjugation-invariant \emph{Floer-theoretic entropy} of  a mapping class $\phi \in \Symp(X)/\Ham(X)$ 
\[ 
h_{Floer}(\phi) \ = \ \limsup \frac{1}{n} \log \rk_{\bK} HF(\phi^n).
\]
As mentioned in the Introduction, for area-preserving diffeomorphisms of a surface $\Sigma$ the Floer-theoretic entropy (of the action on $\Sigma$ itself) co-incides with the topological entropy which in turn co-incides with the periodic entropy, or rather their minimal values on the isotopy class \cite{Cotton-Clay,FLP, Felshtyn}. Thus 
\begin{equation} \label{Eqn:PosEntropy}
h_{Floer}(\phi)>0 \, \Leftrightarrow \, \phi \ \textrm{has a pseudo-Anosov component.}
\end{equation}
This follows from the classical fact that pseudo-Anosov maps have exponentially growing numbers of periodic points realising exponentially many different Nielsen classes;  Floer-theoretically one can actually make much more precise statements \cite{Cotton-Clay}.  Recall also there is a natural representation
\[
\hat\rho: \widehat{\Gamma}(\Sigma) \longrightarrow \pi_0\Symp(\scrM(\Sigma))
\]
where $\Gamma(\Sigma) \rightarrow \Gamma_g$ is a $\bZ_2^{2g}$-extension of the mapping class group $\pi_0\Diff^+(\Sigma_g)$. Mildly abusing notation, we will say $\hat\phi \in \widehat{\Gamma}(\Sigma)$ has a pseudo-Anosov component  if that is true of its image $\phi \in \Gamma_g$.  Our faithfulness criterion Lemma \ref{Lem:EntropyCriterion} relies on a classical observation of Thurston:

\begin{Lemma}[Thurston] \label{Lem:BuildPseudoAnosov}
Let $\Sigma_g$ be a closed surface and $\sigma \subset \Sigma_g$ a separating simple closed curve. There is another separating simple closed curve $\sigma' \subset \Sigma_g$ for which the products $\tau_{\sigma}^k \tau_{\sigma'}^{-k}$ are pseudo-Anosov mapping classes for all $k\geq 1$.
\end{Lemma}

In fact, one chooses $\sigma'$ so that $\sigma \cup \sigma'$ fill $\Sigma_g$ in the sense that their complement is a union of disks. Then any word in positive twists along $\sigma$ and negative twists along $\sigma'$ (with both twists appearing)  yields a pseudo-Anosov diffeomorphism; a proof is given in  \cite[Expos\'e 13, Th\'eor\`eme III.3]{FLP}.

\begin{Lemma} \label{Lem:EntropyCriterion}
Let $g(\Sigma)\geq 2$.  Suppose that for any  $\hat\phi \in \widehat\Gamma(\Sigma)$  with a  pseudo-Anosov component, $\hat\rho(\hat\phi)$  has strictly positive Floer-theoretic entropy.   Then $\hat{\rho}$ is faithful.
\end{Lemma}

\begin{proof}
Suppose a mapping class $\hat\phi \neq \id$ lies in the kernel of $\hat\rho$. We first claim that $\hat\phi$ cannot be in the finite subgroup $H^1(\Sigma;\bZ_2)$ which is the kernel of $\widehat{\Gamma}(\Sigma) \rightarrow \Gamma_g$.   Indeed, a theorem of Narasimhan and Ramanan \cite{NR} implies that for $\iota \in H^1(\Sigma;\bZ_2) \backslash \{\id\}$, 
\[
\textrm{Fix}(\iota) \ =  \ \textrm{Prym}(\iota)\  \cong \ T^{2g-2} 
\]
is an abelian variety of dimension $g-1$, isomorphic to the Prym variety of the double cover of $\Sigma$ defined by $\iota$.  These fixed points form a smooth Morse-Bott manifold, which arises as a clean intersection between the graph of $\iota$ and the diagonal.  Pozniak's ``local-to-global" spectral sequence for Floer cohomology \cite{Pozniak} then implies that
\[
dim_{\bK}(HF(\iota)) \  \leq \ dim_{\bK}(H^*(T^{2g-2}))
\]
which is strictly smaller than $dim_{\bK}(QH^*(\scrM(\Sigma))$, cf. Section \ref{Sec:Topology}. It follows that $HF(\iota) \not\cong HF(\id)$, so $\iota \not\in \ker(\hat\rho)$.   Accordingly, if $\hat\phi$ does lie in this kernel, it has non-trivial image $\phi$ in the classical mapping class group $\Gamma_g$.  The hypothesis implies that $\phi$ has no pseudo-Anosov component;  Thurston's classification of surface diffeomorphisms \cite{FLP} implies $\phi$ is reducible with all components periodic.   If any periodic component  is non-trivial, the mapping class $\hat\rho(\hat\phi)$ acts non-trivially on cohomology, via the $\mu$-map of Equation \ref{Equation:mu} and the non-trivial action on $H^*(\Sigma)$. We therefore reduce to $\phi$ being a product of powers of Dehn twists on disjoint separating simple closed curves. 

If $\sigma\subset \Sigma$ is a separating curve, by Lemma \ref{Lem:BuildPseudoAnosov} we can find another separating curve $\sigma'$ for which the elements $\tau_{\sigma}^k \tau_{\sigma'}^{-k}$ are pseudo-Anosov when $k>0$. By the hypothesis,  these therefore map to elements of positive Floer-theoretic entropy under $\hat\rho$. If a lift to $\widehat{\Gamma}(\Sigma)$ of some power of the twist $\tau_{\sigma}$ was in the kernel of $\hat\rho$, then  the Floer cohomology of iterates of $\tau_{\sigma'}^{-1}$ would grow exponentially in rank, hence so would the Floer cohomologies of iterates of $\tau_{\sigma'}$ recalling $HF(\psi^{-1}) \cong HF(\psi)^*$.  However,  it is straightforward using Proposition \ref{Prop:SeidelCallahan} to write down an explicit representative for the action of $\tau_{\sigma'}$ on $\scrM(\Sigma)$  for which the Floer chain groups grow only linearly in rank under iteration. It follows that the Dehn twist in the separating curve $\sigma$ actually maps to an element of infinite order.  The same trick shows a product of separating twists in disjoint curves  is also infinite order.  The result follows.
\end{proof}

\begin{Remark}
It may be worth emphasising that the hypotheses of the faithfulness criterion Lemma \ref{Lem:EntropyCriterion} could be derived from something much weaker than Conjecture \ref{Conj:SecondHighest};  for instance it would follow (over $\bK = \Lambda_{\bC}$) from the existence of a formal deformation of categories from $\scrF(\scrM(\Sigma_g);4(g-2))$ to $\scrF(\Sigma_g)$. Such (bulk) deformations seem likely to arise in the obvious strategies to attack Conjecture \ref{Conj:SecondHighest} as coming from flip diagrams, cf. Section \ref{Sec:BirModel}. 
\end{Remark}

\begin{Lemma} \label{Lem:Genus2Case}
If $g(\Sigma) =2$, the hypothesis of Lemma \ref{Lem:EntropyCriterion} is satisfied.
\end{Lemma}

\begin{proof}
Again take $\hat\phi \in \widehat{\Gamma}(\Sigma)$ with image $\phi$, and pick a representative  diffeomorphism of the surface $\Sigma$ for that mapping class.  This can be written as a product of Dehn twists in simple closed curves associated to matching paths in $\bC$, which means that the action of the diffeomorphism on $D^{\pi}\scrF(\Sigma) \simeq D^{\pi}\scrF(\scrM(\Sigma);0)$ is determined by the twist functors in the spheres $V_{\gamma}$.  These act compatibly with the equivalence of categories, by Addendum \ref{Thm:Action}, cf. Corollaries \ref{Cor:MappingClassCompatible-1} and \ref{Cor:Compatible}. For any $k\in \bZ$, the Floer group $HF(\phi^k) \cong HH(\scrG_{\phi^k})$ is given by the group of natural transformations between the identity functor and that induced by $\phi^k$, by Corollary \ref{Cor:HFisHH}.   In particular this Floer group is completely determined by the action of $\phi$ on the Fukaya category.  It follows that  the Floer cohomology of the mapping class on the underlying surface is a summand of the Floer cohomology on the moduli space. Equation (\ref{Eqn:PosEntropy}) now proves the Lemma.  \end{proof}

When $g=2$, the Floer-theoretic entropy of $\hat\rho(\hat\phi)$ is positive if and only if $\phi$ has a pseudo-Anosov component; the converse implication holds by the proof of Lemma \ref{Lem:EntropyCriterion}.  Combining the two previous results, we obtain a proof of Theorem \ref{Thm:NotSymplectic} from the Introduction, using an argument due to Wehrheim and Woodward.

\begin{Corollary} \label{Cor:TorelliNontrivial}
For $g\geq 2$, the representation $\widehat{\Gamma}(\Sigma_g) \rightarrow \pi_0\Symp(\scrM(\Sigma_g))$ is non-trivial on the Torelli group $\cI(\Gamma)$.
\end{Corollary}

\begin{proof}
For $g=2$, we have actually seen that the Dehn twist on the separating waist curve $\sigma\subset \Sigma_2$ induces a non-identity functor of Donaldson's quantum category $H(\scrF\scrM(\Sigma_2))$.  For higher genus, following Wehrheim and Woodward, we use induction on $g$ with $g=2$ the base case.  Suppose for contradiction that $\tau_{\sigma}$ induces the identity functor of the Donaldson category $H(\scrF(\scrM(\Sigma_g)))$ but the separating twist is non-trivial at genus $g-1$.  Viewing the co-isotropic vanishing cycles $V_{\gamma}$ as Lagrangian correspondences, there are natural isomorphisms of functors $V_{\gamma}\circ \tau_{\sigma}^g \circ V_{\gamma}^{op} \cong \tau_{\sigma}^{g-1}$ (this is the cohomology level version of Theorem \ref{Thm:QuiltsCompose}).  But if $\tau_{\sigma}^g \simeq id_{H\scrF\scrM(\Sigma_g)}$ then the left hand side gives the identity functor, contradicting the inductive hypothesis.
\end{proof}

\subsection{A fixed point theorem\label{Subsec:FixPoint}}
 Since there are periodic, fixed-point free surface diffeomorphisms, even when $g=2$ Theorem \ref{Thm:Embed} does not imply that $HF(\hat\rho(\phi)) \neq 0$ for an arbitrary mapping class $\phi$.  In fact, such a non-vanishing result does hold in any genus; the proof relies on concentrating not on the summand of the Fukaya category corresponding to the underlying curve but on a ``semi-simple" summand.  The moduli space $\scrM(\Sigma_g)$ is monotone with minimal Chern number 2, hence its quantum cohomology is mod $4$ graded.  We recall a result of Munoz \cite{Munoz}:

\begin{Theorem}[Munoz] \label{Thm:Munoz}
The eigenvalues of quantum cap-product $\ast h$ by the generator $h\in H^2(\scrM(\Sigma_g);\bZ) \cong \bZ$ on the quantum cohomology $QH^*(\scrM(\Sigma_g))$ are $\{0, \pm 4, \pm 8, \ldots, \pm 4(g-1)\}$. Moreover, with field co-efficients, the generalised eigenspaces associated to the eigenvalues $\pm 4(g-1)$ have rank 1.
\end{Theorem}

For $V\subset M$ an $S^k$-fibred coisotropic submanifold with base $B$, denote by $\tau_V$ the fibred positive Dehn twist along $V$. We will not distinguish notationally between $V$ and the associated Lagrangian correspondence $V\subset B\times M$.  
Fix a homologically essential simple closed curve $\gamma \subset \Sigma_g$, hence a Lagrangian correspondence $V_{\gamma} \subset \scrM(\Sigma_{g-1}) \times \scrM(\Sigma_g)$, cf. Section \ref{Subsec:Monodromy}.

\begin{Corollary}[Perutz] \label{Cor:P}
There is an isomorphism 
\[
HF(V_{\gamma}, V_{\gamma}) \ \cong \ QH^*(\scrM(\Sigma_{g-1})) \otimes H^*(S^3;\Lambda_{\bR})
\]
as $QH^*(\scrM(\Sigma_{g-1}))$-modules.
\end{Corollary}

This is  \cite[Theorem 7.5]{Perutz}  (and is a direct application of Perutz' Theorem \ref{Thm:Gysin} in a case where the Euler class of the sphere bundle vanishes). 

\begin{Lemma} \label{Lem:Evals}
Let $\phi \in \Symp(\scrM(\Sigma_g))$. The eigenvalues of quantum cup-product by the generator $h\in H^2(\scrM(\Sigma_g))$ on $HF(V_{\gamma}, (\phi\times\id)V_{\gamma})$ are contained in  $\{0, \pm 4, \ldots, \pm 4(g-2)\}$. 
\end{Lemma}

\begin{proof}
Let $R$ be a unital ring, $A$ a finitely generated $R$-algebra and $r\in R$. The spectrum of multiplication by $r$ on $A$ is a subset of that of multiplication by $r$ on $R$. Indeed, there is a surjective quotient $\pi:R^k \rightarrow A$, and $\ker(\pi)$ is stable under multiplication by elements of $R$. The result then follows from the well-known fact that if a transformation $T$ on a vector space $V$ has invariant subspace $W$, the minimal polynomial of the induced map on $V/W$ is a factor of the minimal polynomial of $T$. Therefore, it suffices to compute the spectrum of quantum cup-product by $h$ on $HF(V_{\gamma}, V_{\gamma})$.

Since $V_{\gamma} \subset \scrM(\Sigma_{g-1}) \times \scrM(\Sigma_{g})$ is simply-connected, it has minimal Maslov number $4$.  Lemma \ref{Lem:quantumcap} implies that under the natural map
\[
QH^*(\scrM(\Sigma_{g-1})) \otimes QH^*(\scrM(\Sigma_{g})) \, \cong \, QH^*(\scrM(\Sigma_{g-1}) \times \scrM(\Sigma_{g})) \longrightarrow HF(V_{\gamma}, V_{\gamma})
\]
the first Chern class maps to zero.  (In the language of Section \ref{SubSec:Floer}, $\frak{m}_0(V_{\gamma}) = 0$.)   Therefore, $1 \otimes c_1(\scrM(\Sigma_{g}))$ and $c_1(\scrM(\Sigma_{g-1})) \otimes 1$ have the same image up to sign, so it is enough to know the spectrum of the latter class. The action of this is determined by Corollary \ref{Cor:P}, and the result follows on combining that with Theorem \ref{Thm:Munoz}.
\end{proof}

\begin{Theorem}[Wehrheim-Woodward] \label{Thm:WWtriangle}
Let $k>1$.  Suppose that $V\subset M$ is an $S^k$-fibred co-isotropic, whose graph defines an orientable monotone  Lagrangian submanifold of $B \times M$ of minimal Maslov number $\geq 4$. Let $\phi \in \Symp(M)$. There is an exact triangle
\[
\cdots \rightarrow HF^*(\phi) \rightarrow HF^*(\tau_V \circ \phi) \rightarrow HF^*(V, (\phi\times\id_B)V) \stackrel{[1]}{\longrightarrow} HF^*(\phi) \rightarrow \cdots
\]
of (relatively graded) modules for the ring $QH^*(M;\bK)$.
\end{Theorem}

This is taken from \cite{WW:triangle}.  We point out that when $g=2$, the case  relevant for determining instanton Floer homology of genus 2 fibred 3-manifolds in Equation \ref{Equation:HFinstanton}, the generalised Dehn twist of Proposition \ref{Prop:SeidelCallahan} is just a classical Dehn twist in a Lagrangian 3-sphere, and the proof of the following Corollary simplifies accordingly.

\begin{Corollary} \label{Cor:FixPoint}
Let $\hat\phi \in \Symp(\scrM(\Sigma_g))$ represent a mapping class in the image of the natural homomorphism $\hat\rho: \widehat{\Gamma}(\Sigma_g) \rightarrow \pi_0\Symp(\scrM(\Sigma_g))$. Then $HF(\hat\phi) \neq 0$.
\end{Corollary}

\begin{proof}
Write the given mapping class $\hat\phi=\hat\rho(\phi^{\circ})$ and write $\phi^{\circ}=\prod_i t_{\gamma_i}^{\pm \varepsilon_i}$ as a product of (positive and negative) Dehn twists along non-separating simple closed curves; this is always possible.  We consider the  long exact sequences in Floer cohomology associated to the induced rank 3 fibred Dehn twists on $\scrM(\Sigma_g)$ of Proposition \ref{Prop:SeidelCallahan}:
\[
HF(\psi) \rightarrow HF(\tau_V\circ\psi) \rightarrow HF(V, (\psi \times \id)V) \stackrel{[1]}{\longrightarrow} \cdots
\]
where $\psi$ is a sub-composite of Dehn twists.  We claim by induction on the number of twists in the expression of $\phi$, equivalently $\phi^{\circ}$, that $HF(\hat\phi)$ is non-trivial.  Indeed, we make the following somewhat stronger inductive hypothesis:  the generalized eigenspace for $\ast h$ on $HF(\psi)$ for the eigenvalue $4(g-1)$ is rank $1$ whenever $\psi$ is a product of $\leq k$ (positive or negative) Dehn twists.   If $k$ is zero, then $\hat\phi = \id$ and we have a ring isomorphism $HF(\id) \cong QH^*(\scrM(\Sigma_g))$, in which case this result is exactly Munoz' theorem.

Taking generalized eigenspaces is an exact functor over any field; consider the associated exact sequence of generalised eigenspaces, for eigenvalue $4(g-1)$, for the operation given by quantum cup-product $\ast h$ in the exact sequence above.  If $g=2$ the third term in the sequence is $HF(L_{\gamma}, \psi(L_{\gamma}))$.  Quantum cohomology acts via the module structure of this group over $HF(L_{\gamma},L_{\gamma}) \cong H^*(S^3)$, which is trivial in relative (mod 4) degree $2$, from which it immediately follows that $\ast h$ is nilpotent. More generally, Lemma \ref{Lem:Evals} implies that the spectrum of $\ast h$ on $HF(V_{\gamma}, (\psi\times\id)V_{\gamma})$ is contained in $\{0,\ldots,\pm 4(g-2)\}$, so the generalized eigenspace for the eigenvalue $4(g-1)$ is trivial.  By exactness, if that eigenspace has rank 1 for $\psi$, it also has rank 1 for $\tau_V\circ\psi$.
The argument for inserting a negative Dehn twist is the same, but working with the exact triangle
\[
HF(\tau_V^{-1} \psi) \rightarrow HF(\tau_V\circ(\tau_V^{-1}\psi)) \rightarrow HF(V,(\tau_V^{-1}\psi \times \id)V) 
\]
In either case, the induction implies that for any mapping class, there is a distinguished rank one summand in $HF(\phi)$, so in particular $HF(\phi)\neq 0$.
\end{proof}

Corollary \ref{Cor:FixPoint}, via the circle of ideas of Section \ref{Subsec:3manifolds}, yields:

\begin{Corollary}
Every fibred 3-manifold admits a non-abelian $SO(3)$-representation.
\end{Corollary}

A stronger statement was proved by Kronheimer-Mrowka \cite{KM2}, who used Feehan and Leness' deep work on the relation between Seiberg-Witten and Donaldson invariants of closed 4-manifolds (note the proof given above uses the Dehn twist exact triangle, but not the existence of the Fukaya category).

Any element of the mapping class group $\Gamma_{2,1}$ can be written as a product of Dehn twists in simple closed curves lifted from arcs in $\bC$.  Addendum \ref{Thm:Action} and Corollary \ref{Cor:FixPoint} imply that over $\bK = \bC$, there is an isomorphism 
\begin{equation} \label{Eq:RingIsoGeneral}
HF(\hat\rho(\phi)) \cong \bC \oplus HF(\phi) \oplus \bC
\end{equation}
determining $HF(\hat\rho(\phi))$ for any $\phi \in \Gamma_{2,1}$, and not just for $\phi=\id$ (the special case of Lemma \ref{Lem:QHmodulispace}, which we knew previously).  The central term on the RHS can be computed, via results of \cite{Cotton-Clay}, from knowledge of the Thurston decomposition of $\phi$ or from the train-track obtained from a description of $\phi$ as a product of positive and negative Dehn twists in simple closed curves.  Equation \ref{Eqn:PosEntropy} and the general discussion of Section \ref{Subsec:3manifolds} now implies Corollary \ref{Cor:RepnGrowth}.  Viewed in terms of Corollary \ref{Cor:HFisHH},  Equation \ref{Eq:RingIsoGeneral} describes the Hochschild cohomology of general functors of the Fukaya category obtained from mapping classes on the surface, rather than just of the identity functor.  Modulo Remark \ref{Remark:FiniteGroupEntwine}, the instanton Floer homologies of other $SO(3)$-bundles on the mapping torus are computed by the Hochschild cohomologies of the other lifts of the monodromy to the split extension of $\Gamma_g$ by $H^1(\Sigma;\bZ_2)$ which acts on $\scrF(\Sigma_g)$.

\begin{Remark} \label{Rmk:PreciseConj}
Corollary \ref{Cor:FixPoint} implies that the summand $\scrF(\Sigma_g;4(g-1)) \neq \emptyset$.  One can exhibit explicit Lagrangian submanifolds in this summand using toric degeneration methods. For instance, when $g=2$, Nishinou \emph{et. al.} \cite{Nishinou} show by degenerating $\scrM(\Sigma_2)$ to the toric intersection $\{x^2=yz, \ p^2=qr\} \subset \bP^5$ that there is a Lagrangian torus with $\frak{m}_0(T^3) = 4$ and with Floer cohomology a Clifford algebra.

In general, take an idempotent summand of a  Lagrangian  $L^+ \in Tw^{\pi}\scrF(\scrM(\Sigma_{g-1}))$ lying in the top summand of its Fukaya category,  corresponding to eigenvalue $4(g-2)$, and with $HF(L^+, L^+) \cong \bK$.  Given any homologically essential simple closed curve $\gamma \subset \Sigma_{g}$ with associated Lagrangian correspondence $V_{\gamma} \subset \scrM(\Sigma_{g-1}) \times \scrM(\Sigma_{g})$, one gets an object $\Phi(V_{\gamma})(L^+) \in Tw^{\pi}\scrF(\scrM(\Sigma_{g}))$. This seems a good candidate for the image of $\gamma$ under the conjectural equivalence $D^{\pi} \scrF(\Sigma_{g}) \stackrel{?}{\simeq} D^{\pi}\scrF(\scrM(\Sigma_{g}); 4(g-2))$.
\end{Remark}

%%%%%%%%%%%%%%%%%%%%%%%%%%%%%%%%%%%%%%%%%%%
\section{The Fukaya category\label{Section:Fukaya}}

\begin{Notation} Fix a coefficient field $\bK$ which is algebraically closed and of characteristic zero; unless otherwise specified we assume $\bK=\bC$, but for much of the paper one could work with the Novikov field $\Lambda_{\bC}$  of formal series $\sum_{i\in \bZ} a_it^{q_i}$ with $a_i\in\bC$, $q_i\in\bR$ and $q_i \rightarrow \infty$.  \end{Notation}

\subsection{Floer and quantum cohomology\label{SubSec:Floer}}
Let $(M^{2n},\omega)$ be a spherically monotone closed symplectic manifold, meaning that the homomorphisms $\pi_2(M) \rightarrow \bR$ defined by the symplectic form and the first Chern class are positively proportional. Recall that a Lagrangian submanifold $L$ is monotone if  the symplectic area and Maslov index homomorphisms $\pi_2(M,L)\rightarrow \bR$ are positively proportional.  If $M$ is spherically monotone and $\pi_1(L)=0$ then $L$ is automatically monotone, with minimal Maslov index given by twice the first Chern class $2c_1(M)$. We will always assume that $L$ is orientable and equipped with a $Spin$ structure.  Floer cohomology is particularly benign in the monotone case, since energy and index of holomorphic curves are correlated.  For general background see \cite{Floer:lagrangian, Oh, BiranCornea}.  Here we collect a number of more specialised results to be used later, and fix notation.

 Monotonicity and orientability imply that $L$ bounds no non-constant holomorphic disk of Maslov index $<2$; since the virtual dimension of unparametrised holomorphic disks with boundary on $L$ in a class $\beta$ is 
\[
n+\mu(\beta)-3
\]
the cycle swept out by boundary values of such disks only contains contributions from Maslov index 2 disks, and defines an \emph{obstruction class} $\frak{m}_0(L)[L] \in H_n(L)$ which is a multiple of the fundamental class (the multiple counts how many Maslov index 2 disks pass through the generic point of $L$, weighted by their symplectic areas). 
If the Maslov number of $L$ is $>2$, $\frak{m}_0(L)=0$ and $HF(L,L)$ is well-defined as a $\bK$-vector space; moreover, one can take $\bK = \bC$ (the $Spin$ structures give orientations of moduli spaces and hence induce a signed differential in the Floer complex, so one can work in characteristic zero \cite[Chapter 9]{FO3}; moreover, there are no convergence issues since only finitely many holomorphic curves contribute to any given differential).  More generally, $HF(L,L')$ is well-defined provided $\frak{m}_0(L)=\frak{m}_0(L')$, since the square of the differential in the complex $CF(L,L')$ is given by the \emph{difference} between these two values, coming from bubbling along either of the boundary components of the strip $[0,1]\times\bR$.  
The group $HF(L,L)$ is naturally a unital ring via the holomorphic triangle product, with the unit $1_L \in HF^{ev}(L,L)$ defined by counting perturbed holomorphic half-planes with boundary on $L$.

The \emph{quantum cohomology} $QH^*(M)$ refers to the ``small quantum cohomology", namely the vector space $H^*(M;\bK)$ with grading reduced modulo $2$ and with product $\ast$ defined by the 3-point Gromov-Witten invariants counting rational curves through appropriate cycles in $M$; see \cite{McD-S} for the details of the construction.  The Floer cohomology $HF(L,L')$ is a bimodule for the quantum cohomology ring $QH^*(M;\bK)$, in particular there is a natural map $QH^*(M) \rightarrow HF^*(L,L)$, which is  a unital ring homomorphism. 
Quantum cohomology itself splits naturally as a ring
\begin{equation} \label{Eqn:QHsummands}
QH^*(M) \ = \ \bigoplus_{\lambda \in \textrm{Spec}(\ast c_1)} QH^*(M;\lambda)
\end{equation}
into the generalised eigenspaces (Jordan blocks) for the linear transformation given by  quantum product $\ast c_1(M): QH^*(M) \rightarrow QH^*(M)$. This splitting is into subrings which are mutually orthogonal for the quantum product.  For later, we quote the following result of Auroux, Kontsevich and Seidel \cite{Auroux:toric} (recall we work with $\bC$-coefficients):

\begin{Proposition}\label{Lem:quantumcap}
The map $QH^*(M)\rightarrow HF^*(L,L)$ takes $c_1(M) \mapsto \frak{m}_0(L)\cdot 1_L$.
\end{Proposition}

\begin{proof}[Sketch]
 Counting holomorphic disks with one interior marked point, constrained to lie on a closed cycle in $M\backslash L$, and one boundary marked point defines a map
\begin{equation} \label{Eqn:AurouxEvaluates}
H^*(M,L) \cong H^*_{ct}(M \backslash L) \longrightarrow HF^*(L,L).
\end{equation}
The Maslov class defines an element of $H^2(M,L)$, hence a class in $H_{2n-2}(M\backslash L)$. Fix a cycle $D\subset M\backslash L$ representing this class (in \cite{Auroux:toric} this cycle is taken to be a holomorphic anticanonical divisor disjoint from $L$, which need not exist in general). The argument of \cite{Auroux:toric} implies that, provided $L$ bounds no Maslov index $\leq 0$ disks, \eqref{Eqn:AurouxEvaluates} takes $D$ to twice the obstruction class $2\frak{m}_0(L)$.  (This argument computes the quantum cap action of $D$ on the fundamental class of $L$; the fact that $L \subset M\backslash D$ eliminates the possibility of a non-trivial contribution from constant holomorphic disks.)  On the other hand, over $\bC$ the map in  \eqref{Eqn:AurouxEvaluates} factors through the natural map $H^*(M,L) \rightarrow H^*(M)$, and the Maslov cycle maps to $2c_1(M) \in H^*(M)$. The result follows.
\end{proof}

The map $QH^*(M) \rightarrow HF^*(L,L)$ counts disks with an interior marked point, viewed as input, and a boundary marked point, which is the output.  Following Albers \cite{Albers}, see also \cite{Abouzaid}, one can reverse the roles of input and output to obtain a map $HF^*(L,L) \rightarrow QH^*(M)$.  (The domain of this map would more naturally be the Floer \emph{homology} of $L$, but we have used Poincar\'e duality to identify this with Floer cohomology; in particular, the second map is not a ring map, and not of degree zero.)

\begin{Lemma}\label{Lem:Composition}
The composite map $QH^*(M) \rightarrow HF^*(L,L) \rightarrow QH^*(M)$ is given by quantum cup-product with the fundamental class $[L]$.
\end{Lemma}

\begin{proof}[Sketch]
The composite counts pairs of disks, each with one interior marked point, and with an incidence condition at their boundary marked points on $L$.  Gluing, one obtains a disk with boundary on $L$ and two interior marked points lying on the real diameter, but no boundary marked point. The relevant one-dimensional moduli space (with modulus the distance $d$ between the interior marked points) has another boundary component, where $d \rightarrow 0$ rather than $d\rightarrow \infty$; this is  geometrically realised by a degeneration to a disk  attached to a sphere bubble which carries the two interior marked points.  Since the disk component has no boundary marked point and a unique interior marked point, for rigid configurations it will actually be constant, so the sphere passes through $L$.  This shows the composite map is chain homotopic to quantum product with $[L]$.  (Compare to \cite[Section 2.4]{BiranCornea:Uniruling} or \cite[Section 6]{Abouzaid}, which study the composite of the maps in the other order via a similar degeneration argument.)
\end{proof}

 Let $V\subset M$ be a co-isotropic submanifold which is fibred by circles $S^1 \hookrightarrow V \rightarrow B$ with reduced space $B$.  Suppose $B$ and $M$ are monotone and that viewed as a Lagrangian correspondence, $V\subset B\times M$ has minimal Maslov number $k$.  If in addition $k=2$, fix a \emph{global angular chain} $[\sigma]\subset V$, which by definition is a chain with boundary the pull-back $\pi_B^*(e_V)$ of the Euler class of the circle bundle $V\rightarrow B$. Finally, let $L_1, L_2$ be Lagrangian submanifolds of $B$, and let $\tilde{L}_i \subset M$ be the lifts of the $L_i$ via $V$, i.e. the circle bundles over the $L_i$ defined by $V|_{L_i}$; note $\tilde{L}_i \subset M$ is also Lagrangian. Then Perutz' quantum Gysin sequence \cite[Sections 1.4 \& 6.1]{Perutz}, together with Wehrheim-Woodward quilt theory,  implies:

\begin{Theorem}[Perutz] \label{Thm:Gysin}
The Floer cohomology $HF(\tilde{L}_1, \tilde{L}_2)$ is the cohomology of the mapping cone of quantum cup-product
\[
\ast(e_V + \sigma): CF(L_1,L_2) \rightarrow CF(L_1,L_2)
\]
with the sum of the Euler class of $V\rightarrow B$ and a correction term $\sigma \cdot 1 \in QH^*(B)$, where $\sigma$ is the algebraic count of the number of Maslov index 2 disks meeting both a generic point of $V$ and the global angular chain [$\sigma$]. 
\end{Theorem}

Under the same hypotheses, the Floer cohomology $HF(V,V)$ (computed in $B\times M$) is the (chain-level) mapping cone on quantum product by the same element on $QH^*(B)$.  If the co-isotropic $V$ has minimal Maslov number $>2$ there is no correction term, and the relevant isomorphisms hold with the map being quantum product by the Euler class.  The exact triangle
\begin{equation} \label{Eqn:LagQuantumGysin}
\cdots \longrightarrow HF(L_1,L_2) \xrightarrow{\ast(e_V+\sigma)}HF(L_1,L_2) \longrightarrow HF(\tilde{L}_1, \tilde{L}_2) \stackrel{[1]}{\longrightarrow} \cdots
\end{equation}
is an exact triangle of $QH^*(B)$-modules, and in particular one obtains analogous triangles for particular generalised eigenspaces in the sense of Equation \ref{Eqn:QHsummands}.

\subsection{Higher order products\label{SubSec:Fukaya}}

The monotone Fukaya category $\scrF(M)$ is a $\bZ_2$-graded $A_{\infty}$-category, linear over $\bC$.  By definition, $\scrF(M)$ has:
\begin{itemize}
\item objects being monotone oriented Lagrangian submanifolds equipped with $Spin$ structures  and additional perturbation data\footnote{One can also equip the Lagrangians with flat unitary line bundles; we will not need that refinement in this paper, though see Remark \ref{Remark:2objects-1}.};
\item morphisms given by Floer cochain complexes $CF(L,L')$; and
\item higher order composition operations from counting pseudoholomorphic polygons.
\end{itemize}
The higher order operations of the $A_{\infty}$-structure comprise a collection of maps $\mu_{\scrF}^d$ of degree $d$ (mod 2), for $d\geq 1$, with $\mu_{\scrF}^1$ being the differential and $\mu_{\scrF}^2$ the holomorphic triangle product mentioned previously:
\[
\mu_{\sF}^d: CF(L_{d-1},L_d) \otimes \cdots \otimes CF(L_0,L_1) \rightarrow CF(L_0,L_d)[2-d]
\]
The $\{\mu^d_{\scrF}\}$ have matrix coefficients which are defined by counting holomorphic disks with $(d+1)$-boundary punctures, whose arcs map to the Lagrangian submanifolds $(L_0,\ldots, L_d)$ in cyclic order.   The construction of the operations $\mu^d_{\scrF}$ is rather involved, and we  defer to \cite{FCPLT} for details.  We should emphasise that  \emph{monotonicity enters crucially} in ensuring that the only disk and sphere bubbling in the zero- or one-dimensional moduli spaces of solutions we wish to count comes from bubbling of Maslov index 2 disks on the Lagrangian boundary conditions,  because the other possible bubbles sweep out subsets of $M$ or the $L_j$ of sufficiently high codimension to not interact with holomorphic polygons of index at most 1.   Monotonicity, together with Lemma \ref{Lem:quantumcap}, further ensures that the Maslov index 2 disks are multiples of chain-level representatives for the units $1_{L_i}$ in the complexes $CF(L_i, L_i)$, which in turn means that their contributions to the boundary strata of one-dimensional solution spaces cancel algebraically.

\begin{Remark} \label{Rem:Pearls}
The simplest construction of $\scrF(M)$ uses Hamiltonian perturbations to to guarantee that the Lagrangian submanifolds $\{L_j\}$ which are boundary conditions for a  given higher-order product are pairwise transverse. In the sequel, we will encounter Lagrangians fibred over arcs in $\bC$ (as matching cycles in Lefschetz fibrations), and the Riemann mapping theorem will provide a useful constraint on the $\mu^d_{\scrF}$. To take advantage of this systematically would preclude using Hamiltonian perturbations.   There is another approach to the Fukaya category, based on counting \emph{pearls} -- configurations of holomorphic polygons and gradient flow trees -- which is the open string counterpart of work of Biran, Cornea and Lalonde.  For closed curves in monotone $M$ the theory is developed in \cite{BiranCornea}, and for open curves the theory is described in \cite[Section 7]{Seidel:HMSgenus2} and \cite[Section 4]{Sheridan} (in the balanced, or monotone, and exact settings respectively). The fact that a pearly definition gives rise to a quasi-isomorphic category to that obtained through Hamiltonian perturbations is addressed by a ``mixed category" trick in \cite[Section 4.8]{Sheridan}, cf. \cite[Section 10a]{FCPLT}. 

 In general, pearly moduli spaces cannot be made transverse, and a complete definition of the Fukaya category using only pearls relies on virtual perturbation techniques.  In this paper, the $A_{\infty}$-structures are essentially always constrained by their formal algebraic properties, rather than by explicit computations (or the computations involve rather benign non-transversal situations, for instance a pair of Lagrangians meeting cleanly in a Morse-Bott intersection). However, a pearly model would be more natural in Section \ref{Section:FiniteDeterminacy}, even if not strictly required. 
  \end{Remark}

Suppose $\lambda \in \bK$ is not an eigenvalue of the quantum cup-product
\[
\ast c_1(M): QH^*(M;\bK) \rightarrow QH^*(M;\bK).
\]
Then $c_1(M) - \lambda 1_M$ is invertible in $QH^*(M)$; on the other hand, if $L \in \scrF(M;\lambda)$, then Lemma \ref{Lem:quantumcap} implies that $c_1(M) - \lambda 1_M \mapsto 0 \in HF^*(L,L)$, which implies that $HF^*(L,L) = 0$.  
Lemma \ref{Lem:quantumcap} therefore implies that a monotone symplectic manifold $M$ gives rise to a collection of mutually orthogonal categories $\scrF(M;\lambda)$,  which are non-trivial only for $\lambda$ an eigenvalue of $\ast c_1(M): QH^*(M) \rightarrow QH^*(M)$.  The summand $\scrF(M;\lambda)$ has as objects the Lagrangian submanifolds $L$ for which $\frak{m}_0(L) = \lambda$, i.e. for which  the natural map of $\bZ_2$-graded unital rings
\begin{equation} \label{Eq:quantumcap}
QH^*(M) \rightarrow HF(L,L) \qquad \textrm{takes} \quad c_1(M) \mapsto \frak{m}_0(L)\cdot (\textrm{unit}).
\end{equation}
We will refer to the category $\scrF(M;0)$ as the \emph{nilpotent summand} of the Fukaya category.  We next recall the definition of the \emph{Hochschild cohomology} of an $A_{\infty}$-category, which is defined by a bar complex  $CC^*(\scrA)$ as follows.  A degree $r$ cochain is a sequence $(h^d)_{d\geq 0}$ of collections of linear maps
\[
h^d_{(X_1,\ldots,X_{d+1})}: \bigotimes_{i=d}^1 hom_{\scrA}(X_i,X_{i+1}) \rightarrow hom_{\scrA}(X_1,X_{d+1})[r-d]
\]
for each $(X_1,\ldots,X_{d+1})\in \Ob(\scrA)^{d+1}$.   The differential is defined by the sum over  concatenations
\begin{equation} \label{Eqn:Hochschild}
\begin{aligned}
(\partial h)^d & (a_d,\ldots, a_1) = \\
& \sum_{i+j<d+1} (-1)^{(r+1)\maltese_i} \mu_{\scrA}^{d+1-j}(a_d,\ldots,a_{i+j+1},h^j(a_{i+j},\ldots,a_{i+1}),a_{i},\ldots,a_1) \\
+ &  \sum_{i+j\leq d+1} (-1)^{\maltese_{i} +r +1}  h^{d+1-j} (a_d,\ldots,a_{i+j+1},\mu_{\scrA}^j(a_{i+j},\ldots,a_{i+1}),a_{i},\ldots,a_1).
\end{aligned}
\end{equation}
It is a basic fact that  $HH^*(\scrA) = H(hom_{fun(\scrA,\scrA)}(\id,\id))$ computes the morphisms in the $A_{\infty}$-category of endofunctors of $\scrA$ from the identity functor to itself; moreover, the Hochschild cohomology of an $A_{\infty}$-category over a field is invariant under passing to a split-closed triangulated envelope \cite[Theorem 4.12]{BM}. The maps of Equation \ref{Eq:quantumcap} are the lowest order pieces of a natural ``open-closed string map" 
\[
QH^*(M) \longrightarrow HH^*(\scrF(M))
\]
from quantum cohomology to Hochschild cohomology, given by counting holomorphic polygons with one interior marked point,  constrained to a cycle in $M$, and a collection of boundary punctures, one of which is outgoing and the rest incoming. Equation \ref{Eq:quantumcap} corresponds to the case in which there are no boundary inputs.  By construction, the open-closed string map is compatible with the decomposition of the Fukaya category into orthogonal summands and with the splitting of Equation \ref{Eqn:QHsummands}, giving maps
\[
QH^*(M;\lambda) \longrightarrow HH^*(\scrF(M; \lambda))
\]
for each $\lambda \in \textrm{Spec}(\ast c_1(M))$.

\begin{Remark} \label{Rem:ScaleSymplecticForm}
The monotone Fukaya category $\scrF(M)$ is invariant, up to quasi-isomorphism, under rescaling the symplectic form. Some care must be taken, however, when considering products of monotone manifolds, since the factors cannot be scaled independently preserving monotonicity. 
\end{Remark}

\subsection{Twisting and generation}

Let $\scrA$ be an $A_{\infty}$-category.   There is a two-stage purely algebraic operation which formally enlarges $\scrA$ to yield a more computable object:  first, passing to twisted complexes to give $Tw \, \scrA$; second, idempotent-completing to give $\Pi(Tw\,\scrA) = Tw^{\pi}\scrA$.  Twisted complexes themselves form the objects of a non-unital $A_{\infty}$-category $Tw(\scrA)$, which has the property that all morphisms can be completed with cones to sit in exact triangles.   Idempotent completion, or split-closure, includes objects quasi-representing all cohomological idempotents, and is discussed further in Section \ref{Section:AlgebraicInterlude}.  We write $nu$-$fun(\scrA,\scrB)$ for the $A_{\infty}$-category of non-unital functors from $\scrA$ to $\scrB$; $mod$-$\scrA$ for $nu$-$fun(\scrA^{opp},Ch)$, where $Ch$ is the $dg$-category  (viewed as an $A_{\infty}$-category with vanishing higher differentials) of chain complexes of $\bK$-vector spaces.  Given $Y \in$ Ob$\,\scrA$ and an $\scrA$-module $\scrM$, we define the \emph{algebraic twist} $\scrT_Y\scrM$ as the module
\[
\scrT_Y\scrM(X) = \scrM(Y)\otimes hom_{\scrA}(X,Y)[1] \oplus \scrM(X)
\]
(with operations we shall not write out here).    The twist is the cone over the canonical evaluation morphism
\[
\scrM(Y) \otimes \scrY \rightarrow \scrM
\]
where $\scrY$ denotes the Yoneda image of $Y$.  For two objects $Y_0, Y_1 \in$ Ob$\, \scrA$, the essential feature of the twist is that it gives rise to a canonical exact triangle in $H(\scrA)$
\[
\cdots \rightarrow Hom_{H(\scrA)}(Y_0,Y_1) \otimes Y_0 \rightarrow Y_1 \rightarrow T_{Y_0}(Y_1) \stackrel{[1]}{\rightarrow} \cdots
\]
(where $T_{Y_0}(Y_1)$ is any object whose Yoneda image is $\scrT_{Y_0}(\scrY_1)$).  Note that if $\scrA$ decomposes into a collection of orthogonal subcategories, and if $Y_0$ lives purely in one of these, the algebraic twist by definition acts trivially on all the other summands.

 The twist functor $\scrT_L \in nu$-$fun(\scrF(M),\scrF(M))$ plays an essential role when $L$ is \emph{spherical}, meaning that $Hom_{H(\scrF(M))}(L,L) \cong H^*(S^n)$.  Suppose $L\subset M$ is a Lagrangian sphere;  the \emph{geometric twist} $\tau_L$ is the autoequivalence of $\scrF(M)$ defined by the positive Dehn twist in $L$.  On the other hand, since $M$ is spherically monotone, such a sphere gives rise to a well-defined (though not \emph{a priori} non-zero) object of the Fukaya category by the classical work of Oh \cite{Oh}, hence an algebraic twist functor in the sense described above. 

\begin{Proposition}[Seidel] \label{Prop:twists}
If $L\subset M$ is a Lagrangian sphere, equipped with the non-trivial \emph{Spin} structure if $L \cong S^1$, then the geometric twist and the algebraic twist are quasi-isomorphic in $nu$-$fun(\scrF(M),\scrF(M))$.
\end{Proposition}

This is \cite[Corollary 17.17]{FCPLT} for exact symplectic manifolds; the argument carries over \emph{mutatis mutandis} to the monotone case (see \cite{WW:triangle} for a detailed account, but note most of their work is needed only for the strictly more difficult case of fibred Dehn twists).  Seidel also proves that if $\scrG: \scrA \rightarrow \scrB$ is an $A_{\infty}$-functor, then for $Y\in \scrA$,
\begin{equation} \label{Eqn:TwistFunctor}
\scrG \circ \scrT_Y \ \simeq \ \scrT_{\scrG(Y)}\circ \scrG \ \in nu{\--}fun(\scrA,\scrB).
\end{equation}

\begin{Corollary}\label{Cor:TwistDeterminesSphere}
A Lagrangian sphere $L \subset M$ is  determined up to quasi-isomorphism in $\scrF(M)$ by its associated twist functor $\scrT_L$ together with the natural transformation $\id \rightarrow \scrT_L[-1]$. 
\end{Corollary}
 
 \begin{proof}
 In general:  if $Y_1$ and $Y_2 \in \scrA$ are spherical and $\scrT_{Y_1}$ and $\scrT_{Y_2}$ are quasi-isomorphic as objects of $nu$-$fun(mod$-$\scrA,mod$-$\scrA)$, by an isomorphism which entwines the natural transformations $\id \rightarrow \scrT_{Y_i}[-1]$, then $Y_1$ and $Y_2$ are quasi-isomorphic objects of $\scrA$.   In the geometric situation at hand, the natural transformation of functors $\id \rightarrow \scrT_L[-1]$ arises from counting holomorphic sections of a Lefschetz fibration over the annulus with a unique interior critical point having vanishing cycle the given Lagrangian sphere.  The cone on this natural transformation is the evaluation map, whose image lies in the subcategory of $mod$-$\scrA$ generated by the object $Y$ itself; indeed, all elements in the image are twisted complexes of the form $V\otimes Y$, for graded vector spaces $V$.  The only spherical such object is $Y$ itself, for reasons of rank, hence $\id \rightarrow \scrT_Y[-1]$ determines $Y$.  \end{proof}

We will later need a split-generation criterion for finite collections of Lagrangian spheres, which is a minor variant on an argument due to Seidel \cite{Seidel:HMSgenus2}.  Let $M$ be a closed symplectic manifold and $\{V_1,\ldots, V_k\}$ a collection of Lagrangian spheres in $M$ for which there is a \emph{positive relation}, i.e. some word $w$ in the positive Dehn twists $\tau_{V_j}$ is equal to the identity in $\Symp(M)/\textrm{Ham}(M)$; if $M$ is simply-connected, this quotient is the mapping class group $\pi_0\Symp(M)$.  The word $w$, strictly speaking together with a choice of Hamiltonian isotopy from that product of twists to the identity,  defines a Lefschetz fibration $\scrW \rightarrow S^2$ with fibre $M$, where $w$ encodes the monodromy homomorphism of the fibration.  We suppose that the Lagrangian spheres all lie in the $\lambda$-summand $\scrF(M;\lambda)$.  Fix a homotopy class of sections $\beta$ of $\scrW \rightarrow S^2$; the moduli space of $J$-holomorphic sections in this homotopy class has complex virtual dimension
\[
\langle c_1(T^{vt}(\scrW), \beta \rangle + \dim_{\bC}(M).
\]
Suppose now the moduli spaces of sections define  pseudocycles, or more generally carry a virtual class.  Evaluation at a point therefore defines, by Poincar\'e duality, a cycle in $H^{ev}(M;\bC)$ depending on $\beta$, and arranging these for different homology classes of section defines an element of $H^{ev}(M;\Lambda_{\bR})$, or of $H^{ev}(M;\bC)$ in the fibre-monotone case, which we call the cycle class $\scrC(w)$ of the word $w$.

\begin{Proposition} \label{Prop:Split-Generate-basic}
If for some positive relation $w$ in the $\tau_{V_j}$ quantum cup-product by the cycle class $\scrC(w)$ is nilpotent as an operation on $QH^*(M;\lambda)$,  the spheres $\{V_j\}$ split-generate  $D^{\pi}\scrF(M;\lambda)$. 
\end{Proposition}

\begin{proof}
By the correspondence between algebraic and geometric Dehn twists, Proposition \ref{Prop:twists}, there are exact triangles
\[
\cdots \rightarrow  HF(V,K)\otimes V \rightarrow K \rightarrow \tau_V(K) \stackrel{[1]}{\longrightarrow} \cdots
\]
for any Lagrangian $K \subset M$ lying in the $\lambda$-summand of the category.  Concatenating the triangles for the $\tau_{V_j}$ occuring in $w$ defines a natural map
\begin{equation}\label{eqn:concatenate}
K \rightarrow \prod_{i_j\in I} \tau_{V_{i_j}} K \cong K
\end{equation}
defined by an element of $HF^{ev}(K,K)$.  Recall from  the construction of the long exact sequence in Floer cohomology \cite{Seidel:triangle}, cf. the proof of Corollary \ref{Cor:TwistDeterminesSphere}, that the natural map $\id \mapsto \scrT_L[-1]$ arises from counting sections of a Lefschetz fibration over an annulus with a single critical point.  It is well-known that one can achieve transversality for holomorphic curves without fibre components using almost complex structures which make the fibration map pseudo-holomorphic.  The gluing theorem \cite[Proposition 2.22]{Seidel:triangle} implies that the concatenation of such maps counts holomorphic sections of the Lefschetz fibration given by sewing several such annuli together.  The map in  Equation \ref{eqn:concatenate} is accordingly given by the image of  the cycle class $r(\scrC(w)) \in HF^{ev}(K,K)$ under the natural restriction map $r: H^{ev}(M) \rightarrow HF^{ev}(K,K)$.  If $\scrC(w)$ vanishes, one map in the concatenated exact triangle vanishes, and we therefore see that $K$ is a summand in an iterated cone amongst the objects $\{V_{i_j}\}$.  If the cycle class is not trivial but quantum product with the cycle class is nilpotent, we can run the same argument after taking an iterated fibre sum of the Lefschetz fibration with itself. 
\end{proof}

\begin{Proposition} \label{Prop:Split-Generate}
If  the vanishing cycles $\{V_j\}$ all lie in the  summand $\scrF(M;\lambda)$ and for some positive relation $w$ in the $\tau_{V_j}$ the cycle class $\scrC(w)$ is a multiple of $c_1(M)-\lambda\id$, then the $\{V_j\}$ split-generate $D^{\pi}\scrF(M;\lambda)$.
\end{Proposition}

\begin{proof}
The image of the cycle class under the natural map $QH^*(M) \rightarrow HF(L,L)$ is zero for any $L \in \textrm{Ob} \, \scrF(M;\lambda)$, by definition of that summand of the category, cf. Equation \ref{Eq:quantumcap}.  The proof now proceeds as before.
\end{proof}

This is particularly useful when $\lambda =0$, for the nilpotent summand.

\subsection{Functors from quilts\label{Section:Quilts}}

Geometrically, functors between Fukaya categories are obtained from counts of quilted holomorphic surfaces, using the theory developed by Mau, Wehrheim and Woodward \cite{WW1, WW2,MWW:Announce,Mau}.  They begin by defining an extended category $\scrF^{\#}(M)$ which comes with a fully faithful embedding $\scrF(M) \hookrightarrow \scrF^{\#}(M)$; objects of the extended category are generalised Lagrangian submanifolds, which comprise an integer $k\geq 1$ and a sequence of symplectic manifolds and Lagrangian correspondences
\[
(\{pt\}=M_0, M_1,\ldots, M_k=M); \qquad L_{i,i+1} \subset M_i^-\times M_{i+1}, \ 0\leq i\leq k-1
\]
where $(X,\omega)^-$ is shorthand for $(X,-\omega)$.  These form the objects of a category $H(\scrF^{\#}(M))$, in which morphisms are given by quilted Floer cohomology groups. A \emph{quilted $(d+1)$-marked disc} is a disc $D\subset \bC$ with boundary marked points $\{z_0, z_1, \ldots, z_d\}$ and with a distinguished horocycle -- the seam -- at the point $z_0$.   Any quilted Riemann surface whose boundary components are labelled by Lagrangian submanifolds and whose seams are labelled by Lagrangian correspondences determines an elliptic boundary value problem.   This problem studies a collection of holomorphic maps, one defined on each subdomain of the surface, subject to Lagrangian boundary conditions as prescribed by the labelling data along boundaries and seams. There is  an $A_{\infty}$-category $\scrF^{\#}(M)$ underlying the cohomological category $H(\scrF^{\#}(M))$, in which the higher order operations count suitable quilted disks.

The moduli space $\mathcal{M}_{d,1}$ of \emph{nodal stable} quilted $(d+1)$-marked discs  (which we shall not define) is a convex polytope homeomorphic to the multiplihedron.  The codimension one boundary faces of the multiplihedron correspond to the terms of the quadratic $A_{\infty}$-functor equation
\begin{multline*}
\sum_{i_j, k} \mu_{\sB}^k(\cF(a_n,\ldots,a_{i_1}),\cF(a_{i_1-1},\ldots, a_{i_2}), \ldots, \cF(a_{i_k-1},\ldots a_1)) \\
= \sum_d (-1)^{\maltese_{j-d}} \cF(a_n,\ldots, a_{j}, \mu_{\sA}^d (a_{j-d},\ldots, a_{j-d-1}), a_{j-d},\ldots a_1) 
\end{multline*}
for an $A_{\infty}$-functor $\cF: \sA \rightarrow \sB$.  Mau's gluing theorem \cite{Mau} shows that the counts of nodal stable quilted discs indeed reflect the combinatorial boundary structure of the $\mathcal{M}_{d,1}$ and hence the $A_{\infty}$-functor equations.  It follows that  to every monotone, oriented and spin Lagrangian correspondence $L^{\flat} \subset M^-\times N$, there is a $\bZ_2$-graded $A_{\infty}$-functor $\cF_{L^{\flat}}: \scrF^{\#}(M) \rightarrow \scrF^{\#}(N)$ defined on objects by
\[
\bL \longrightarrow  \cF_{L^{\flat}}(\bL) \quad \textrm{taking} \quad 
(L_1, L_{23}, \ldots, L_{k-1,k}) \,   \mapsto \,  (L_1,L_{23},\ldots, L_{k-1,k},L^{\flat})
\]
and on morphisms and higher products by a signed count of quilted discs. 
The upshot is the following.  

\begin{Theorem}[Mau, Wehrheim, Woodward] \label{Thm:QuiltsGiveFunctors} 
Let $M$ and $N$ be monotone symplectic manifolds.  The association $L^{\flat} \mapsto \cF_{L^{\flat}}$ defines a $\bZ_2$-graded $A_{\infty}$-functor
\[
\Phi: \scrF(M^- \times N) \ \longrightarrow \ nu\mbox{-}fun(\scrF^{\#}(M),\scrF^{\#}(N)).
\]
\end{Theorem}

We will combine this general theory with an argument from \cite{Abouzaid-Smith}, which uses properties of Yoneda embeddings to deduce that, in special cases, the Mau-Wehrheim-Woodward functor is actually fully faithful.

\begin{Corollary}\label{Cor:QuiltsQH}
In the situation of Proposition \ref{Prop:Split-Generate}, the Hochschild cohomology
\[
HH^*(D^{\pi}\scrF(M;\lambda)) \cong QH^*(M;\lambda)
\]
is isomorphic to the $\lambda$-generalised eigenspace of quantum cohomology.
\end{Corollary}

\begin{proof}
We will use quilts to resolve the $\lambda$-summand of the diagonal $\Delta \subset M \times M$, which in turn will yield information on Hochschild cohomology.   
Recall from Proposition \ref{Prop:twists} that the Dehn twist acts on the Fukaya category via an algebraic twist, hence sits in a canonical exact triangle of functors
\[
\rightarrow id \rightarrow \scrT_L \rightarrow Hom(L,\cdot)\otimes L \stackrel{[1]}{\longrightarrow}
\]
in $nu$-$fun(\scrF(M),\scrF(M))$.  This triangle is the image under the Mau-Wehrheim-Woodward functor $\Phi$ of a triangle
\begin{equation} \label{Eqn:LiftTriangle}
\rightarrow \Delta_M \rightarrow \Gamma(\tau_L) \rightarrow L\times L \stackrel{[1]}{\longrightarrow}
\end{equation}
with $\Delta_M$ the diagonal and $\Gamma(\tau_L)$ the graph of the geometric Dehn twist, cf. \cite[Theorem 7.2]{WW:triangle}.  The functor $\Phi$ is in general not fully faithful. In the setting of Proposition \ref{Prop:Split-Generate}, we let $\scrA(M)$ denote the subcategory generated by the vanishing cycles $\{V_j\}$ and let $\scrA(M\times M)$ denote the subcategory generated by product Lagrangians $V_i \times V_j$.  We define $\scrA^{\oplus} \subset \scrF^{\#}(M)$ to comprise those generalised correspondences
\[
(\{pt\}=M_0, M_1,\ldots, M_k=M); \qquad L_{i,i+1} \cong V_{i_0} \times V_{i_1} \subset M_i^-\times M_{i+1}, \ 0\leq i\leq k-1
\]
so each correspondence lies in $\scrA(M\times M)$. Lemma 7.5 of \cite{Abouzaid-Smith} shows $Tw(\scrA) \simeq Tw(\scrA^{\oplus})$, and moreover that  $\Phi$ induces a functor
\[
\Phi: \scrA(M \times M) \rightarrow nu\-- fun(Tw^{\pi}\scrA^{\oplus}(M),Tw^{\pi}\scrA^{\oplus}(M))
\]
which is a fully faithful embedding \cite[Lemma 7.8]{Abouzaid-Smith}. Note that this argument relies on the existence of chain-level units in the Fukaya category satisfying 
\[
\mu^2(e,e) = e; \quad \mu^k(e,\ldots,e) = 0 \qquad \textrm{for} \ k>2.
\] In \cite{Abouzaid-Smith} these existed for grading reasons; in general, to obtain such units one must pass from the Fukaya category to its category of modules, in the manner of \cite[Section 4]{Abouzaid}.
 In the situation at hand, $Tw^{\pi}\scrA(M) \simeq Tw^{\pi}(\scrF(M;\lambda))$, by Proposition \ref{Prop:Split-Generate}.  We now go back to the proof of that Proposition;  the argument, lifted by $\Phi$ to the category $\scrF(M\times M)$ as in Equation \ref{Eqn:LiftTriangle}, leads to a collection of exact triangles which concatenate to one of the shape
\[
\Delta_M \rightarrow \Gamma(\prod_{i_j} \tau_{V_{i_j}}) \rightarrow \langle V\times V' \rangle
\]
where the third term is built out of product Lagrangians of the shape $V_i \times \tau_{V_j}(V_k)$ and the first arrow is given by quantum product by the cycle class $\mathcal{C}(w)$.  Since $HF(\Delta_M, \Delta_M) \cong QH^*(M)$ as a $QH^*(M)$-module, the first arrow in this triangle vanishes on the idempotent summand of $\Delta_M$ corresponding to $QH^*(M;\lambda)$.  In other words, that idempotent summand $\Delta^{\lambda}$ of the diagonal is split-generated by product Lagrangians, and lies in the split-closure of the category $\scrA(M\times M)$.  At this point, we infer that $\Phi$ restricted to the extended category
\[
\scrA^{\Delta^{\lambda}}(M\times M)
\]
which comprises product Lagrangians and the relevant summand of the diagonal, is also a fully faithful embedding. Since the diagonal maps to the identity functor of $Tw^{\pi}\scrA(M)\subset Tw^{\pi}\scrF^{\#}(M)$, this implies that 
\[
HF((\Delta, \Delta);\lambda) \cong Hom_{nu-fun}(\id, \id)
\]
where the RHS is computed in endofunctors of $Tw^{\pi}\scrA(M) \simeq Tw^{\pi}\scrF(M;\lambda)$.  But then the natural transformations of the identity functor exactly compute Hochschild cohomology.
\end{proof}

For any endofunctor $\scrG: \scrA \rightarrow \scrA$ of an $A_{\infty}$-category, there is a Hochschild cohomology group $HH^*(\scrG) = H(hom_{fun(\scrA,\scrA)}(\scrG,\id)$, generalising the Hochschild cohomology of the category which arises for $\scrG = \id$.  Symplectomorphisms of a monotone symplectic manifold act naturally on the monotone Fukaya category. Write $HF(\phi)$ for the fixed-point Floer cohomology of a symplectomorphism $\phi$, so $HF(\id) \cong QH^*(M)$ by the Piunikhin-Salamon-Schwarz isomorphism. As usual, the group $HF(\phi)$ is a module for quantum cohomology, hence has an eigenspace splitting into summands indexed by $\textrm{Spec}(\ast c_1(M))$.

\begin{Corollary} \label{Cor:HFisHH}
In the situation of Proposition \ref{Prop:Split-Generate}, for any  $\phi \in \Symp(M)$ in the subgroup generated by Dehn twists in the vanishing cycles $V_j$ inducing a functor $\scrG_{\phi}$ of $\scrF(M;\lambda)$, one has
\[HF(\phi;\lambda) \cong HH^*(\scrG_{\phi}).\]
\end{Corollary}

\begin{proof}
We use the standard identification $HF(\phi) \cong HF(\Gamma_{\phi}, \Delta)$ where $\Gamma_{\phi} \subset M\times M$ denotes the graph.  The $\lambda$-summand of this graph is resolved by suitable product Lagrangians built from the vanishing cycles, as in Corollary \ref{Cor:QuiltsQH}.   The proof now proceeds as before, using fullness and faithfulness of the Mau-Wehrheim-Woodward functor in this setting.
\end{proof}

\begin{Remark}
For any subcategory $\scrA \subset \scrF(M)$, there are always natural open-closed string maps
\[
HH_*(\scrA,\scrA) \rightarrow QH^*(M) \rightarrow HH^*(\scrA,\scrA).
\]
Abouzaid \cite{Abouzaid} proves (in the exact case, and for wrapped Floer cohomology) that \emph{if the unit is in the image of the first map, then $\scrA$ split-generates the Fukaya category}.  Such split-generation results are often proved by showing that the diagonal is resolved by products, in the manner of Beilinson's classical argument for sheaves on projective space.  In our case, we know from Proposition \ref{Prop:Split-Generate} that the appropriate $\scrA$ split-generates, without knowing that the diagonal is resolved by products.  Quilts enable us to refine split-generation to this more geometric fact, which in turn has consequences for quantum cohomology. 
\end{Remark}

Suppose $L_0^{\flat} \subset M^-\times N$ and $L_1^{\flat}\subset N^-\times P$ are Lagrangian correspondences. Their  geometric composition is given by pulling $L_0^{\flat}$ back to $M\times N\times N\times P$, intersecting with the diagonal $M \times \Delta_N \times P$, and pushing forward to $N^-\times P$; in general, although the intersection can be made transverse, the push-forward is only immersed. 

\begin{Theorem}[Mau, Wehrheim, Woodward] \label{Thm:QuiltsCompose}
Given correspondences $L_0^{\flat} \subset M^-\times N$ and $L_1^{\flat}\subset N^-\times P$ with the geometric composition $L_1^{\flat} \circ L_0^{\flat} \subset M^-\times P$ smooth, embedded and transversely cut out, there is a quasi-isomorphism of functors 
\[
\cF_{L_1^{\flat}\circ L_0^{\flat}} \ \simeq \  \cF_{L_1^{\flat}} \circ \cF_{L_0^{\flat}}.
\]
\end{Theorem}

The construction of a morphism in the functor-category between $\cF_{L_1^{\flat}\circ L_0^{\flat}}$ and $\cF_{L_1^{\flat}} \circ \cF_{L_0^{\flat}}$ comes from counting ``biquilted" disks with two interior seams (two parallel horocycles). The quasi-isomorphism statement is a  difficult result which relies on a delicate strip-shrinking argument  for the associated cohomological functors, proved in detail in \cite{WW2}. A different proof of the cohomological isomorphism has recently been given by Lekili and Lipyanskiy \cite{LekiliLipyanskiy}.

\subsection{Grading and deformations\label{Sec:Indices}}
A one-parameter \emph{deformation of $A_{\infty}$-categories} is an $A_{\infty}$-category $\scrA_q$ over $\bK[[q]]$, with $\mu^d_{\scrA_q} = \mu^d_{\scrA} + O(q)$, i.e.
\begin{equation} \label{Eqn:Deformation}
\mu^d_{\scrA_q} = \mu^d_{\scrA} + q\mu^d_{\scrA_q,1} + q^2\mu^d_{\scrA_q,2} + \cdots
\end{equation}
with each $\mu^d_{\scrA_q,j}$ comprising $\bK$-linear maps 
\[
hom_{\scrA}(X_{d_1},X_d) \otimes \cdots \otimes hom_{\scrA}(X_1,X_2) \rightarrow hom_{\scrA}(X_1,X_d)[2-d]
\]
for any $(X_1,\ldots, X_d) \in $(Ob$ \, \scrA)^d$.  The operations $\mu^d_{\scrA_q}$ should be $q$-adically convergent (for us they will always vanish for $q \gg 0$).   In applications, such deformations of Fukaya categories will arise from (partial) compactifications of symplectic manifolds $M\subset \bar{M}$, obtained by adding in a divisor $\Delta = \bar{M} \backslash M$ disjoint from the Lagrangian submanifolds under consideration.  The operation $\mu_{\scrF(\bar{M}),j}^d$ will count pseudoholomorphic $d$-gons with multiplicity $j$ over $\Delta$, so setting $q=0$ corresponds to working in the open part $M=\bar{M}\backslash \Delta$.  Assuming that one can achieve transversality within the class of almost complex structures which make $\Delta \subset \bar{M}$ a pseudoholomorphic submanifold, positivity of intersections ensures that the $\mu^d_{\scrA_q}$ are compatible with the filtration on Floer cochain complexes in $\bar{M}$ by powers of $q$.  Under suitable monotonicity hypotheses on $\bar{M}$, one can further set $q=1$ since all series converge.  

The formal parameter $q$ of a deformation arising from a partial compactification of an open symplectic manifold with $c_1=0$ typically has non-zero degree.

\begin{Lemma} \label{Lem:DeformationDegree}
If the first Chern class $c_1(\bar{M}) = r[\Delta]$, with $r\in \bQ$, then $deg(q) = 2r$. 
\end{Lemma}

This is a consequence of the index theorem for the Cauchy-Riemann equations with totally real boundary conditions, see  \cite[Section 11]{FCPLT}.  It is most familiar in the case when $c_1(\bar{M})=0$, and the deformed category is $\bZ$-graded, or when $\Delta \subset \bar{M}$ is an anticanonical divisor and the deformation parameter has degree 2 (for instance, deforming the cohomology of the zero-section $T^n \subset (\bC^*)^n$ from an exterior algebra to the Clifford algebra on compactifying to $\bC\bP^n$, with $\Delta$ the toric anticanonical divisor).   

Suppose then $c_1(\bar{M})$ is effective, that $c_1(M) = 0$,  and fix a holomorphic volume form $\eta \in H^0(K_{\bar{M}})$ with \emph{poles} of order $r$ along the divisor $\Delta \subset \bar{M}$.  We denote by $\iota^{\eta}$ the absolute indices for generators of Floer complexes between graded Lagrangian submanifolds (which are assumed to lie within the open part $M$).   If $L'$ is obtained as a Hamiltonian perturbation of $L$ for a Hamiltonian function whose restriction to $L$ is Morse,  the index $i^{\eta}$ of a point of $L\cap L'$ is equal to the Morse index of the corresponding critical point. Suppose now one has a holomorphic polygon $u:D\rightarrow \bar{M}$ with Lagrangian boundary conditions in $M$.  In this setting,
\begin{equation} \label{Eqn:IndexFormula}
\textrm{vdim}(u) = i^{\eta}(x_0) - \sum_{j=1}^k i^{\eta}(x_k) + k -2 + 2r \Delta \cdot \im (u)
\end{equation}
for a holomorphic $(k+1)$-gon with one outgoing boundary puncture $x_0$ (the sign of the last term corresponds to the fact that $\eta$ has poles on $\Delta = r\cdot PD[c_1(\bar{M})]$, and the virtual dimension is increased by adding \emph{positive} Chern number components).

\subsection{Formality\label{Section:Formality}}

A natural and recurring question studying Fukaya categories is to determine the $A_{\infty}$-structure on an algebra  $A=\oplus_{i,j} HF(L_i,L_j)$, where the $\{L_i\}$ are some finite collection of Lagrangians distinguished by the geometry in some fashion.  In general this is a rather intractable problem, but in the case when $A$ is smooth and commutative, $A_{\infty}$-structures up to gauge equivalence correspond to classical Poisson structures, by Kontsevich's formality theorem.  The particular case of this that pertains to exterior algebras was used by Seidel in computing $\scrF(\Sigma_2)$, and will play an important role later.

Let $A$ be a $\bZ_2$-graded algebra.  We will be interested in $\bZ_2$-graded $A_{\infty}$-structures on $A$ which have trivial differential. Such a structure comprises a collection of maps 
\[
\{ \alpha^i: A^{\otimes i} \rightarrow A \}_{i\geq 2}
\]
of parity $i$; usually we will assume that the product agrees with the given product on $A$, and consider $(\alpha^i)_{i\geq 3}$.  We introduce the mod 2 graded Hochschild cochain complex
\[
CC^{\bullet+1}(A,A) \ = \ \prod_{i \geq 2} Hom^{(\bullet+i)}(A^{\otimes i},A)
\]
where $Hom^{(\bullet +i)}$ means we consider homomorphisms of parity $\bullet+i$. This carries the usual differential $d_{CC^{\bullet}}$ and the Gerstenhaber bracket $[\cdot, \cdot]$.

\begin{Remark}  \label{Rem:Pronilpotent}
The shift in grading follows the convention of \cite{Seidel:HMSgenus2}, since we will later imitate some arguments from that paper. The shift fits with a view of $A_{\infty}$-structures as defined by formal super-vector fields,  but disagrees with the usual conventions of the bar complex.  One typically incorporates a formal parameter $q$ of even degree into $CC^{\bullet+1}(A,A)$, which makes the resulting Lie algebra $\frak{g} = CC^{\bullet+1}[[q]]$ filtered pronilpotent.  This is necessary to guarantee convergence of formal gauge transformations; $L_{\infty}$-quasi-isomorphisms of filtered pronilpotent $L_{\infty}$-algebras  induce bijections between equivalence classes of Maurer-Cartan elements.  In the sequel, we will encounter gauge transformations on exterior algebras $\Lambda^*(\bC^3)$ arising from formal diffeomorphisms of $\bC^3$ which can be seen explicitly to act on $A_{\infty}$-structures defined over $\bC$, cf. Equation \ref{Eqn:exp-gamma}.  Geometrically, the formal parameter $q$ can be incorporated as that of a deformation obtained by partial compactification, in the sense of Equation \ref{Eqn:Deformation}.  The power of $q$ in any given expression is therefore determined by Lemma \ref{Lem:DeformationDegree}, and can be reconstructed from the other geometric data. To simplify notation, we omit this parameter throughout, compare to \cite[Sections 4,5]{Seidel:HMSgenus2}.
\end{Remark}
 
  $A_{\infty}$-structures $(\alpha^i)$ on $A$ which extend its given product correspond to elements of $CC^{1}(A,A)$ satisfying the Maurer-Cartan equation
\begin{equation} \label{Eqn:Maurer-Cartan}
\partial \alpha + \frac{1}{2} [\alpha, \alpha] = 0\qquad \textrm{with} \quad \alpha^1 = \alpha^2 = 0.
\end{equation}
The first terms of the Maurer-Cartan equation give
\begin{equation} \label{Eqn:Maurer-Cartan-FirstTerms}
\partial \alpha^3 = 0; \ \partial \alpha^4 + \frac{1}{2}[\alpha^3,\alpha^3]=0; \ \partial \alpha^5 + [\alpha^3,\alpha^4]=0, \ldots
\end{equation}
Elements $(g^i)_{i\geq 1}$ of $CC^0(A,A)$, comprising maps $A^{\otimes i} \rightarrow A$ of parity $i-1$,   act by gauge transformations on the set of Maurer-Cartan solutions $(\alpha^i)_{i\geq 2}$, at least when suitable convergence conditions apply.   Suppose that $A$ is finite-dimensional in each degree; let $g \in CC^0(A,A)$ have constant term $g^0 \in A^1$ vanishing. Define
\begin{equation} \label{Eqn:exp-gamma}
\begin{cases}
\phi^1 = \mathrm{id} + g^1 + \frac{1}{2}g^1 g^1 + \cdots = \exp(g^1), \\
\phi^2 = g^2 + \frac{1}{2} g^1 g^2 + \frac{1}{2} g^2(g^1 \otimes \mathrm{id}) +
\frac{1}{2} g^2(\mathrm{id} \otimes g^1) + \textstyle\frac{1}{3} g^2(g^1 \otimes g^1) + \cdots, \\
\vdots
\end{cases}
\end{equation}
The term $\phi^j$ is the sum of all possible concatenations of components of $g$ to yield a $j$-linear map. For a term involving $r$ components, which may be ordered in $s$ different ways compatibly with their appearance in the concatenation,  the constant in front of the associated term is $s(r!)^{-1}$, which  ensures convergence of $\phi^j$.   If $\alpha$ and $\tilde\alpha$ are Maurer-Cartan elements which are related by $(g^i)$, the associated $A_\infty$-structures  are related by $\phi$, which is an $A_\infty$-isomorphism. There are obvious simplifications to the formulae when $g^1=0$, which will be the case for gauge transformations acting trivially on cohomology relating solutions with $\alpha^2=0$.

\begin{Lemma} \label{Lem:DeformedHH}
Let $\scrA$ denote the $A_{\infty}$-algebra defined by the Maurer-Cartan element $\alpha \in CC^1(A,A)$, with $H(\scrA) = A$.  The Hochschild cohomology $HH^*(\scrA,\scrA)$ is the cohomology of the complex $CC^{\bullet}(A,A)$ with respect to the twisted differential $d_{CC^{\bullet}} + [\cdot, \alpha]$.
\end{Lemma}

\begin{proof}
The Maurer-Cartan equation for $\alpha$ implies that the twisted differential does square to zero. The Lemma follows on comparing the resulting complex with the bar complex defining Hochschild cohomology of an $A_{\infty}$-algebra given in Equation \ref{Eqn:Hochschild}.
\end{proof}

Kontsevich's remarkable result \cite{Kontsevich}, specialised to the case of relevance in the sequel, is then:

\begin{Theorem}[Formality Theorem, Kontsevich] \label{Thm:KontsevichFormality}
When $A=\Lambda(V)$ is an exterior algebra, with its natural $\bZ_2$-grading placing $V$ in degree 1, there is an $L_{\infty}$-quasi-isomorphism
\[
\Phi: CC^{\bullet}(A,A) \longrightarrow Sym^{\bullet}(V^{\vee}) \otimes \Lambda(V)
\]
between the Hochschild complex and its cohomology, the Lie algebra of polyvector fields.
\end{Theorem}

\begin{Remark}\label{Rem:Quasi-Inverse}
In fact, this $\Phi$ is strictly left-inverse to an explicit quasi-isomorphism
\[
\Psi: \Sym^{\bullet}(V^{\vee}) \otimes \Lambda(V) \rightarrow CC^{\bullet}(\Lambda(V),\Lambda(V))
\]
defined by Kontsevich. The components of $\Psi$ are linear combinations $\Psi^i = \sum \lambda^i \otimes \mathcal{U}_{\Gamma^i}$, where $\mathcal{U}_{\Gamma^i}$ is a multilinear operation indexed by a certain class of graphs $\{\Gamma^i\}$, and the coefficients $\lambda^i$ are explicit integrals over configuration spaces of points in the upper half-plane.
\end{Remark}

The Lie algebra of polyvector fields on a finite-dimensional vector space $V$ is the algebra 
\[
Sym^{\bullet}(V^{\vee}) \otimes \Lambda(V) \ = \ \bC[[V]] \otimes \Lambda(V).
\]
The Lie algebra structure comes from the Schouten bracket
\begin{equation}
\begin{aligned} {}
& [f \,\xi_{i_1} \wedge \cdots \wedge \xi_{i_k}, g \,\xi_{j_1} \wedge \cdots \xi_{j_l}] = \\
& \textstyle \qquad \sum_q (-1)^{k-q-1} f \, (\partial_{i_q}g) \,
\xi_{i_1}  \wedge \cdots \wedge \widehat{\xi_{i_q}} \wedge \cdots
\wedge \xi_{i_k} \wedge \xi_{j_1} \wedge \cdots \wedge \xi_{j_l} + \\ & \textstyle \qquad
 \sum_q (-1)^{l-q+ (k-1)(l-1)} g \,(\partial_{j_q}f)\, \xi_{j_1} \wedge \cdots \wedge \widehat{\xi_{j_q}} \wedge \cdots \wedge \xi_{j_l} \wedge \xi_{i_1} \wedge
 \cdots \wedge \xi_{i_k}.
\end{aligned}
\end{equation}
The quasi-isomorphism $\Phi$ identifies gauge equivalence classes of Maurer-Cartan solutions.  In particular, since the Schouten bracket vanishes on functions, any formal function $W\in\bC[[V]]$ defines a solution of Maurer-Cartan, hence an $A_{\infty}$-structure on $\Lambda(V)$.  Lemma \ref{Lem:DeformedHH} and the Formality Theorem \ref{Thm:KontsevichFormality} imply that the Hochschild cohomology of a $\bZ_2$-graded $A_{\infty}$-structure on $\Lambda(V)$ defined by an element $\alpha \in HH^1$ is given by the cohomology of the $\bZ_2$-graded complex
\begin{equation}
 (\bC[[V]]\otimes \Lambda^{ev}(V))\  \xrightleftharpoons[{\ [\cdot, \alpha]\ }]{{\ [\cdot, \alpha]\ }} \ (\bC[[V]]\otimes \Lambda^{odd}(V)).
\end{equation}

The following Example will recur in the discussion of Fukaya categories of quadric hypersurfaces in Section \ref{Section:Quadric}.

\begin{Example} \label{Ex:AinftyExterior}
Let $V$ have dimension 1 and $A=\Lambda(V)$, with basis $1, x$ say.  Any formal function $W\in\bC[[x]]$ of degree $\geq 2$ defines an $A_{\infty}$-structure on $A$, which induces the given product if and only if the degree 2 component vanishes.  Up to formal change of variables (gauge transformation), any 
\[
W(x) = x^k + O(x^{k+1}) \ \simeq \ y^k \quad \textrm{via} \ y = x(1+\alpha_1x + \alpha_2x^2 + \cdots)
\]  
is equivalent to a monomial, hence the resulting $A_{\infty}$-structures are equivalent to those defined by monomial functions.  Equation \ref{Eqn:HHPolyFields} simplifies to the complex
\[
0 \longrightarrow \bC[[x]] \otimes \langle dx\rangle   \xrightarrow{\ \iota_{dW} \ } \bC[[x]] \longrightarrow 0
\]
with cohomology $\bC[[x]]/\langle W'(x) \rangle$ given by the Jacobian ring, so  the Hochschild cohomology of $\scrA_k = (A, W=x^k)$ has rank $k-1$. 
\end{Example}

\begin{Remark} \label{Rem:HH}
For the ring $R=\bC[x]/ \langle x^2-1 \rangle$, the Hochschild cohomology\footnote{Classically, $HH^*$ would be 2-periodic and we're just taking $HH^0$ and $HH^1$.} group of the category of $\bZ_2$-graded modules has rank 2 if $deg(x)$ is even and rank 1 if $deg(x)$ is odd (the latter fact is the computation of Example \ref{Ex:AinftyExterior} in the special case $k=1$, where the underlying product on $A$ is itself deformed).  This comes down to the fact that the group
\[
Hom_{R\otimes R}(R\otimes R, R)
\]
of bimodule homomorphisms has rank 2 in the first case, comprising the elements $1\otimes 1 \mapsto \alpha 1 + \beta x$, but only rank 1 in the second, since for a $\bZ_2$-graded map one must have $1\otimes 1 \mapsto \alpha 1$.  From another viewpoint, when $deg(x)$ is even, $R$ is semisimple, splitting into two orthogonal idempotent summands, but when $deg(x)$ is odd the idempotents are not of pure degree and there is no such splitting.
\end{Remark}

\subsection{The Fukaya category of a curve\label{Section:Curve}}

Take $\Sigma_g\cong C \stackrel{p}{\rightarrow} \bP^1$ to be the unique curve branched over the points $\{0, \xi^j, \ 1\leq j\leq 2g+1\}$, where $\xi$ is a primitive $(2g+1)$-st root of unity.  We can exhibit embedded circles in $C$ via matching arcs amongst these points, i.e. the projections of $\bZ_2$-invariant curves to $\bP^1$. For instance, when $g=2$,  5 Lagrangian submanifolds $\{L_j\}$ of $C$ are determined by the arcs of the pentagram on the left of Figure \ref{Figure:Pentagram} (the vertices are the $5$th roots of unity); more generally there is a spikier $(2g+1)$-gram defining $2g+1$ associated spheres permuted by the obvious cyclic symmetry.  In an act of shameless innumeracy, we will refer to all these figures as \emph{pentagrams}.  If we divide $C$ by the lift of this cyclic group action from $\bP^1$ to $C$,  all the circles $\{L_j\}$ are identified with a single immersed circle $\bar{L} \subset S^2_{orb}$ in an orbifold $S^2$ with 3 orbifold points each of order $2g+1$ (these come from the unique preimage of $0\in\bP^1$ in $C$ and the two preimages of $\infty \in \bP^1$).  This immersed curve actually bounds no teardrops, hence Floer theory can be defined unproblematically for it, or one can work equivariantly upstairs.

\begin{center}
\begin{figure}[ht]
\includegraphics[scale=0.5]{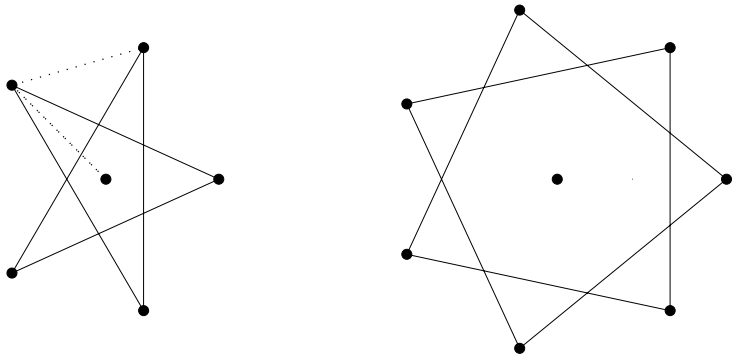} 
\caption{Lagrangian spheres from the $g=2$ pentagram and its $g=3$ cousin; on the left one radial path and one basic path are also dotted.\label{Figure:Pentagram}}
\end{figure}
\end{center}

For any symplectic surface $\Sigma_g, \omega_{\Sigma}$ of genus $g\geq 2$, let $\sigma: ST\Sigma \rightarrow \Sigma$ be the projection from the unit tangent bundle.  Seidel defines the \emph{balanced} Fukaya category, an $A_{\infty}$-category linear over $\bC$, with respect to an auxiliary choice of one-form  $\theta \in \Omega^1(ST\Sigma)$ which is a primitive for the pullback of the symplectic form, $d\theta = \sigma^*\omega_{\Sigma}$.  (Such forms exist since, by the Gysin sequence, $H^2(ST\Sigma,\bR) = 0$ for $g(\Sigma) \geq 2$.)  Objects of the balanced Fukaya category are simple closed curves $L \subset \Sigma$ for which $\int_L \iota_L^*\theta = 0$, where $\iota_L: L \rightarrow ST\Sigma$ is the tangent lift of $L\hookrightarrow \Sigma$ (these curves should be equipped with orientations and $Spin$ structures as usual).  Homotopically trivial curves are never balanced; there is a unique balanced representative up to Hamiltonian isotopy in every other homotopy class of embedded curve.  The balancing condition is an analogue of monotonicity; it implies that energy and index are correlated for holomorphic curves with balanced boundary conditions, which in turn means that the balanced category $\scrF(\Sigma)$ is defined over $\bC$.  Different  choices of $\theta$ yield quasi-isomorphic categories, and the mapping class group $\Gamma_g$ acts by autoequivalences of $D^{\pi}\scrF(\Sigma_g)$, cf. \cite{Seidel:HMSgenus2}.

\begin{Remark} \label{Remark:2objects-1}
An isotopy class of essential simple closed curve $\gamma \subset \Sigma$ defines two objects in $\scrF(\Sigma)$ modulo shifts, namely the balanced representative of $\gamma$ equipped with either $Spin$ structure.  Label these $\gamma$ and $\bar{\gamma}$ for the non-trivial respectively trivial $Spin$-structure.  Then
\begin{equation} \label{Eqn:TwoSpinChoices}
HF(\gamma,\gamma) \cong \bC \cong HF(\bar{\gamma},\bar{\gamma}); \quad HF(\gamma, \bar{\gamma}) = 0.
\end{equation}
One can equivalently regard $\bar{\gamma}$ as the pair $(\gamma, \xi)$ where $\xi \rightarrow \gamma$ is the non-trivial line bundle with holonomy in $\bZ_2$, from which point of view one sees that the group $H^1(\Sigma;\bZ_2)$ acts on $\scrF(\Sigma)$, tensoring by flat line bundles. \end{Remark}

\begin{Remark} \label{Remark:GpActingOnCurve}
One can combine the actions of diffeomorphisms and tensoring by flat line bundles into an action of the canonical split extension $\widehat{\Gamma}_{split}(\Sigma_g)$ of $\Gamma_g$ by $H^1(\Sigma;\bZ_2)$.  Strictly, this is only a \emph{weak} action, meaning that the functors $\scrG_{g_i}$  associated to group elements $g_i$ satisfy $\scrG_{g_1}\scrG_{g_2} \cong \scrG_{g_1g_2}$ (but we do not keep track of coherence amongst these isomorphisms).    We will suppress this point in the sequel, but it would probably enter in reconciling the appearance of $\widehat{\Gamma}_{split}$ with Remark \ref{Remark:GpActingOnModuliSpace}; compare to \cite[Remark 10.4]{FCPLT}.
\end{Remark}

We collect a number of other facts from the work of Seidel \cite{Seidel:HMSgenus2} and Efimov \cite{Efimov}:  let $V$ be a 3-dimensional vector space over $\bK = \bC$. 
\begin{enumerate}
\item The pentagram circles $\{L_1,\ldots, L_{2g+1}\}$, each equipped with the non-trivial $Spin$ structure, split-generate $\scrF(\Sigma_g)$ (this is an application of Proposition \ref{Prop:Split-Generate-basic});
\item The Floer cohomology algebra $A=\bigoplus_{i,j=1}^{2g+1} HF(L_i, L_j) \cong \Lambda^*(V) \rtimes \bZ/(2g+1)$, where $V$ is graded in degree 1 and the exterior algebra inherits the obvious $\bZ/2$-grading;
\item The $A_{\infty}$-structure $\scrA$ on $A=H(\scrA)$ is characterised up to $A_{\infty}$-equivalence as follows: the operation $\mu^d = \mu^d_0 + \mu^d_1 + \ldots $ decomposes as a sum of pieces $\mu^d_i$ of degrees $6-3d+(4g-4)i$; and there is a basis $\{x_1, x_2, x_3\}$ for $V$ such that, if $e$ generates $\Lambda^0(V)$, then
\begin{equation} \label{Eqn:AinftyCurve}
\mu^3_0 (x_1,x_2,x_3) = -e; \quad \mu^{2g+1}_1 (x_i, \ldots, x_i) = e \ \textrm{for} \  i\in \{1,2,3\}
\end{equation}
\item Concretely, the operation $\mu^d_i$ counts holomorphic polygons $u$ which, projected to $\bP^1$, have $i = mult_u(\infty) + (1/2) mult_u(0)$.
\end{enumerate}

The critical statement (3) is a ``finite determinacy" theorem, and at heart is an application of Theorem \ref{Thm:KontsevichFormality}.  Unfortunately, we will not be able to use the result in quite this form:  even working over $\bC^*\subset \bP^1$, the Fukaya category of the relative quadric $Z_{\neq 0,\infty}$ which appears later is not $\bZ$-graded.    However, Seidel and Efimov deduce this finite determinacy statement from the result in singularity theory that a power series with an isolated singularity at the origin is determined up to formal change of variables by a finite part of its Taylor series.  This result does not rely on gradings, so we will reduce ourselves to applying that directly.

%%%%%%%%%%%%%%%%%%%%%%%%%%%%%%%%%%%%%%%%%%%%%

%%%%%%%%%%%%%%%%%%%%%%%%%%%%%%%%%%%%%%%%%%

\section{From the base locus to the relative quadric}\label{Sec:BaseToQuadric}

\subsection{The Fukaya category of a quadric\label{Section:Quadric}}

Recall from Equation \ref{Eqn:Kuznetsov} we are supposed to identify, categorically, a pair of points $S^0$ and a smooth even-dimensional quadric $Q$.  For completeness, we discuss quadrics of both even and odd dimension. 

\begin{Lemma} \label{Lem:uniquesphere}
A quadric hypersurface $Q \subset \bP^{n+1}$contains a distinguished isotopy class of Lagrangian sphere $L$, which is the unique sphere arising as the vanishing cycle of an algebraic degeneration. If $n>1$, this sphere  lies in the nilpotent summand $\scrF(Q;0)$.
\end{Lemma}

\begin{proof}
A singular quadric is a cone on a quadric of lower-dimension, hence a quadric with isolated singularities has a unique singular point.  In particular, a quadric can have at most one node. Moreover, the space of quadrics with one node is connected, since the generic point of the space of singular quadrics is given by specifying the location of the node and the smooth quadric of lower dimension which is the base of the cone, and these parameter spaces are irreducible.  This implies the first statement by a standard Moser argument.  The minimal Maslov number of the Lagrangian sphere is $2n$, so when $n>1$ the sphere trivially lies in the nilpotent summand of the category (the obstruction class $\frak{m}_0$ counts Maslov 2 disks through the generic point of $L$).  
 \end{proof}
 
 \begin{Lemma} \label{Lem:QuadricTwist}
 Let $n=\dim_{\bC}(Q)$ and let $\tau_L$ denote the Dehn twist in the Lagrangian sphere $L$ of Lemma \ref{Lem:uniquesphere}. 
 \begin{enumerate}
\item  If $n$ is odd, then $\tau_L$ is Hamiltonian isotopic to the identity.
 \item If $n$ is even, then $\tau_L \in \pi_0\Symp(Q)$ has order 2.
 \end{enumerate}
 \end{Lemma}
 
 \begin{proof}
Consider a Lefschetz pencil of hyperplane sections of $Q^{n} \subset \bP^{n+1}$, with general fibre $Q^{n-1}$ and base locus $Q^{n-2}$.  An Euler characteristic count shows that there are exactly two singular fibres, and the vanishing cycles for both singularities must be given by the sphere $L$, by Lemma \ref{Lem:uniquesphere}.  Considering the monodromy of the corresponding family over $\bP^1$, this implies that for any $n$, the square $\tau_L^2$ must be symplectically and hence Hamiltonian isotopic to the identity.  Now consider instead the pencil of quadrics defined by
\begin{equation} \label{Eqn:Pencil}
Q_0 = \{ \sum_{j=0}^{n+1} z_j^2 = 0\}, \qquad Q_1 = \{\sum_{j=0}^{n+1} \lambda_j z_j^2 = 0\}
\end{equation}
which has singular fibres precisely at the $\{\lambda_j\}$. More precisely $tQ_0-Q_1$ defines a  singular hypersurface if and only if $t\in\{\lambda_j\}$, in which case it has a node.  The vanishing cycles all equal $L$, and the monodromy of the corresponding family over $\bP^1$ shows that  $\tau_L^{n+2} = \id$. If $n$ is odd, the two results imply that $\tau_L \simeq \id$. If $n$ is even, the Dehn twist cannot have order 1 rather than 2, since it acts non-trivially on homology (reversing the orientation of $L$, on which it acts antipodally).
 \end{proof}

We recall the quantum cohomology ring of the quadric, as determined by Beauville \cite{Beauville}.

\begin{Lemma} \label{Lem:QuadricQH}
Let $Q \subset \bP^{n+1}$ denote a smooth $n$-dimensional quadric hypersurface.  The quantum cohomology ring admits the presentation
\begin{itemize}
\item if $n=4k+2$, $\langle h, a, b \ | \ h^{n+1} = 4h, \ h(a-b) = 0, \ ab = (h^n-2)/2, \ a^2=1=b^2 \rangle$;

\item if $n=4k$, $\langle h, a, b \ | \ h^{n+1} = 4h, \ h(a-b) = 0, \ ab = 1, \ a^2=(h^n-2)/2=b^2 \rangle$;

\item if $n=2k+1$, $\langle h \ | \ h^{n+1} = 4h \rangle$.

\end{itemize}
The element $h\in QH^2$ is induced by the hyperplane class, whilst $a,b \in QH^{2k}$ are the classes of the two ruling $\bP^k$'s on the quadric.
\end{Lemma}

Now take a pencil of quadrics $\{Q_t\}_{t \in \bP^1} \subset \bP^{n+1}$ generated by two smooth quadrics $Q_0, Q_1$ meeting transversely.  Let $Z$ denote the \emph{relative quadric}, the blow-up of $\bP^{n+1}$ along $Q_0\cap Q_1$.

\begin{Lemma} \label{Lem:AmpleBlowUp}
If $Y$ is the blow-up of $\bP^{n+1}$ along a $(d_1,d_2)$-complete intersection, with $d_1\geq d_2$,  then $|pH-E|$ is very ample provided $p\geq d_1+1$.
\end{Lemma}

A proof is given in either of \cite{Coppens, Gimigliano-Lorenzini}.   Therefore  $Z$ is a Fano variety, in fact the anti-canonical class is very ample; this is obvious if one views $Z\subset \bP^{n+1} \times \bP^1$ as a divisor of bidegree $(2,1)$.  We equip $Z$ with the monotone symplectic form.  For generic pencils of quadrics, the fibration $w:Z\rightarrow \bP^1$ of proper transforms of quadric hypersurfaces is a Lefschetz fibration with $(n+2)$ singular fibres; Lemma \ref{Lem:uniquesphere} shows these all define the same vanishing cycle in $Q$.

\begin{Example} \label{Ex:LayGround}
Suppose $n=2g$.  The monodromy 
\begin{equation} \label{Eqn:monodromy}
\pi_1(\bP^1 \backslash \{\lambda_j\}) \rightarrow \pi_0\Symp(Q)
\end{equation}
of the pencil \eqref{Eqn:Pencil} factors through the cyclic subgroup $\bZ_{2}$ generated by the Dehn twist $\tau_L$, by Lemma \ref{Lem:QuadricTwist}. The monodromy therefore defines a hyperelliptic curve $\Sigma_g\rightarrow \bP^1$ branched over the points $\{\lambda_0,\ldots, \lambda_{2g+1}\}\subset \bC \subset \bP^1$.
\end{Example}

\begin{Lemma} \label{Lem:QuadricSplitGenerated}
The nilpotent summand  $\scrF(Q;0)$ is split-generated by the Lagrangian sphere of Lemma \ref{Lem:uniquesphere}.
\end{Lemma}

\begin{proof}
Denote by $w: Z^{n+1} \rightarrow \bP^1$ the Lefschetz fibration constructed above.  
We compute the cycle class $\mathcal{C}(w)$;   we work with the given integrable almost complex structure $J$. Write $L$ for the class of a line in $\bP^{n+1}$, $R$ for the ruling class of the exceptional divisor $E\subset Z$, and $H$ for the hyperplane class on $\bP^{n}$; so the intersection pairing $H_{2n}(Z)\times H_2(Z) \rightarrow \bZ$ is given by 
\[
H.L = 1, \ H.R = 0, \ E.L = 0, \ E.R=-1.
\]
Since the class of the fibre $2H-E$ is base-point free, it meets every effective curve non-negatively, so effective curves live in classes $dL-rR$ with $r\leq 2d$ (and $d\geq 0$). The space of sections of $w$ in a homology class $\beta$ has virtual complex  dimension $c_1(T^{vt}(Z))[\beta]+\dim_{\bC}(Q) = c_1(Z)[\beta]+n-2$.   Therefore the space of sections has dimension greater than the fibre, and the cycle class vanishes, unless $0 \leq c_1(Z)[\beta] = (n+2)d-r \leq 2$, which forces $d=0$ and $r\in \{-1,-2\}$.    Since sections satisfy $\langle 2H-E, [\beta]\rangle =1$, the case $r=-2$ is also excluded, so we are left with sections which are ruling curves,  $d=0$ and $r=-1$.  All such curves lie inside the exceptional divisor $E$, since $E.R = -1$.   Moreover, the curves are regular by the ``automatic regularity" criterion of \cite[Lemma 3.3.1]{McD-S}: we are dealing with holomorphic spheres for an integrable almost complex structure for which all summands of the normal bundle have Chern number $\geq -1$.  The associated cycle class is therefore a copy of the base locus inside the fibre $\mathcal{C}(w)=Q_0\cap Q_1\subset Q$.    This is a multiple of the first Chern class of $Q$ by the Lefschetz theorem; the result now follows from Proposition \ref{Prop:Split-Generate}.
\end{proof}

\begin{Lemma} \label{Lem:Ainfinitytrivial}
Let $L^g\subset Q^{2g}$ be the Lagrangian sphere of Lemma \ref{Lem:uniquesphere}, in an even-dimensional quadric.  Then $HF(L,L)$ is semisimple, and carries a formal $A_{\infty}$-structure.\end{Lemma}

\begin{proof}
Since $L$ has minimal Maslov number $4g-2 > dim_{\bR}(L)+1$, its Floer differential is undeformed and additively, $HF(L,L) \cong H^*(S^{2g})$.  We claim that  the ring structure on this group is semisimple, so $HF(L,L) \cong \bC[t] / \langle t^2=1\rangle$. From Lemma \ref{Lem:QuadricQH}, the generators $a,b$ of $QH^{2g}(Q)$  are invertible elements in the quantum cohomology ring, hence have invertible image in $HF^{2g}(L,L)$, and the natural map $QH^*(Q) \rightarrow HF^*(L,L)$ is surjective. This implies that  the generator of $HF^{2g}(L,L)$ has non-zero square; see for instance Biran-Cornea \cite[Proposition 6.33]{BiranCornea}. Formality of the $A_{\infty}$-structure now follows from the Hochschild cohomology computation of Remark \ref{Rem:HH}.  More precisely, that computation shows that up to rescaling, the unique non-formal $A_{\infty}$-structure treats the idempotent summands of $1_L$ asymmetrically, and hence the non-formal structure admits no autoequivalence which acts on $HF(L,L)$ in the same way as the Dehn twist $\tau_L$ (this must act non-trivially since it reverses orientation on $L$, hence interchanges the two idempotents).
\end{proof}

We remark that one can also derive the above formality result from the fact that the Lagrangian sphere is invariant under an anti-symplectic involution of the quadric, together with $4$-divisibility of all its Maslov indices, using \cite[Chapter 8]{FO3}. 

\begin{Lemma} \label{Lem:Ainfinitytrivial2}
Let $L^{2k+1}\subset Q^{2k+1}$ be the Lagrangian sphere of Lemma \ref{Lem:uniquesphere}, in an odd-dimensional quadric.  Then $HF(L,L) \cong Cl_1$ is quasi-isomorphic to the Clifford algebra with vanishing higher order products.\end{Lemma}

\begin{proof}
Because the minimal Maslov number is $2k$, additively $HF(L,L) \cong H^*(S^{2k+1})$, which with its natural mod 2 grading is additively isomorphic to an exterior algebra.  According to Corollary \ref{Cor:QuiltsQH}, the Hochschild cohomology of the category $\scrF(Q;0)$ is isomorphic to the $0$-generalised eigenspace of $QH^*(Q)$, which from Lemma \ref{Lem:QuadricQH} has rank 1 for an odd-dimensional quadric. 
According to Section \ref{Section:Formality}, $A_{\infty}$-structures on the exterior algebra on one generator are determined by formal power series $W \in \bC[[x]]$ in one variable, and Example \ref{Ex:AinftyExterior} shows both that the only structure with the correct Hochschild cohomology is the one in which the product is deformed, and that up to gauge equivalence one can assume that the higher products vanish.
\end{proof}

Let $S^0$ denote a zero-dimensional sphere (which is a rather trivial symplectic manifold), and continue to write $Cl_1$ for the Clifford algebra structure on $\Lambda^*(\bR)$.

\begin{Corollary} \label{Cor:QuadricSorted}
We have the following equivalences:
\begin{itemize}
\item If $Q\subset \bP^{2g+1}$ is an even-dimensional quadric, 
$D^{\pi}\scrF(Q;0) \simeq D^{\pi}\scrF(S^0)$.
\item If $Q\subset \bP^{2g}$ is an odd-dimensional quadric, 
$D^{\pi}\scrF(Q;0) \simeq D^{\pi}(mod$-$Cl_1)$, except when $dim(Q)=1$ and the nilpotent summand of the category is empty.
\end{itemize}
\end{Corollary}

\begin{proof}
This is an immediate consequence of Lemmata \ref{Lem:QuadricSplitGenerated}, \ref{Lem:Ainfinitytrivial} and \ref{Lem:Ainfinitytrivial2}.
\end{proof}

\begin{Remark}
Every twisted complex in $Tw\,\scrF(S^0)$ is quasi-isomorphic to a pair of graded vector spaces, one supported at each point of $S^0$.  The endomorphism algebra of any idempotent is therefore a sum of matrix algebras, hence has rank two if and only if it is a direct sum of two one-dimensional matrix algebras. It follows that, for an even-dimensional quadric, $D^{\pi}\scrF(Q;0)$ contains a unique spherical object up to quasi-isomorphism.  
\end{Remark}

\subsection{Hyperelliptic curves and quadrics\label{Subsec:Pencil}}  This section expands on Example \ref{Ex:LayGround}.  If $\Sigma \rightarrow \bP^1$ is a hyperelliptic curve branched over the (distinct) points $\{\lambda_1,\ldots, \lambda_{2g+2}\} \subset \bC\subset \bP^1$, there is an associated pencil of quadric hypersurfaces $\{sQ_0 + tQ_1\}_{[s:t]\in\bP^1}$, where
\begin{equation} \label{Eqn:PencilCoord}
Q_0 = \left\{ \sum z_j^2 = 0\right\} \subset \bP^{2g+1}; \quad Q_1 = \left\{ \sum \lambda_j z_j^2 = 0 \right\} \subset \bP^{2g+1}.
\end{equation}

\begin{Remark}\label{Rem:BranchingInfinity} 
The fundamental group of the universal family of hyperelliptic curves is a central extension 
\[
1 \rightarrow \bZ_2 \longrightarrow \Gamma_g^{hyp} \longrightarrow \pi_1\Conf_{2g+2}(\bC\bP^1) \rightarrow 1
\]
with the first factor generated by the hyperelliptic involution. More explicitly, as Seidel pointed out (MIT reading group, unpublished notes), to identify the hyperelliptic curves
\[
y^2 = \prod_j (\lambda_j - \mu_j x)  \quad \textrm{and} \quad  y^2 = \prod_j (t_j \lambda_j - t_j\mu_j x) 
\]
associated to points $(\lambda_j, \mu_j),  (t_j\lambda_j, t_j\mu_j) \subset \bC^2\backslash \{0\}$ relies on a choice of square root of $\prod t_j$, which is how the central $\bZ_2$-factor arises when constructing parallel transport maps. (Accordingly, there is no universal hyperelliptic curve  over $\Conf_{2g+2}(\bP^1)$ \cite{Mess},  rather than over a covering space, which is why Addendum \ref{Thm:Action} was stated for once-pointed curves.)  For the $(2,2)$-intersection $Q_0 \cap Q_1 \subset \bP^{2g+1}$, the corresponding fundamental group $\widetilde{\Gamma}$ fits into a sequence
\[
1 \rightarrow \bZ_2^{2g+1} \longrightarrow \widetilde{\Gamma} \longrightarrow \pi_1 \Conf_{2g+2}(\bC\bP^1) \rightarrow 1
\]
Now the first factor $\bZ_2^{2g+1} = \bZ_2^{2g+2} / (\Z_2)$ acts by changing signs of homogeneous co-ordinates in $\bP^{2g+1}$, and this can again be interpreted as choosing square roots of each $t_j$ (modulo changing all choices simultaneously) when identifying 
\[
\left\{ \sum \lambda_j x_j^2 = 0 = \sum \mu_j x_j^2 \right\}  \quad \textrm{and} \quad \left\{ \sum t_j \lambda_j x_j^2 = 0 = \sum t_j \mu_j x_j^2 \right\}.
\]
It follows that the group $\widetilde{\Gamma}$ can be expressed as an extension
\[
1 \rightarrow \bZ_2^{2g} \longrightarrow \tilde{\Gamma}_g^{hyp} \longrightarrow \Gamma_g^{hyp} \rightarrow 1
\]
where the first factor is the subgroup of even elements of $\bZ^{2g+1}$. 
This is the group occuring in Corollary \ref{Cor:Faithful2}. Denote by $\iota$ a generator of $\bZ_2^{2g}$, for instance
 \[
 [x_0: x_1: x_2 : \ldots : x_{2g+1}] \stackrel{\iota}{\longrightarrow}  [-x_0: -x_1: x_2 : \ldots : x_{2g+1}].
 \]
This has fixed point set on $Q_0 \cap Q_1$ the intersection of two quadrics in $\bP^{2g-1}$. Therefore, 
   \[
   \rk H^*(\textrm{Fix}(\iota)) \, < \, \rk H^*(Q_0 \cap Q_1).
   \]    
   The local-to-global spectral sequence of Pozniak \cite{Pozniak} now implies that $HF(\iota) \neq HF(\id)$, which means that $\iota \in \pi_0\Symp(Q_0 \cap Q_1)$ is non-trivial (by contrast, since it has zero Lefschetz number it acts trivially on cohomology).  From here, it is straightforward to deduce Corollary \ref{Cor:Faithful2} from Theorem \ref{Thm:Embed}, either by direct consideration of the action of the mapping class group on $\scrF(\Sigma_g)$,  or by following the argument of Section \ref{Sec:Faithful}.
\end{Remark}

\emph{From now on, we will always work with curves unbranched at infinity, and co-ordinate representations of the associated pencil of quadrics as in (\ref{Eqn:PencilCoord}).}   \newline

There are numerous classical connections between the topology of $\Sigma$ and that of the base locus $Q_0 \cap Q_1$, see \cite{Reid, Wall, AL}; Reid's unpublished thesis is especially lucid. 
\begin{itemize}
\item There is an isomorphism $H^1(\Sigma_g) \cong H^{2g-1}(Q_0 \cap Q_1)$ of odd cohomologies.
\item The variety of $(g-1)$-planes in $Q_0 \cap Q_1$ is isomorphic to the Jacobian $J(\Sigma_g)$.
\item The moduli spaces of smooth hyperelliptic curves and of  pencils of quadrics for which the discriminant has no multiple root co-incide.
\end{itemize}
The last statement in particular implies that, by taking parallel transport in suitable families, there is a canonical representation
\[
\rho: \Gamma_{g,1}^{hyp} \longrightarrow \pi_0\Symp(Q_0 \cap Q_1)
\]
from the hyperelliptic mapping class group of once-pointed curves (unbranched at infinity)  to the symplectic mapping class group of the associated $(2,2)$-complete intersection in $\bP^{2g+1}$.  

\begin{Lemma}[Wall] \label{Ex:genus2}
An $A_k$-chain of curves $\{\gamma_1,\ldots, \gamma_k\} \subset \Sigma$ each invariant under the hyperelliptic involution defines an $A_k$-chain of Lagrangian spheres $\{V_{\gamma_1},\ldots, V_{\gamma_k}\} \subset Q_0 \cap Q_1$.
\end{Lemma}

When $g=2$, this is precisely the content of Lemma \ref{Lem:HolonomySpheresMeetRight}.  In general, it is a special case of Wall's \cite[Theorem 1.4]{Wall}, who showed that an isolated singularity in the base locus of a linear system of quadrics has the same topological type as the singularity in the discriminant of the family, provided all the quadrics have corank at most 1.  For pencils, the only possible singularities of the (isolated) discriminant are multiple points, which are $A_k$-singularities; the Lagrangian spheres in the base locus arise as the vanishing cycles of the corresponding degeneration.    In particular, a simple closed curve $\gamma \subset \Sigma_g$ invariant under the hyperelliptic involution defines a Hamiltonian isotopy class of Lagrangian sphere $V_{\gamma} \subset Q_0 \cap Q_1$. Note that when $g=2$, once the identification $Q_0 \cap Q_1 \cong \scrM(\Sigma_2)$ is fixed, $V_{\gamma}$ is actually defined canonically and not just up to isotopy. 

\begin{Remark}
A classical result of Kn\"orrer \cite{Knorrer} asserts that if $Q_0 \cap Q_1$ is a complete intersection of two quadrics in $\bP^{2g+1}$ with isolated singularities, then the number of singular points of $Q_0 \cap Q_1$ is at most $2g+2$ (these are then necessarily all nodes, and the bound is realised).  The curve $\Sigma_g$ admits at most $g+1$ pairwise disjoint balanced simple closed curves invariant under the hyperelliptic involution, each of which can carry either of two $Spin$ structures, so this fits nicely.  
\end{Remark}

  Consider now a net of quadric $2g$-folds in $\bP^{2g+1}$, spanned by $Q_0 = \{\sum z_j^2=0\}$, $Q_1 = \{\sum \lambda_jz_j^2 = 0\}$ and $Q_2 = \{\sum_j \mu_jz_j^2 = 0\}$.   We suppose the scalars $\lambda_j$ and $\mu_k$ are generic and in particular pairwise distinct. Indeed, we make the stronger hypothesis that the discriminant
\[
\left\{ [s_0:s_1:s_2] \in \bP^2 \, \big| \, det(s_0Q_0+s_1Q_1+s_2Q_2) = 0 \right\}
\]
defines a smooth curve $B\subset \bP^2$ of degree $2g+2$, and hence a smooth surface $K\rightarrow \bP^2$ double covering the plane branched along $B$.     The preimages of a generic pencil of lines in $\bP^2$ define a genus $g$ Lefschetz pencil on $K$ with $(2g+1)(2g+2)$ singular fibres (the degree of the dual curve to the branch locus) and two base-points.  The monodromy of this pencil is well-known:  if $t_i$ denotes the Dehn twist in the curve $\zeta_i$ of Figure \ref{Fig:genus2pencil}, the monodromy when $g=2$ is given by $(t_1t_2t_3t_4t_5)^6 = \id$ (or, in the pointed mapping class group, the same expression is the product of the Dehn twists $t_{\partial_1}t_{\partial_2}$ around the two punctures, which lie in the two halves of the surface given by cutting along the $\{\zeta_j\}$).  For $g>2$ one gets the obvious generalisation $(t_1\ldots t_{2g+1})^{2g+2} = \id$.  After compactifying by adding in one of the two points of the base locus, one can view this monodromy as determined by a representation $\eta: \pi_1(S^2 \backslash \{\mathrm{Crit}\}) \rightarrow \Gamma_{g,1}$.

\begin{center}
\begin{figure}[ht]
\setlength{\unitlength}{1cm}
\begin{picture}(5,2)(0,-1)
\qbezier[200](0,0)(0,1.2)(1.5,1)
\qbezier[200](1.5,1)(2.5,0.85)(3.5,1)
\qbezier[200](3.5,1)(5,1.2)(5,0)
\qbezier[200](0,0)(0,-1.2)(1.5,-1)
\qbezier[200](1.5,-1)(2.5,-0.85)(3.5,-1)
\qbezier[200](3.5,-1)(5,-1.2)(5,0)
\qbezier[60](1,0)(1,0.3)(1.5,0.3)
\qbezier[60](2,0)(2,0.3)(1.5,0.3)
\qbezier[60](1,0)(1,-0.3)(1.5,-0.3)
\qbezier[60](2,0)(2,-0.3)(1.5,-0.3)
\qbezier[60](3,0)(3,0.3)(3.5,0.3)
\qbezier[60](4,0)(4,0.3)(3.5,0.3)
\qbezier[60](3,0)(3,-0.3)(3.5,-0.3)
\qbezier[60](4,0)(4,-0.3)(3.5,-0.3)
\put(1.5,0){\circle{1.1}}
\put(3.5,0){\circle{1.1}}
\qbezier[60](0,0)(0,-0.1)(0.5,-0.1)
\qbezier[60](1,0)(1,-0.1)(0.5,-0.1)
\qbezier[20](0,0)(0,0.1)(0.5,0.1)
\qbezier[20](1,0)(1,0.1)(0.5,0.1)
\qbezier[60](2,0)(2,-0.1)(2.5,-0.1)
\qbezier[60](3,0)(3,-0.1)(2.5,-0.1)
\qbezier[20](2,0)(2,0.1)(2.5,0.1)
\qbezier[20](3,0)(3,0.1)(2.5,0.1)
\qbezier[60](4,0)(4,-0.1)(4.5,-0.1)
\qbezier[60](5,0)(5,-0.1)(4.5,-0.1)
\qbezier[20](4,0)(4,0.1)(4.5,0.1)
\qbezier[20](5,0)(5,0.1)(4.5,0.1)
\qbezier[100](2.5,0.91)(2.6,0.91)(2.6,0)
\qbezier[100](2.5,-0.91)(2.6,-0.91)(2.6,0)
\qbezier[30](2.5,0.91)(2.4,0.91)(2.4,0)
\qbezier[30](2.5,-0.91)(2.4,-0.91)(2.4,0)
\put(-0.35,0){$\zeta_1$}
\put(0.8,0.5){$\zeta_2$}
\put(2.1,-0.35){$\zeta_3$}
\put(3.9,0.5){$\zeta_4$}
\put(5.1,0){$\zeta_5$}
\put(2.65,0.7){$\sigma$}
\end{picture}

\caption{Dehn twists for the genus 2 pencil on K3 \label{Fig:genus2pencil}}
\end{figure}
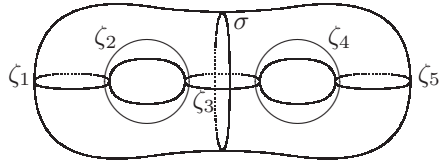
\end{center}

There is also a distinguished pencil of quadric-quadric intersections on the $2g$-fold $Q_2 = \{\sum \mu_j z_j^2=0\}$, which is a Lefschetz pencil by Lemma \ref{Ex:genus2}, again with $(2g+1)(2g+2)$ singular fibres.  Let $\scrW \rightarrow S^2$ be the Lefschetz fibration with fibre $Q_0 \cap Q_1$ given by blowing up $Q_2$ at the base of this pencil  (if $g=2$ this is a blow-up along a genus 17 curve). 

\begin{Lemma} \label{Lem:Monod}
$\scrW$ is defined by $\rho\,\circ\,\eta: \pi_1(S^2 \backslash \{\mathrm{Crit}\}) \rightarrow \pi_0\Symp(Q_0 \cap Q_1).$\end{Lemma}

This is true essentially tautologically, by definition of $\rho$.    There are $(2g+1)(2g+2)$ singular fibres in $\scrW$, defining a collection of that many Lagrangian $(2g-1)$-spheres in $Q_0 \cap Q_1$. These actually only give rise to $(2g+1)$ distinct Hamiltonian isotopy classes $\{V_j\} \subset Q_0 \cap Q_1$; when $g=2$ these are associated to the 5 distinct vanishing cycles $\{\zeta_j\}$ in the pencil on the sextic $K3$, cf. Figure \ref{Fig:genus2pencil}.

\begin{Lemma}\label{Lem:SplitGenQuadricIntersection}
The Lagrangian spheres $\{V_j\}$ split-generate $D^{\pi}\scrF(Q_0 \cap Q_1;0)$.
\end{Lemma}

\begin{proof}
We have a basis $H,E$ comprising the hyperplane and exceptional divisor for $H^2(\scrW)$, dual to a basis $L,-R$ of a line and ruling curve for $H_2(\scrW)$.   Working with the obvious integrable almost complex structure, rational curves in $\scrW$ must meet the fibre non-negatively, which means that curves only exist in classes $dL-rR$ with $r \leq 2d$; on the other hand, curves in classes with a positive coefficient of $R$ meet $E$ negatively, hence are entirely contained in that divisor.  The cycle class $\mathcal{C}(w)$ of $w:\scrW\rightarrow S^2$ counts curves in homotopy classes $\beta$ for which
\[
vdim_{\bC}(\mathcal{M}(\beta)) = c_1(T^{vt}(\scrW))[\beta] + dim_{\bC}(Fibre) = c_1(\scrW)[\beta]+1
\]
is at most the dimension of the fibre, here $(2g-1)$; that forces $c_1(W)[\beta] \leq 2g-2$. But $c_1(\scrW) = 2gH-E$, and together with the constraint $r \leq 2d$, it follows that only curves in classes which are multiples of $R$ can count. Since we only count sections, that means we are actually interested in $R$-curves. These are automatically regular, being spheres whose normal bundle is a sum of bundles of Chern number $\geq -1$, and the regular moduli space forms a copy of the base locus of the pencil. But this is exactly a cycle representing the first Chern class of the fibre, hence Proposition \ref{Prop:Split-Generate} applies.
\end{proof}

\begin{Corollary} \label{Cor:QHIntersectQuadrics}
There is a ring isomorphism $HH^*(\scrF(Q_0 \cap Q_1);0)) \cong H^*(\Sigma_g)$.
\end{Corollary}

\begin{proof}
Corollary \ref{Cor:QuiltsQH} implies that $HH^*(\scrF(Q_0 \cap Q_1;0)) \cong QH^*(Q_0 \cap Q_1;0)$.  The odd-dimensional cohomology  $H^{2g-1}(Q_0 \cap Q_1)$ is isomorphic to $H^1(\Sigma_g)$, is generated by Lagrangian spheres of minimal Maslov number $>2$, and therefore lies entirely in the zero eigenspace for the action of $\ast c_1$. The quantum cohomology ring of $Q_0 \cap Q_1$ was determined by  Beauville \cite{Beauville}: $QH^*(Q_0 \cap Q_1)$ is generated by the primitive middle-degree cohomology and a class $h$ subject to
\[
h^{\ast 2g} = 16 h\ast h; \ h\ast (H^{prim}) =0; \ \alpha \ast \beta = \delta(\alpha,\beta) (h^{2g-1}/4-4h^{2g-3})
\]
with $\delta$ the intersection pairing. This implies that $QH^*(Q_0\cap Q_1;0)$ has rank  $2g+2$; checking the ring structure is easy algebra (we did this for $g=2$ in Lemma \ref{Lem:QHmodulispace}).  
\end{proof}

\subsection{Idempotents for functors\label{Section:AlgebraicInterlude}}

We require a functoriality property of split-closures of $A_{\infty}$-categories.  We begin by recalling the construction of a split-closure.  An \emph{idempotent up to homotopy} for an object $Y\in\scrA$ is an $A_{\infty}$-functor from $\bK$ to $\scrA$ taking the unit of the field to $Y$.  Explicitly, this amounts to giving a collection of elements
 \[
 \wp^d \in hom_{\scrA}^{1-d}(Y,Y); \quad d\geq 1
 \]
 satisfying the equations
 \begin{equation} \label{Eq:idempotent-up-to-homotopy}
 \sum_r \sum_{s_1,\ldots, s_r} \mu_{\scrA}^r(\wp^{s_1},\ldots,\wp^{s_r}) = \begin{cases} \wp^{d-1} & \textrm{if $d$ is even}, \\ 0 & \textrm{if $d$ is odd} \end{cases}
 \end{equation}
summing over partitions $s_1+\cdots+s_r=d$.  If $p\in Hom_{H(\scrA)}(Y,Y)$ is an idempotent endomorphism in the cohomological category\footnote{Since our categories are $\bZ_2$-graded, such idempotents are necessarily of \emph{even} degree.}, there is always an idempotent up to homotopy $\wp$ for which $[\wp^1]=p$ \cite[Lemma 4.2]{FCPLT}.  Any $A_{\infty}$-category has a split-closure $\Pi\scrA$, meaning a fully faithful functor $\scrA \rightarrow \Pi\scrA$ with the property that in the larger category all \emph{asbtract images} of idempotents up to homotopy are quasi-represented.  The abstract image of $\wp$ (an idempotent with target $Y \in \scrA$) is a certain $A_{\infty}$-module $\scrZ$ in $mod$-$\scrA$, with underlying vector spaces
\[
X \mapsto hom_{\scrA}(X,Y)[q] \qquad \text{where $q$ is a formal variable of degree $-1$}
\]
and with operations beginning with
\[
\mu_{\scrZ}^1(b(q)) = \sum _r \sum_{s_2,\ldots, s_r} \delta_q^{s_2 + \cdots + s_r} \mu_{\scrA}^r (\wp^{s_r},\ldots, \wp^{s_2}, b(q)) + \Delta b(q).
\]
Here $\delta_q$ is normalised formal differentiation, $q^k \mapsto q^{k-1}$ and $q^0 \mapsto 0$, whilst $\Delta$ denotes antisymmetrization followed by $\delta_q$ (we will not write out the higher order operations). By \cite[Lemma 4.5]{FCPLT}, one has the critical property that:
\begin{equation} \label{Eq:AbstractIdempotentModules}
Hom_{H(mod-\scrA)}(\scrZ,\scrZ') \ \cong \ e\cdot Hom_{H(\scrA)}(Y,Y')\cdot e'
\end{equation}
whenever $\scrZ$ and $\scrZ'$ are the abstract images of idempotents up to homotopy with targets $Y,Y'$ corresponding to cohomological idempotents $e\in Hom_{H(\scrA)}(Y,Y)$ and $e' \in Hom_{H(\scrA)}(Y', Y')$.

\begin{Lemma} \label{Lem:SplitIdempotentFunctors}
Let $\scrA, \scrB$ and $\scrC$ be non-unital $A_{\infty}$-categories and suppose there is an $A_{\infty}$-functor $\Phi: \scrA \rightarrow nu$-$fun(\scrB,\scrC)$. To each idempotent up to homotopy $\wp$ for $Y\in\scrA$ one can canonically associate an element of $nu$-$fun(\scrB,\Pi\scrC)$.
\end{Lemma}

\begin{proof}
$\wp$ is defined by a functor $\bK \rightarrow \scrA$ so there is a composite functor $\bK \rightarrow nu$-$fun(\scrB,\scrC)$ with target $\Phi_Y$.  This is defined by a sequence of elements
\[
\wp_{\Phi_Y}^d \in hom_{nufun(\scrB,\scrC)}^{1-d}(\Phi_Y,\Phi_Y).
\]
The degree zero term in this morphism of functors yields elements
\[
\wp_{\Phi_Y(B)} \in hom_{\scrC}^{1-d}(\Phi_Y(B),\Phi_Y(B))
\]
which satisfy the conditions of Equation \ref{Eq:idempotent-up-to-homotopy}.  Thus $\wp_{\Phi_Y(B)}$ is an idempotent up to homotopy for $\Phi_Y(B) \in \scrC$, for each $B \in \scrB$.  The functor $\Phi_Y$ is defined by a collection of natural maps
\[
hom_{\scrB}(B_{d-1},B_d) \otimes \cdots \otimes hom_{\scrB}(B_1,B_2) \rightarrow hom_{\scrC}(\Phi_Y(B_1), \Phi_Y(B_d))[1-d]
\]
We claim these induce maps
\[
hom_{\scrB}(B_{d-1},B_d) \otimes \cdots \otimes hom_{\scrB}(B_1,B_2) \rightarrow hom_{\Pi\scrC}(\wp_{\Phi_Y(B_1)}, \wp_{\Phi_Y(B_2)})[1-d]
\]
where $\wp_{\bullet}$ denotes the abstract image of the given idempotent.  To see this, since $\Pi\scrC$ is by definition a certain category of modules, thus of functors from $\scrC$ to chain complexes, an element of $hom_{\Pi\scrC}(\scrF,\scrG)$ is explicitly defined by the collection of maps
\[
hom_{\scrC}(C_{d-1},C_d) \otimes \cdots \otimes hom_{\scrC}(C_1,C_2) \rightarrow hom_{Ch}(\scrF(C_1), \scrG(C_d)).
\]
Taking the functors $\scrF = \Phi_Y(B_1)$ and $\scrG=\Phi_Y(B_d)$, we require maps
\begin{multline*}
hom_{\scrC}(C_{d-1},C_d) \otimes \cdots \otimes hom_{\scrC}(C_1,C_2) \longrightarrow  \\ \quad hom_{Ch}\left(hom_{\scrC}(C_1,\Phi_Y(B))[q], hom_{\scrC}(C_d,\Phi_Y(B))[q] \right) \end{multline*}
These are canonically obtained by using composition $\mu^d_{\scrC}$ in the category $\scrC$.  It is a straightforward check to see that the construction is compatible with varying $B$, hence  yields a functor from $\scrB$ to $\Pi\scrC$ as required.
\end{proof}

\begin{Corollary}
There is a natural functor $\Pi\scrA \rightarrow nu$-$fun(\scrB,\Pi\scrC)$.
\end{Corollary}

We will not  need this stronger result and omit the proof.

\subsection{An illustrative example}\label{Section:BlowupC2}

Let $X = Bl_{0}(\bC^2)$ denote the blow-up of the complex plane at the origin.  We equip this with a symplectic form giving the exceptional divisor area $c> 0$. Let $Cl_k$ denote the Clifford algebra associated to a non-degenerate quadratic form on $\bC^k$, of total dimension $2^k$.

\begin{Lemma} \label{Lem:ToricFibre}
$X$ contains a Lagrangian torus $T$ which, equipped with the \emph{Spin} structure which is non-trivial on both factors, has $HF(T,T) \cong Cl_2 \neq 0$.
\end{Lemma}

\begin{proof}
There is a natural Hamiltonian torus action on $\bC^2$, which is inherited by the blow-up since we have blown up a toric fixed point.  The relevant moment polytope is depicted in Figure \ref{Fig:MomentPolytope}; in this figure the fibre lying over the black dot, placed  symmetrically at the ``internal corner" of the polytope, defines the Lagrangian torus $T$.
\begin{center}
\begin{figure}[ht]
\includegraphics[scale=0.4]{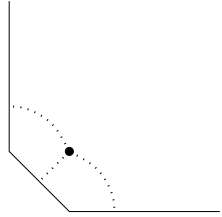}
\caption{A Lagrangian torus in $\mathcal{O}(-1)$\label{Fig:MomentPolytope}} 
\end{figure}
\end{center}
 All the holomorphic disks with boundaries on the fibres of the moment map can be computed explicitly, cf. \cite{Cho, Auroux:toric}; in this case the torus bounds 3 families of Maslov index 2 disks, all having the same area, so $T$ is in fact monotone.  The three dotted lines emanating from the black dot in Figure \ref{Fig:MomentPolytope} are schematic images of the three Maslov index 2 disks through a point; the disks actually fibre over arcs in the affine structure on the moment polytope determined by complex rather than symplectic geometry, but the disks do intersect the boundary facets as indicated.    We equip $T$ with the $Spin$ structure which is bounding on both factors. The contributions of these disks to the Floer differential and product can be determined explicitly exactly as in \emph{op. cit.}; the Lagrangian moment map fibres with non-zero Floer cohomology for some $Spin$ structure are precisely those lying over critical values of the \emph{superpotential}, which in this example is 
 \[
W(z_1,z_2) \ = \ z_1+z_2+e^c z_1z_2.
\] The constituent monomials are indexed by the toric boundary strata of the moment polytope, cf. \cite[Proposition 4.3]{Auroux:toric}, with the logarithm of the corresponding coefficient giving the distance of the defining hyperplane from the origin in $\bR^2$. There is a unique critical point $(-e^{-c},-e^{-c})$ lying over the fibre of the moment map given by the black dot $(c,c) = (-log |z_1|,-log |z_2|)$.  Since the critical point is non-degenerate, this torus has $HF(T^2, T^2) \cong Cl_2$.  
\end{proof}

\begin{Remark} \label{Remark:LocalSlag}
The torus $T$ is actually special Lagrangian with respect to the pullback of the holomorphic volume form $dz_1/z_1 \wedge dz_2/z_2$ from $\bC^2$.  The pullback has first order poles along the proper transforms of the axes and the exceptional divisor, hence any Maslov index 2 disk must meet this locus transversely once, as in Figure \ref{Fig:MomentPolytope}.
\end{Remark}

For another viewpoint on the Lagrangian torus $T \subset \mathcal{O}(-1)$, recall the construction of the symplectic blow-up.  We let 
\[
\mathcal{O}(-1) = \left \{ ((z_1,z_2), [w_1:w_2]) \in \bC^2 \times \bP^1 \ \big| \ z_1w_2=w_1z_2\right\}
\]
be the standard embedding of the tautological line bundle and write $\mathcal{O}(-1)_{\delta} = \Phi^{-1}(B(\delta))$ for the preimage of the ball $B(\delta) \subset \bC^2$ under the first projection $\Phi$. The standard symplectic form giving the exceptional sphere area $\pi\chi^2 > 0$ is just $\Omega_{\chi} = \Phi^*\omega_{std} + \chi^2 \pi_{\bP^1}^* \omega_{FS}$, where $\omega_{FS}$ is the appropriately normalised Fubini-Study form on the projective line.  The basic observation, see e.g. \cite[Lemma 7.11]{McD-S}, is that
\begin{equation} \label{Eq:blowup}
(\mathcal{O}(-1)_{\delta} \backslash \mathcal{O}(-1)_0, \Omega_{\chi}) \ \cong \ (B(\sqrt{\chi^2 + \delta^2} \backslash B(\chi), \omega_{std}) \subset \bC^2.
\end{equation}
Thus there is a symplectic form on the blow-up with area $c$ on the exceptional sphere whenever a ball of radius strictly larger than $c/\sqrt{\pi}$ embeds in the original space. Now take the equator $S^1 \subset \bP^1 = E$ inside the exceptional divisor.  Consider the union of the circles inside fibres lying over the equator which bound disks of area $c$; from the above discussion, such a tubular neighbourhood embeds in the blow-up (in general this will be true if $c$ is sufficiently small, but a ball of any size fits in $\bC^2$).  The torus $T$ swept out by these fibres of constant radius is Lagrangian, since via Equation \ref{Eq:blowup} it comes from a Clifford-type torus inside a sphere in $\bC^2$; it is invariant under the usual $S^1 \times S^1$ action, and in fact co-incides with  the toric fibre discussed in Lemma \ref{Lem:ToricFibre}.   

To put the reformulation in a more general context, note that we have correspondences
\begin{equation}\label{Eqn:correspondences}
\Gamma_1 = S^1_{eq} \subset \{pt\} \times \bP^1 \cong E; \qquad \Gamma_2 \subset E \times X
\end{equation}
with $\Gamma_2 = \partial \nu_E$ the 3-sphere which is the boundary of a tubular neighbourhood $\nu_E$ of the exceptional curve, viewed not as a submanifold in $X$ but in the product $E\times X$.  This 3-sphere is the graph of a Hopf map from $\bC^2\backslash \{0\} \rightarrow \bP^1$.  We claim that for suitable symplectic forms, these correspondences are monotone Lagrangian (compare to Remark \ref{Rem:ScaleSymplecticForm} to see why this is delicate).  Let $(\bullet,\omega)^-$ denote the symplectic manifold $(\bullet, -\omega)$.  Recall that a co-isotropic submanifold $W \subset (M,\omega_M)$ with integrable characteristic distribution and reduced space $(N, \omega_N)$ defines a Lagrangian graph in $(N,\omega_N)^- \times (M,\omega_M)$.

\begin{Lemma} \label{Lem:CorrsAreLagrangian}
Equip $E=\bP^1$ with the symplectic form $2\omega_{FS}$ of area $2\pi$, and equip $X=\mathcal{O}(-1)$ with the symplectic form $\Omega_{1}$ giving the zero-section area $\pi$.  Define  $\partial \nu_E$ as the boundary of the sphere of radius $\sqrt{2}$ in $\bC^2$, viewed via \eqref{Eq:blowup} as the boundary of the unit disk subbundle $\mathcal{O}(-1)_{1} \subset \mathcal{O}(-1)$. Then the correspondence  $\Gamma_2 \subset E^- \times X$ associated to $\nu_E$ is a monotone Lagrangian submanifold.
\end{Lemma}

\begin{proof}
The key point is that the symplectic form on $\bP^1$ obtained by taking the reduction of the coisotropic submanifold $\nu_E \subset (\mathcal{O}(-1), \Omega_1)$ has area $2\pi$, even though the zero-section in that line bundle has area $\pi$. This is clear from the description of \eqref{Eq:blowup}, or from a toric picture, Figure \ref{Fig:MomentPolytope} (the line parallel to the boundary facet defining the exceptional curve and passing through the distinguished point $\bullet$ defines a co-isotropic 3-sphere, and the reduced space is obviously a $\bP^1$ of area determined by the radius of the corresponding ball in $\bC^2$ before blowing up).   Therefore, the co-isotropic $\Gamma_2$ defines a Lagrangian graph inside $(\bP^1, -2\omega_{FS}) \times (\mathcal{O}(-1), \Omega_1)$. On the other hand, the first Chern class of this product is the class $(-2,1) \in H^2(\bP^1) \oplus H^2(\mathcal{O}(-1)) \cong \bZ^2$, so $\Gamma_2$ is indeed monotone.
\end{proof}

The composite of the  two correspondences
\[
\Gamma_1 \circ \Gamma_2 \subset \{pt\} \times X
\]
is exactly the Lagrangian torus $T$. This also partially explains the particular choice of $Spin$-structure on $T$; it should be equipped with a $Spin$-structure which in the $\Gamma_2$-direction is inherited from the ambient correspondence $S^3$, hence on that circle factor we must pick the \emph{bounding} structure.  

\begin{Lemma}
The Lagrangian torus $T$ of Lemma \ref{Lem:ToricFibre} has a non-trivial idempotent.
\end{Lemma}

\begin{proof}
 $Cl_2$ is a matrix algebra, hence splits into two  idempotents each of which has endomorphism algebra of rank 1.  More precisely, if $Cl_2$ is generated by $x$ and $y$ of degree $1$ subject to $xy+yx=1$ with $x^2=0=y^2$, then $xy=e^+$ and $yx=e^-$ are idempotents; the morphisms between the corresponding summands of $T$ are $HF(e^+, e^-) \cong \bK$.  
\end{proof}

\begin{Remark}
The proof shows there are two ``indistinguishable" idempotents in $HF(T,T)$, and one could use either in the construction of the functor below.  In fact, 
\[
e^+ \cong e^-[1] \quad \textrm{and} \quad e^- \cong e^+[1]
\] (which is compatible with the fact that the double shift is isomorphic to the identity in the $\bZ_2$-graded Fukaya category). The same binary choice appears on the mirror side, when one considers Kn\"orrer periodicity for Orlov's derived categories of singularities.  In principle, this choice fibrewise in a family of Clifford-like 2-tori could lead to an obstruction to the existence of a global idempotent, but this obstruction will vanish when we later blow up the base-locus of a pencil of hyperplanes.
\end{Remark}

\begin{Remark} \label{Rem:CorrNonZero}
Non-triviality of $HF(T,T)$ implies that the Lagrangian correspondences $\Gamma_i$ are non-trivial elements of their Fukaya categories.  Note that the individual correspondences $\Gamma_1 \cong S^1$ and $\Gamma_2\cong S^3$ cannot split for reasons of degree: in both cases only the identity lives in $HF^{ev}$. \end{Remark}

The composite correspondence $\Gamma_1 \circ \Gamma_2$ has an idempotent summand which defines a functor
\begin{equation} \label{Eqn:Embed}
\Phi(e^+): \scrF(\{pt\}) \rightarrow Tw^{\pi}\scrF(X)
\end{equation}
which takes the unique Lagrangian submanifold of the base locus $\{pt\}$ to an object on the RHS with the same Floer cohomology; i.e. the summand of the functor is fully faithful.  To obtain a slightly less trivial example, one can consider blowing up a collection of points in $\bC^2$ by equal amounts.  Taking a common size for the blow-ups implies that the resulting space contains Lagrangian 2-spheres, which arise from the symplectic cut of trivial Lagrangian cylinders in $\bC^2$.

\begin{Example} \label{Ex:SpinModel}
Let $p=(0,0) \in \C^2$ and consider the arc $\gamma \subset \bR^+ \subset \bC = \{w=0\} \subset \bC^2$, with co-ordinates $(z,w)$ on $\bC^2$.  The blow-up at $p$ is effected symplectically by choosing a ball with boundary 
\[
S^3_{\varepsilon} \ = \ \{|z|^2+|w|^2 = \varepsilon^2\}
\]
and quotienting out Hopf circles.  Suppose now we consider the circles $\{|w|=\varepsilon\}$ lying over the arc $\gamma \subset \{w=0\}$.  The union of these circles is a Lagrangian tube in $\bC^2$ which meets $S^3_{\varepsilon}$ only over the end-point of $\gamma$ at the origin, where the intersection is precisely the circle $\{z=0, |w|=\varepsilon\}$ which is collapsed in forming the blow-up, cf. Figure \ref{Fig:SpinTube}.  
\begin{center}
\begin{figure}[ht]
\includegraphics[scale=0.4]{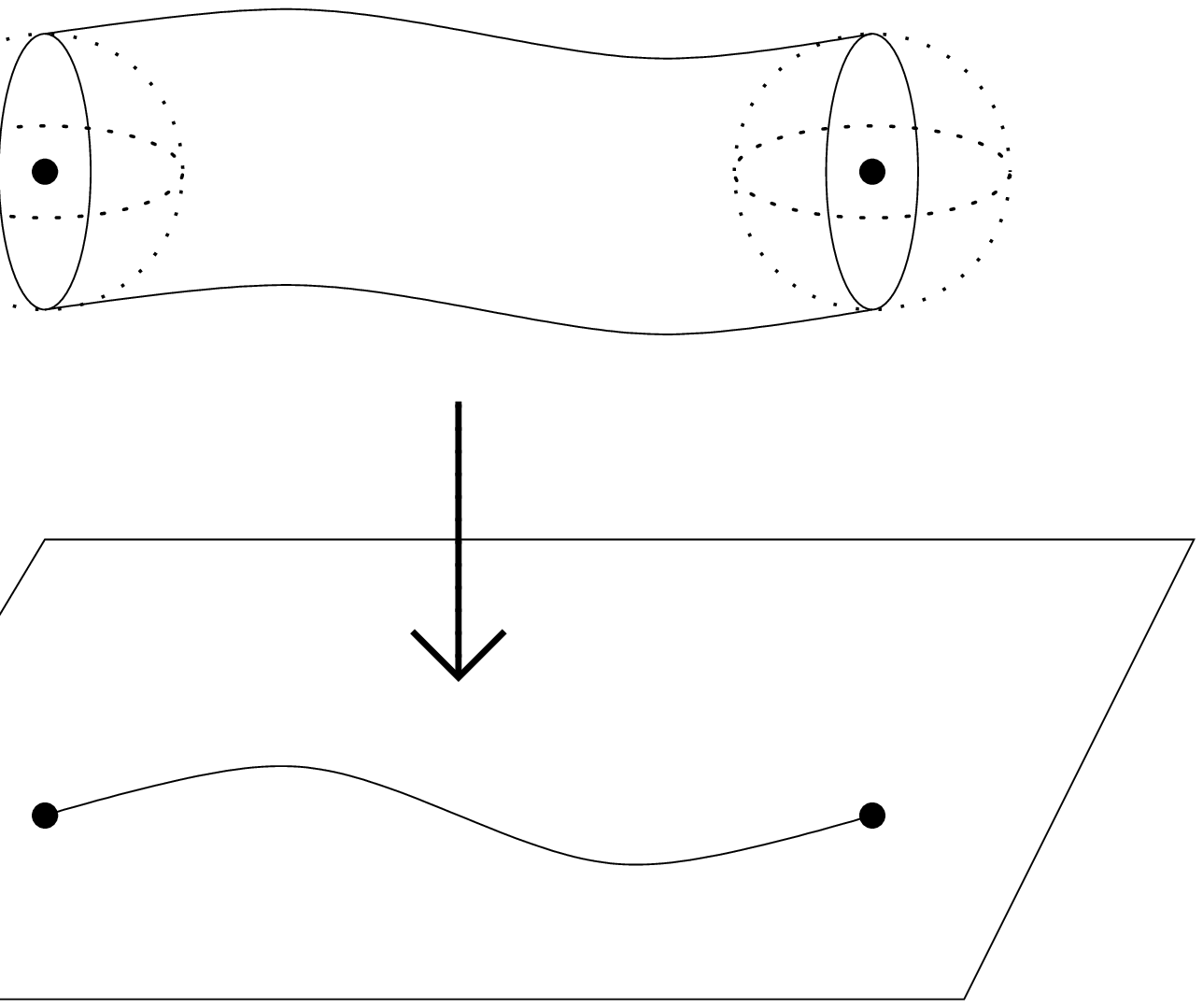} 
\caption{The symplectic cut of a Lagrangian tube in $Bl_{p\amalg q}(\bC^2)$ (the dotted balls are collapsed in performing the blow-up) \label{Fig:SpinTube}}
\end{figure}
\end{center}
The Lagrangian tube therefore defines a smooth Lagrangian disk in $Bl_p(\bC^2)$.  If we now blow up $\bC^2$ at $p=(0,0)$ and $q=(2,0)$, with equal weights $c=\pi\varepsilon^2$ at the two-points, there is a Lagrangian sphere $L \subset Bl_{p\amalg q}\bC^2$ in the homology class $\pm(E_1-E_2)$. This arises from a pair of such disks, and is naturally a matching sphere associated to a path  $\gamma \subset \bC=\{w=0\}$ between the points $p=(0,0)$ and $q=(2,0)$.
\end{Example}

\begin{Lemma} \label{Lem:Clean}
The torus $T \subset Bl_0(\bC^2)$ meets the Lagrangian disk $\Delta$ which is the symplectic cut of $\bR_{\geq 0} \times \{|w| = \varepsilon\}$ cleanly in a circle.  The Floer cohomology $HF(T^2,\Delta) \cong H^*(S^1)$.
\end{Lemma}

\begin{proof}
The disk $\Delta$ fibres over the horizontal arc emanating from the vertex on the $y$-axis in the moment image of Figure \ref{Fig:MomentPolytope}.  The disk and the torus meet cleanly along the circle where $z=\varepsilon \in \bR_+$. We claim the Morse-Bott-Floer differential vanishes: it is enough to see that $HF(T^2, \Delta) \neq 0$. To prove this, we introduce 
 a Lefschetz fibration viewpoint on these constructions. The trivial fibration $\bC^2 \rightarrow \bC$ which is projection to the $z$-co-ordinate induces a Lefschetz fibration on $Bl_0(\bC^2)$ with a unique critical fibre, which contains the exceptional sphere. The vanishing cycle is homotopically trivial, and the monodromy is Hamiltonian isotopic to the identity.  There is a unique radius of circle centred at the origin in the base, determined by the size $c$ of the blow-up parameter, for which the Lagrangian torus which is swept out by the vanishing cycle in each fibre (defined relative to parallel transport along a radial arc)  is actually monotone, and this is our torus $T$.  The disk $\Delta$ of Lemma \ref{Lem:Clean} is the unique Lefschetz thimble, and the non-vanishing of $HF(T,\Delta)$ is an instance of the general fact that  a basis of vanishing thimbles generates a category containing all weakly unobstructed compact Lagrangians.  This is a theorem of \cite{FCPLT} in the case of exact Lefschetz fibrations; Proposition \ref{Prop:generate} explains why the argument applies in the current setting. \end{proof}

 From Lemma \ref{Lem:Clean} one sees that $HF(e^{\pm}, \Delta) \cong \bK$.  In the Fukaya category $\scrV$ of the Lefschetz fibration, the individual summands $e^{\pm}$ are isomorphic  up to shifts to the disk $\Delta$.   We record one elementary point from this reformulation.  Suppose one has an $A_2$-chain of Lagrangian $S^0$'s, i.e. two zero-spheres $V_0$ and $V_1$ sharing a point, say $V_0 = \{(-1,0), (0,0)\}$ and $V_1 = \{(0,0), (1,0)\} \subset \bC^2$.  We fix real arcs $\Delta_i \subset \bC\times \{0\}$ bounding the $V_i$ and meeting at the origin, with respect to which we obtain spun Lagrangian 2-spheres $\hat{V}_i \subset Bl_{3 \ \textrm{points}} (\bC^2)$, and we consider the Floer product
\begin{equation} \label{Eqn:ProductOnSpunSpheres}
HF(\hat{V}_0, \hat{V}_1) \otimes HF(T, \hat{V}_0) \longrightarrow HF(T, \hat{V}_1).
\end{equation}
The relevant Lagrangians fibre, in the Lefschetz fibration described above -- locally restricting to a subspace $Bl_0(\bC^2)$ of the 3-point blow-up --  over a circle centred at the origin and the positive and negative real half-lines; the $\hat{V}_i$ are here viewed as compactifications of Lagrangian disks fibring over these half-lines.  In this local picture, the triangle product must be non-vanishing, just because the Lagrangian disks are actually isomorphic in the Fukaya category of the Lefschetz fibration, both being the unique Lefschetz thimble.  On the other hand, by the maximum principle all the holomorphic triangles are local, hence this non-vanishing actually applies to Equation \ref{Eqn:ProductOnSpunSpheres}.

\subsection{Idempotents and blow-ups\label{Section:Idempotent}}

Let $(X,\omega)$ be a projective variety, $\{H_t\}_{t\in\bP^1}$ a Lefschetz pencil of hypersurface sections with smooth base locus $B = H_0\cap H_1 \subset X$.  For simplicity, we suppose $\pi_1(B) = 0$.  We also fix an additional ``reference" hyperplane section $H_{ref}$ linearly equivalent to the $H_i$ but not in the pencil defining $B$, so there is an associated net of hyperplanes $\langle H_0, H_1, H_{ref}\rangle$. We denote by $b: Y\rightarrow X$ the blow-up $Y=Bl_B(X)$ along the base locus, and by $E\subset Y$ the exceptional divisor; note that $E \cong B\times \bP^1$ canonically, since the normal bundle $\nu_{B/X} \cong \mathcal{L} \otimes \bC^2$ for the line bundle $\mathcal{L} = \nu_{H_i/X}$. We assume:
\begin{equation}
\begin{aligned} \label{Equation:Hypothesis}
\bullet \ &  X, B \ \textrm{and}  \ Y \ \textrm{are all monotone (Fano)};  \ c_1(X) = (d+2)H \ \textrm{with} \ d\geq 2; \\
\bullet \ & c_1(\mathcal{L}|_B) \ \textrm{is proportional to} \  c_1(B);  \ B \ \textrm{has even minimal Chern number} \  \geq 2; \\
\bullet \ & B  \ \textrm{contains a Lagrangian sphere}. 
\end{aligned} \end{equation} 
The first assumption determines the size $\lambda$ of the blow-up parameter, i.e. the area of the ruling curve $\bP^1 \subset E$.  We actually strengthen the first assumption, and demand that
\begin{equation} \label{Equation:Hypothesis2}
\bullet \  c_1(Y) - \textrm{PD}[E] = b^*c_1(X) - 2\textrm{PD}[E] \ \textrm{is ample, so contains K\"ahler forms}; \quad 
\end{equation}
note that this implies there are K\"ahler forms on $Y$ which give the ruling curves $\bP^1 \subset E$ \emph{twice} the area they have for the monotone form. Recall our  notation  $(\bullet,\omega)^-$ for  the symplectic manifold $(\bullet,-\omega)$.   As in \eqref{Eqn:correspondences}, we consider the two correspondences
\[
\Gamma_R = B\times S^1_{eq} \ \subset \ B^- \times (B \times \bP^1)
\]
where $S^1_{eq} \subset \bP^1$ denotes the equator and 
\[
\Gamma_E = U_{\nu_E} \ \subset \ E^-\times Y
\]
where $\nu_E \rightarrow E$ is the normal bundle to $E$ in $Y$ and $U_{\nu_E}$ denotes a circle sub-bundle.  The correspondences are equipped with bounding $Spin$ structures.  As in Lemma \ref{Lem:CorrsAreLagrangian}, by choosing the radius of this subbundle and the symplectic forms appropriately:

\begin{Lemma} \label{Lem:Monotone}
Under hypothesis (\ref{Equation:Hypothesis2}), the correspondences $\Gamma_E$ and $\Gamma_R$ can be chosen to be embedded monotone Lagrangian submanifolds, with embedded monotone composition  $\Gamma = \Gamma_R \circ \Gamma_E \subset B^- \times Y$.
\end{Lemma}

\begin{proof}
Let $H \subset Z$ be a codimension 2 symplectic submanifold of a symplectic manifold $(Z,\omega_Z)$, with the symplectic form $\omega_H = (\omega_Z)|_H$.  Equip the normal bundle $\nu \rightarrow H$ with a Hermitian metric. The symplectic form on a neighbourhood of the zero-section in the total space of  $\nu$  is 
 \begin{equation} \label{Eqn:SymplecticDiskBundle}
\Omega_{\nu}  = \pi^*\omega_H + \frac{1}{2} d(r^2 \alpha)
\end{equation}
where $r$ is the radial co-ordinate and $\alpha$ is a connection form for the bundle satisfying $d\alpha = -\pi^*\tau$, with $\tau$ a curvature 2-form on $H$ with  $[\tau] = [c_1(\nu)]$ as obtained from Chern-Weil theory.  Note that the form $\Omega_{\nu}$ is a standard area form on the fibres of the disk bundle.
 The symplectic form induced on the reduced space $H_{red}$ of the co-isotropic submanifold which is the boundary of the radius $r$ disk bundle in $\nu$ differs from $\omega_H$ by subtracting $r^2 \tau / 2$.  In particular, when it is possible to take  $r=\sqrt{2}$, the reduced form lies in the cohomology class $[\omega_H - c_1(\nu)]$, compare to Lemma \ref{Lem:CorrsAreLagrangian}. 

The symplectic normal bundle to $B \times \bP^1 \cong E \subset Y$ is the tensor product $\pi_B^* \mathcal{L} \otimes \pi_{\bP^1}^* \mathcal{O}(-1)$, with first Chern class  $(c_1(\mathcal{L}),-1) \in H^2(B) \oplus H^2(\bP^1)$.  We fix a tensor product Hermitian metric and connexion, taking $\tau$ to be a product 2-form.   The ruling $\bP^1 \subset E$ has area $1$ for a monotone K\"ahler form $\omega_{mon}$ on $Y$ which is cohomologous to the first Chern class.  By hypothesis, we have K\"ahler forms $\omega$  on $Y$ for which the area of $\bP^1$ is $> 2$.  After normalising the K\"ahler form near $E$, one can blow down to obtain a symplectic form on $X$ which contains a radius $\sqrt{2}+\delta$ disk bundle around the base locus $B$; compare to \cite[Lemma 7.11]{McD-S} or \cite[Theorem 2.3]{Gompf}. If we now blow up again by amount $1$, the exceptional divisor $E\subset (Y, \omega_{mon})$ naturally lies inside a standard symplectic disk bundle whose fibres have area $1+\delta$.  This is sufficiently large to contain an embedded circle bundle $U_{\nu_E}$ in which the circles bound fibre disks of area $1$.  (Blowing down and up in this way effects a symplectic inflation.  Indeed, when Gompf puts symplectic structures on Lefschetz pencils in general dimension, he precisely patches in a \emph{large} symplectic disk bundle over the base locus \cite[p.279, and his embedding $\varphi$ onto a disk bundle of radius $R$]{Gompf}.)

From \eqref{Eqn:SymplecticDiskBundle}, the reduced space of the co-isotropic $U_{\nu_E}$  has symplectic form $(\omega_Y)|_E - \tau$ in cohomology class
\begin{equation} \label{Eqn:ReducedClass}
[\omega_Y]|_E - [(c_1(\mathcal{L}), -1)].
\end{equation}
We can view $Y\subset X\times \bP^1$ as a divisor in class $(c_1(\mathcal{L}),1)$. It follows that the monotone form $\omega_Y$ lies in cohomology class $(c_1(X) - c_1(\mathcal{L}), 1) \in image \left( H^2(X) \oplus H^2(\bP^1) \rightarrow H^2(Y) \right)$ (the latter map being an isomorphism by the Lefschetz theorem in our examples).  Combining this with \eqref{Eqn:ReducedClass} and noting that $c_1(B) = c_1(X)|_B -2c_1(\mathcal{L})$, by adjunction, shows that the reduced symplectic form on $B\times \bP^1$ coming from the co-isotropic $U_{\nu_E}$ is monotone.     

We have arranged that the coisotropic defining the Lagrangian correspondence $\Gamma_E$ lives inside the blow-up $Bl_B(\nu_{B/X}) = Tot_{\mathcal{O}(-1)}(E)$ of a large disk normal bundle to $B$ in $X$. By construction, all Maslov index 2 disks in this open subset have the same area, compare to Lemma \ref{Lem:ToricFibre}.  Since $H_1(B) = 0$, these local disks together with rational curves in the total space span $H_2(E\times Y, \Gamma_E)$, hence are sufficient to detect monotonicity. It follows that $\Gamma_E$ is monotone.  The corresponding statement for $\Gamma_R$ is trivial, taking the monotone forms on $B$ and $B \times \bP^1$; and monotonicity is then inherited by the Lagrangian composition (note this  is well-defined since the symplectic forms on $B\times \bP^1$ arising as the output of $\Gamma_R$ and the input of $\Gamma_E$ do co-incide).
\end{proof}

\begin{Remark} \label{Rem:MonotoneFormFromCut}
We emphasise that there are two different symplectic forms on $B\times \bP^1$ in play;  first, the non-monotone  form arising by restricting a  monotone form on $Y$ to $E\subset Y$; second, the monotone form arising from co-isotropic reduction of the boundary of a suitable disk bundle $\nu_E \subset Y$.  We only use the Fukaya category of the monotone form; as a mild abuse of notation, we will continue to write $E$ for the space $B\times \bP^1$ even though it is not equipped with the symplectic form it inherits from the embedding in the ambient space $Y$.
\end{Remark}

\begin{Remark} \label{Rem:HypothesisHolds}
For the case of interest in this paper, namely $(2,2)$-intersections in $\bP^{2g+1}$ with $g\geq 2$, the ampleness hypothesis (\ref{Equation:Hypothesis2}) holds by Lemma \ref{Lem:AmpleBlowUp}.  When $g=1$, this conditions \emph{fails}; the class $c_1(\bP^3)-2\textrm{PD}[E]$ lies on the boundary of the K\"ahler cone of the blow-up of $\bP^3$ along a complete intersection elliptic curve.  In that case, there is no monotone correspondence, which fits with the fact that the elliptic curve has no balanced Fukaya category and  $\scrF(\Sigma_1)$ is naturally defined over $\Lambda_{\bC}$.
\end{Remark}

\begin{Remark} \label{Rem:CorrespondenceViaCut}
There is another viewpoint on the composite correspondence $\Gamma = \Gamma_R \circ \Gamma_E$ which is often useful.  Again starting from a K\"ahler form on $Y$ which is standard in a large disk bundle over $E$,  we perform a symplectic cut \cite{Lerman} along the boundary of a sub-disk bundle.  Algebro-geometrically, symplectic cutting is effected by deformation to the normal cone \cite[Chapter 5]{Fulton}.   We obtain a space with two components meeting along a divisor. One component is a $\bF_1$-bundle over $B$, given by projective completion of the normal bundle to $E$, and the other $Y'$ is holomorphically isomorphic to $Y$, but equipped with a different K\"ahler form (of smaller volume).  The hypothesis \eqref{Equation:Hypothesis2} implies that one can perform the cut so as to obtain a K\"ahler form which is monotone on the total space of the $\bF_1$-bundle. The associated monotone $\bP^2$-bundle over $B$ is the completion $\bP(\mathcal{L} \oplus \mathcal{L} \oplus \mathcal{O})$ of its normal bundle. Being the projectivisation of a direct sum of line bundles, the structure group  reduces to a maximal torus $T^2 \leq PGL_3(\bC)$, which implies that there is a (cohomologous) linear K\"ahler form on the $\bP^2$-bundle for which parallel transport preserves the fibrewise Lagrangian torus which is the monotone torus in $\bP^2$.  Globally, this gives rise to a Lagrangian submanifold
\begin{equation} \label{Eqn:Cut}
\Gamma_{cut} = T^2 \tilde{\times} B \ \subset \ \bP(\mathcal{L}^{\otimes 2}\oplus\mathcal{O}) \times B
\end{equation}
which is a monotone 2-torus bundle over the diagonal $\Delta_B$.  If one now blows back up  and moreover smooths the normal crossing locus of the symplectic cut, in each case deforming the symplectic form in a region disjoint from $\Gamma_{cut}$,   one obtains a Lagrangian correspondence $\Gamma_{cut} \subset B\times Y$, which is exactly $\Gamma$.
\end{Remark}

As remarked above, $\Gamma= \ T^2\widetilde{\times} B$ is a 2-torus bundle, which in general is differentiably non-trivial.  However, given any simply-connected Lagrangian $L\subset B$, the restriction of the correspondence to $L$ \emph{is} trivial (because the symplectic line bundle $\mathcal{L} \rightarrow L$ is flat and hence trivial), and the geometric composition of $L$ with the correspondence is a Lagrangian submanifold diffeomorphic to a product $L\times T^2 \subset Y$.

The quilted Floer theory of Mau, Wehrheim and Woodward, Theorem \ref{Thm:QuiltsGiveFunctors} of Section \ref{Section:Quilts},  associates to a monotone Lagrangian submanifold $\Gamma \subset M^-\times N$ a $\bZ_2$-graded $A_{\infty}$-functor $\Phi_{\Gamma}: \scrF(M) \rightarrow \scrF^{\#}(N)$, where $\scrF^{\#}(\bullet)$ is the $A_{\infty}$-category of generalised Lagrangian correspondences.  Theorem \ref{Thm:QuiltsCompose} implies that the functors $\Phi_{\Gamma_R}\circ\Phi_{\Gamma_E}$ and $\Phi_{\Gamma_R\circ \Gamma_E}$ are quasi-isomorphic. 

\begin{Lemma} \label{Lem:ExtendedFukNotNeeded}
The functor $\Phi_{\Gamma_R\circ\Gamma_E}$ is quasi-isomorphic to a functor with image in the fully faithfully embedded $A_{\infty}$-subcategory $\scrF(Y) \subset \scrF^{\#}(Y)$.
\end{Lemma}

\begin{proof}[Sketch]
 Write $\Gamma=\Gamma_R\circ\Gamma_E \subset B^- \times Y$.  
For every object $L\in \scrF(B)$ the image of $L$ under $\Gamma$ is a smooth Lagrangian submanifold $\hat{L}$ in $Y$. Quilt theory, via Theorem \ref{Thm:QuiltsCompose}, again provides a quasi-isomorphism of functors $\scrF(pt) \rightarrow \scrF^{\#}(Y)$ with images $\hat{L}$ and $\Phi_{\Gamma}(L)$; the quasi-isomorphism of functors implies that the target objects
\[
\hat{L} \ \stackrel{\sim}{\longrightarrow} \ \{pt \stackrel{L}{\longrightarrow} B \stackrel{{\Gamma}}{\longrightarrow} Y\} \, = \, \Phi_{\Gamma}(L)
\]
 are quasi-isomorphic objects of $\scrF^{\#}(Y)$, and in particular $HF(\hat{L}, \Phi_{\Gamma}(L)) \cong HF(\hat{L}, \hat{L})$.  More explicitly, the Floer chain group
\[
CF(\hat{L}, \{pt \stackrel{L}{\rightarrow}B \stackrel{\Gamma}{\rightarrow} Y\}) \ = \ CF(\hat{L} \times \Gamma^-, \Delta \times L)
\] 
where the RHS is computed in $Y \times Y^- \times B$.  These Lagrangian submanifolds meet cleanly, and projection to the first factor identifies the clean intersection locus with $\hat{L}$.  A chain-level representative $e \in CF(\hat{L} \times \Gamma^-, \Delta \times L) = C^*(\hat{L})$ for the identity $1 \in HF(\hat{L},\hat{L})$  provides the desired quasi-isomorphism.  Similarly, 
\[
CF(pt  \stackrel{L_0}{\rightarrow}B \stackrel{\Gamma}{\rightarrow} Y,  \ pt \stackrel{L_1}{\rightarrow}B \stackrel{\Gamma}{\rightarrow} Y) 
\]
is given by chains on the intersection locus $(\Gamma \times \Gamma^-) \cap (L_0^- \times \Delta_Y \times L_1) \subset B^- \times Y \times Y^- \times B$, and this clean intersection is exactly $\widehat{L_0 \cap L_1} = \hat{L}_0 \cap \hat{L}_1$.  Projection to the central two factors $Y\times Y^-$ again identifies the Floer chain group with $CF(\hat{L}_0, \hat{L}_1)$.  The subcategories of $\scrF^{\#}(Y)$ generated by the $\hat{L}$-images and the corresponding generalised Lagrangian branes are quasi-equivalent, and the former lies inside $\scrF(Y)$.
\end{proof}

We will henceforth replace $\Phi_{\Gamma}$ by such a quasi-isomorphic functor to $\scrF(Y)$, without however changing our notation.
We would like this composite correspondence $\Gamma$ to behave like a sheaf of Clifford algebras over $B$.  One can filter the Floer complex for $T^2 \tilde{\times} B$ by projecting generators down to critical points of a Morse function on $B$ and filtering by Morse degree.  This yields a spectral sequence of rings
 \[
E_2^{*,*} =  H^*(B; HF(T^2,T^2))\cong H^*(B)\otimes Cl_2\  \Rightarrow \ HF(T^2\tilde{\times} B, T^2\tilde{\times} B)
 \]
In the simplest case, the spectral sequence degenerates,  the non-trivial idempotents in the Clifford algebra survive to $E_{\infty}$, and they split the total correspondence,  as in Remark \ref{Rem:CorrNonZero}.  This was trivially the situation in the example considered in Lemma \ref{Lem:ToricFibre}, where $B=\{pt\}$.  In general, the picture is less clear: even additively the Floer cohomology of the composite correspondence is more complicated, coming from the non-triviality of the normal bundle $\mathcal{L}$ of the original hyperplanes $H_i \subset X$.  However, this complication becomes irrelevant on the appropriate summand of the Fukaya category of the base locus. 

\begin{Lemma} \label{Lem:FirstCorrSplits}
Suppose $L \in \scrF(B;0)$, and let $\Gamma_L$ denote its image under the correspondence $\Gamma = \Gamma_R \circ \Gamma_E$. Either $HF(\Gamma_L, \Gamma_L) = 0$ for all $L \in \scrF(B;0)$, or 
  there is an isomorphism $HF(\Gamma_L, \Gamma_L) \cong HF(L,L) \otimes Cl_2$ of $QH^*(B;0)$-modules.  
\end{Lemma}

\begin{proof}
We work with the monotone symplectic forms provided by Lemma \ref{Lem:Monotone}. 
The Floer cohomology of $L\times S^1_{eq} \subset B^-\times \bP^1$ can be computed using the K\"unneth theorem, since all holomorphic curves split in the obvious product description, as does their deformation theory. One therefore gets
\[
HF(L\times S^1_{eq}, L\times S^1_{eq}) \ \cong \ HF(L,L) \otimes HF(S^1_{eq}, S^1_{eq}). 
\]
By the K\"unneth theorem in quantum cohomology, $QH^*(E) \cong QH^*(B)\otimes QH^*(\bP^1)$.
We next apply the quantum Gysin sequence, in the form of Theorem \ref{Thm:Gysin}, or rather its Lagrangian analogue.  This says that $HF(\Gamma_L,\Gamma_L)$ is the cone on quantum cup-product on $HF(L\times S^1_{eq}, L\times S^1_{eq})$ by the corrected Euler class $e+\sigma\cdot 1$, where $\sigma \in \bK$ counts Maslov index 2 disks through a global angular chain for $U_{\nu_E}$.  In our case, writing $t\in QH^2(\bP^1)$ for the fundamental class,
\[
e =  h\otimes 1 + 1\otimes (-t)
\]
where $h = c_1(\mathcal{L}|_B)\in H^2(B)$ is the restriction of $c_1(\nu_{H/X})$.  Recall that our initial assumptions imply this is proportional to $c_1(B)$. We should therefore compute quantum product by
\[
h \otimes 1 + (1 \otimes (\sigma\cdot 1 - t)): HF(L,L) \otimes HF(S^1_{eq}, S^1_{eq}) \rightarrow HF(L,L) \otimes HF(S^1_{eq}, S^1_{eq}).
\]  
Now \emph{ for $L$ in the nilpotent summand} $\scrF(B;0)$, $h \mapsto 0 \in HF(L,L)$ by Lemma \ref{Lem:quantumcap}, and the above therefore simplifies to the map
\[
x \otimes y \mapsto x \otimes (\sigma 1 -t)y
\]
which has cone $HF(L,L) \otimes Cone(\sigma 1 -t )$.     If $\sigma \neq 1$ then the map is $\id \otimes$ (invertible) and hence $HF(\Gamma_L,\Gamma_L) = 0$, so $\Gamma_L$ is quasi-isomorphic to zero for every $L$.  If $\sigma = 1$ this map vanishes and then $HF(\Gamma_L, \Gamma_L) \cong HF(L,L) \otimes Cl_2$.  Since the quantum Gysin sequence is an exact sequence of $QH^*$-modules, the statement on the module structure follows.
\end{proof}

The force of Lemma \ref{Lem:FirstCorrSplits} is that to obtain a conclusion for an arbitrary $L\in \scrF(B;0)$, it suffices to find a single such $L$ for which $HF(\Gamma_L, \Gamma_L) \neq 0$.  
 We now aim to prove this non-vanishing when $L$ is a Lagrangian sphere as provided by hypothesis (\ref{Equation:Hypothesis}).   Let $\mathcal{U} \subset \mathcal{L}^{\oplus 2} \rightarrow B$ denote an open neighbourhood of $B\subset X$, symplectomorphic to a sub-disk bundle of the normal bundle.  Write $\mathcal{L}_{taut}$ for a symplectic disk bundle neighbourhood of $E \subset Y$ inside its normal bundle,  which is just the tautological bundle over $E = \bP(\mathcal{L}^{\oplus 2})$. Since $\pi_2(B \times Y, \Gamma) = \pi_2 (B \times Bl_B(\mathcal{U}), \Gamma)$,  all possible homotopy classes of holomorphic disk with boundary on $\Gamma$ are visible in $B\times Bl_B(\mathcal{U})$.

\begin{Lemma} \label{Lem:DisksLocalise}
If for some Lagrangian sphere $L\subset B$ all holomorphic disks of Maslov index 2 with boundary on $\Gamma_L$ lie inside a disk bundle $\mathcal{L}_{taut} \rightarrow E$, then $\Gamma_L \not \simeq 0$.
\end{Lemma}

\begin{proof}
We continue to write $\Gamma = \Gamma_R\circ \Gamma_E$.
  There is a canonical projection
\begin{equation} \label{Equation:Project}
p: B \times Bl_B(\mathcal{U}) \ \longrightarrow \ B\times B
\end{equation}
under which $\Gamma \mapsto \Delta_B$. Any Lagrangian sphere $L\subset B$ necessarily lies in $\scrF(B;0)$ since $B$ has minimal Chern number $\geq 2$. Disks with non-constant image in $B \times B$ under the map of Equation \ref{Equation:Project} have Maslov number at least 4. Therefore the only Maslov index 2 disks with boundary on $\Gamma_E(L\times S^1_{eq})$ lie in homotopy classes which project trivially to $B\times B$, and thus lie in $Bl_0(\bC^2)$-fibres of the projection $p$, with boundary on $T^2$-fibres of $\Gamma \cong T^2\tilde{\times} B$. All such disks are therefore visible in the model of Lemma \ref{Lem:ToricFibre}, and in that case we know $\sigma = 1$ since we know \emph{a priori} that the Floer cohomology of the torus $T$ was non-vanishing.   (In the model one can check explicitly that only one of the three holomorphic disks through the generic point of $T$ has boundary passing through a global angular chain for the unit circle fibration of $\mathcal{O}(-1) \rightarrow \bP^1$.) 
\end{proof}

Since $Y = Bl_B(X)$ is the blow-up of a base locus of a pencil of hyperplanes $\{H_t\}$, it has a natural fibration structure $\pi: Y\rightarrow \bP^1$, with fibre homologous to $\pi^*H - E$.  
 We choose a holomorphic volume form $\Omega$ on $Y$ which has poles along the reducible divisor
\begin{equation} \label{Equation:HolVolForm}
D_{\Omega} \ = \ E + [\pi^{-1}(0)] + [\pi^{-1}(\infty)] + d H_{ref}
\end{equation}
where we recall $c_1(X) = (d+2)H$. For suitable choices of fibres $H_0 = \pi^{-1}(0)$ and $H_{\infty} = \pi^{-1}(\infty)$ we have $L\times T^2 \subset Y\backslash D_{\Omega}$.  Remark \ref{Remark:LocalSlag} implies that $L\times T^2$ can be graded with respect to $\Omega$. 

\begin{Lemma} \label{Lem:MaslovViaIntersection}
Suppose $\dim_{\bC}(B) > 1$.  Then Maslov index 2 disks with boundary on $L \times T^2$ meet exactly one of $E$, $H_0$ and $H_{\infty}$ once, and are disjoint from $H_{ref}$.  In particular, no such Maslov index 2 disk can project onto $\bP^1$ under $\pi$.
\end{Lemma}

\begin{proof}
Since $dim_{\bC}(B)>1$, the Lagrangian sphere $L$ is simply connected, which implies that $H_2(B) \rightarrow H_2(B,L)$ is surjective.  One can therefore assume that the Maslov index for holomorphic disks in $B$ with boundary on $L$ is given by intersection number with $dH_{ref} \cap B$.  The choice of holomorphic volume form then implies that Maslov index of disks with boundary on $L\times T^2$ is given by their intersection number with $D_{\Omega}$.  The result follows.
\end{proof}

\begin{Example} (Auroux)
The hypothesis on $\dim(B)$ plays a definite role. Take a pencil of linear hyperplanes in $\bP^3$ with base locus $\bP^1$ and blow that up.  This may be done torically, and we obtain a toric variety whose moment map is a truncated tetrahedron.  The Lagrangian $L\times T^2$ is just the monotone torus $T^3$ which is the central fibre of that moment map, and it bounds one family of Maslov index 2 disks for each facet of the moment polytope, in particular it bounds 5 families, \emph{not all of which} miss the reference hyperplane $H_{ref}$ if this is taken to be one face of the tetrahedron. 

This is because the original Lagrangian $L\subset B$ is an equatorial $S^1 \subset \bP^1$, in particular it bounds Maslov index 2 disks. The natural map $H_2(B) \rightarrow H_2(B,L)$ is not surjective; although the first Chern class ``contains" $H_{ref}$ with multiplicity $d=2$, the actual toric divisor which naturally represents it does not contain $H_{ref}$ with multiplicity 2, but rather $H_{ref}$ and another homologous component.
\end{Example}

Away from $B \cap H_{ref}$, the original net of hyperplanes on $X$ defines a map to $\bP^2$.  Blowing up along $B$, one obtains a map $\phi$ from $Y \backslash (E\cap H_{ref})$ to the first Hirzebruch surface $\bF_1$.  Then
\begin{enumerate}
\item the projection map $\pi$ is the composite $Y \backslash (E\cap H_{ref}) \stackrel{\phi}{\longrightarrow} \bF_1 \rightarrow \bP^1$, with the second map the natural projection;  
\item the proper transforms of $H_0$ and $H_{\infty}$ lie over the fibres of $\bF_1$ over $0, \infty \in \bP^1$;
\item the divisors $E$ and $H_{ref}$ (minus their common intersection) live over the $-1$ respectively $+1$ section of $\bF_1$.
\end{enumerate}
This is pictured schematically in Figure \ref{Fig:Hirz}, via the usual moment map image for $\bF_1$.  The fibre over the centre of gravity $\bullet$  is a monotone Lagrangian torus $\bT \subset \bF_1$. \begin{center}
\begin{figure}[ht]
\includegraphics[scale=0.6]{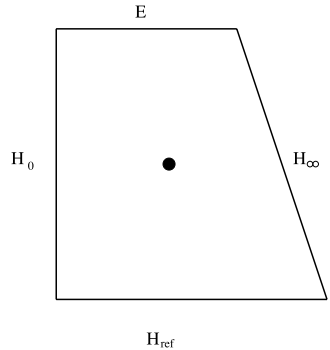} 
\caption{The moment polytope for $\mathbb{F}_1$ and associated divisors\label{Fig:Hirz}}
\end{figure}
\end{center}

\begin{Lemma} \label{Lem:NonZero}
Assuming that $\dim_{\bC}(B)>1$ and that the hypotheses (\ref{Equation:Hypothesis}, \ref{Equation:Hypothesis2}) hold, then $\Gamma_L \neq 0$ for a Lagrangian sphere $L\subset B$.
\end{Lemma}

\begin{proof}
The proof uses a degeneration argument modelled on deformation to the normal cone; compare to \cite[Section 5]{Seidel:Flux}.  We symplectically blow up $Y\times \bC$ along $E\times \{0\}$; the resulting space $W$ has an obvious projection $p$ to $\bC$, which we take to be holomorphic.  The zero-fibre has two components, one of which $C_0$ is diffeomorphic to $E$ but with a K\"ahler form of smaller volume (shifted in cohomology class by a real multiple of the divisor $E$), and the other of which $C_1$ is the projective completion of the normal bundle to $E$, hence is the total space of an $\bF_1$-bundle over $B$.  We can arrange that:
\begin{enumerate}
\item The total space $W$ of the blow-up carries a K\"ahler form $\Omega$ which on the component $C_1$ co-incides with the form obtained by symplectic reduction of the co-isotropic $U_{\nu_E} \subset Y$ appearing in Lemma \ref{Lem:Monotone}. In particular, $C_1$ contains the monotone Lagrangian submanifold $K_L= L\times T^2$ which is the Clifford torus bundle over $L\subset B$, as in Remark \ref{Rem:CorrespondenceViaCut}. 
\item The total space $(W, \Omega)$ of the blow-up contains a Lagrangian submanifold $\textbf{L}$ diffeomorphic to $L \times T^2 \times [0,\infty)$, fibring over the positive real half-line in $\bC$, with $p^{-1}(0) \cap \textbf{L} = K_L$.
\item The generic fibre of $\pi$ is a K\"ahler manifold symplectomorphic to $Y$ with its monotone form, and the Lagrangian submanifold $\textbf{L} \cap \pi^{-1}(t)$ is Hamiltonian isotopic to $\Gamma_E(L)$.
\end{enumerate}
The proof relies on a picking a suitable symplectic connexion on the regular part of the fibration near the component $C_1$ and parallel transporting $K_L$ along the real axis, compare to an identical construction in the case $B=\{pt\}$ given in \cite[Section 4]{SeidelSmith:Ramanujam}. The last statement in (3) follows from the reformulation of $\Gamma_E$ given in Remark \ref{Rem:CorrespondenceViaCut}.

We now consider a sequence $(u_k: D^2 \rightarrow W)$ of holomorphic disks in $W$ with boundary on the total Lagrangian $\textbf{L}$.  Assume moreover that for each $k$, the disk $u_k$ has image inside the fibre $p^{-1}(t_k)$ for some sequence $t_k \rightarrow 0$ of positive real numbers.  Gromov compactness in the total space $(W,\textbf{L})$ implies that there is some limit stable disk $u_{\infty}$, but since the projections $p\circ u_t$ are constant maps, that limit disk must lie entirely in the zero fibre of $p$.  The general description of a Gromov-Floer limit of disks as a tree of holomorphic disks and sphere bubbles, see for instance \cite{Frauenfelder},  implies that $u_{\infty}$ has at least one disk component inside $C_1$, with boundary on $K_L$, and that any component of $u_{\infty}$ meeting $C_0 \backslash (C_0 \cap C_1)$ must be a rational curve, since $\textbf{L} \cap C_0 = \emptyset$.

We now suppose that all the disks $u_k$ have Maslov index 2. Any non-constant disk in $C_1$ with boundary on $K_L$ has Maslov index $\geq 2$, by monotonicity.  Moreover, any non-trivial holomorphic sphere in $C_0$ has strictly positive Chern number, since that space is a Fano variety. It follows that the limit disk $u_{\infty}$ has no component contained in $C_0$, and indeed no component  which intersects the divisor $C_0 \cap C_1$ (such a component would be glued to a non-constant sphere in the component $C_0$ or it would not be locally smoothable to the nearby fibre; compare to  the degeneration formula for Gromov-Witten invariants \cite{IP}). Therefore, for sufficiently large $k$, the holomorphic disk $u_k$ must live inside the open neighbourhood of $E\subset \pi^{-1}(t_k)$ bound by $U_{\nu_E}$. 
The result now follows from Lemma \ref{Lem:DisksLocalise}.
\end{proof}

\begin{Remark}
Another viewpoint on the previous proof is also informative. The fibration over $\bF_1$ has a (generally singular) discriminant curve, whose amoeba projects to some complicated subset of the moment map image.  However, this image is disjoint from the $-1$-section, because we originally chose $H_{ref}$ transverse to $B$ (i.e. a net of hypersurfaces with smooth base locus).  A holomorphic rescaling of $\bF_1$ (the lift of the holomorphic action on $\bP^2$ taking $[x:y:z] \mapsto [x:y:cz]$ where we obtain $\bF_1$ from blowing up $[0:0:1]$) will move the discriminant into a small neigbourhood of $H_{ref}$.  On the other hand, the symplectic structure on $Y$ was constructed, via inflation, to contain a large standard symplectic disk bundle, containing the monotone Lagrangian $L\times T^2$, cf. the proof of  Lemma \ref{Lem:Monotone}.  
In consequence, we may assume that the monotone $L\times T^2$ projects into the upper half of Figure \ref{Fig:Hirz}, whilst the discriminant lies in the lower half.  

The effect of deformation to the normal cone is to make the K\"ahler form standard in a neighbourhood of $E$ containing $L\times T^2$.   Suppose we start with a net of sections obtained by perturbation from a net on the normal crossing degeneration which are degree 1 in the $\bP^2$-fibres of $\bP(\mathcal{L}^{\otimes 2}\oplus\mathcal{O})$ (we have essentially linearised, replacing $s_0$ and $s_1$ by $Ds_0, Ds_1$ which span the normal bundle of $B \subset X$).  As in Remark \ref{Rem:CorrespondenceViaCut}, one can then suppose -- after smoothing and small compactly supported perturbation of the symplectic form near the normal crossing divisor -- that $L\times T^2$ fibres over $\bT$.  There is now a  Lagrangian isotopy of $L\times T^2$,  moving the image point in the moment polytope towards the facet $E$ at the top of Figure \ref{Fig:Hirz}. Since this is not a Hamiltonian isotopy, and the Lagrangian does not stay monotone, there may \emph{a priori} be wall-crossing. However,  Maslov index zero disks  project holomorphically to $\bF_1$, where Maslov index is given by intersection number with the toric boundary.  The projection of any such wall-crossing disk must therefore be constant, so it in fact lay inside a fibre of the projection $\phi$.  Equivalently, the disk would lie inside a copy of $B$.  The Lagrangian isotopy viewed inside $B$ \emph{is} monotone, since $\pi_1(L)=0$, so in fact there is no wall-crossing.  We move $L\times T^2$ into a small neighbourhood of $E$ without changing the families of Maslov 2 disks which it bounds.  These can now be understood by projection to $\bF_1$; each disk hits one of $E, H_0, H_{\infty}$, and there are 3 well-known families of such disks, which were those visible in Lemma \ref{Lem:DisksLocalise}. \end{Remark}

Lemma \ref{Lem:FirstCorrSplits} now says  that $HF(\Gamma_L, \Gamma_L) \cong HF(L,L) \otimes Cl_2$ as a module over $QH^*(B;0)$.  Although the ring structure on the second tensor factor is not \emph{a priori} determined by the quantum Gysin sequence, for degree reasons the Floer product amongst the codimension 1 cycles given by the circle factors $L \times \{pt\} \times S^1$ and $L \times S^1 \times \{pt\}$ is completely determined by disks of Maslov index 2.  The constraints on the  Maslov 2 disks obtained in the proof of Lemma \ref{Lem:NonZero} are sufficient to imply that 
\[
HF(\Gamma_L, \Gamma_L) \ \cong \ HF(L,L) \otimes Cl_2
\] as a ring. This is true for all Lagrangian spheres $L\in \scrF(B;0)$, hence for any such $L$ the image $\Gamma_L$ has a non-trivial idempotent.  We can analyse the Floer cohomology $HF(\Gamma, \Gamma)$ of the entire correspondence $\Gamma \subset B\times Y$ via an analogous Gysin sequence, using the fact that $\Gamma$ is the image of a Lagrangian submanifold $\Delta_B \times S^1_{eq} \subset B\times E$ under a correspondence $B\times E \rightsquigarrow B\times Y$ (the appropriate symplectic forms come from Remark \ref{Rem:MonotoneFormFromCut}).  Taking the $0$-generalised eigenspace for the action of $c_1(B)$, we again arrive at the mapping cone of a quantum product
\begin{equation} \label{Eqn:QproductCorr}
\ast (e+ \sigma \cdot 1): QH^*(B;0) \otimes HF(S^1_{eq}, S^1_{eq}) \longrightarrow QH^*(B;0) \otimes HF(S^1_{eq}, S^1_{eq})
\end{equation}
with $e = h\otimes 1 + 1 \otimes (-t)$ as before.  Now $h$ acts nilpotently, rather than trivially, on $QH^*(B;0)$.  Lemma \ref{Lem:FirstCorrSplits}, and non-triviality of $\Gamma_L$ for some $L$, shows that $\sigma = 1$ in \eqref{Eqn:QproductCorr}, which means that the second term $1 \otimes (\sigma \cdot 1 - t)$ acts trivially on $HF(S^1_{eq}, S^1_{eq})$, so $HF(\Gamma, \Gamma; 0_B)$ inherits a Clifford algebra factor. The idempotents for the individual $\Gamma_L$ derived above have as common source an idempotent for $\Gamma$ arising from splitting this Clifford algebra. We deduce that the functor $\Phi_{\Gamma}$ associated to $\Gamma$ by quilt theory has non-trivial idempotents, if one views its domain as the nilpotent summand $\scrF(B;0)$. 
 We now appeal to Lemma \ref{Lem:SplitIdempotentFunctors} to see that there is a functor $\Phi_{\Gamma}^+: Tw\,\scrF(B;0) \rightarrow Tw^{\pi}\,\scrF(Y)$ associated to either choice of idempotent in the Clifford algebra factor (strictly, associated to a choice of idempotent up to homotopy for the chosen cohomological idempotent).

\begin{Lemma} \label{Lem:FullyFaithful}
The functor $\Phi_{\Gamma}^+$ is fully faithful.
\end{Lemma}

\begin{proof}
Let $L_i$, $i=1,2$, be objects of $\scrF(B;0)$. For simplicity we suppose they are simply-connected (this is the only case which shall be used in the sequel).  The functor $\Phi_{\Gamma}^+$ takes these to idempotent summands of $L_i\times T^2$, which we formally denote by $(L_i\times T^2)^+$.  According to Equation \ref{Eq:AbstractIdempotentModules}, we know
\[
HF((L_i\times T^2)^+, (L_j \times T^2)^+) \ \cong \ (+_i)HF(L_i\times T^2, L_j\times T^2)(+_j)
\]
where the RHS denotes the action on the cohomological morphism group of the relevant idempotents under the module structure of that group for $HF(L_i\times T^2, L_i\times T^2)$ on the left and $HF(L_j\times T^2, L_j\times T^2)$ on the right. From the proof of Lemma \ref{Lem:FirstCorrSplits}, the idempotents both come from a fixed idempotent $p \in HF(T^2,T^2)$, via an isomorphism
\[
HF((L_i\times T^2)^+, (L_j \times T^2)^+) \ \cong \ HF(L_i,L_j) \otimes p\cdot HF(T^2,T^2)\cdot p.
\] 
The second tensor factor is just a copy of the coefficient field $\bK$, completing the proof.
\end{proof}

\begin{Proposition} \label{Prop:BlowUpEmbeds}
Assuming  (\ref{Equation:Hypothesis}, \ref{Equation:Hypothesis2}) and that $\dim_{\bC}(B)>1$, there is a fully faithful embedding $D^{\pi}\scrF(B;0) \hookrightarrow D^{\pi}\scrF(Y).$
\end{Proposition}

This is exactly the content of Lemmas  \ref{Lem:ExtendedFukNotNeeded} and \ref{Lem:FullyFaithful}.
Now take $B=Q_0\cap Q_1 \subset \bP^{2g+1}$ for some $g>1$, so the space $Y \rightarrow \bP^1$ is the relative quadric.  Lemma \ref{Lem:AmpleBlowUp} shows that the crucial hypothesis (\ref{Equation:Hypothesis2}) holds, and Proposition \ref{Prop:BlowUpEmbeds} yields one of the embeddings of  Theorem \ref{Thm:Embed}.

\subsection{Spinning\label{Section:Spinning}}
 There is a more geometric approach to the material of Section \ref{Section:Idempotent}, relevant for proving Addendum \ref{Thm:Action}.  Let $\nu \rightarrow H$ be a real rank 2 symplectic vector bundle over a symplectic manifold $(H,\omega_H)$, and equip $\nu$ with a Hermitian metric.  Suppose the first Chern class $c_1(\nu) = t [\omega_H] \in H^2(H)$ is positively proportional to the symplectic form.  The symplectic form on the total space of the disk bundle of $\nu$  of radius $< 1/\sqrt{t}$ is 
 \[
\Omega_{\nu}  = \pi^*\omega_H + \frac{1}{2} d(r^2 \alpha)
\]
where $r$ is the radial co-ordinate and $\alpha$ is a connection form for the bundle satisfying $d\alpha = -t \pi^*\omega_H$. Fibrewise, this is the standard area form on $\bR^2$. If $\Delta \subset H$ is a Lagrangian cycle and $\hat{\Delta}$ is the lift of $\Delta$ to $\nu$ given by taking the union over $\Delta$ of circles $\{r=constant\}$, then $\hat{\Delta}$ is also Lagrangian, since
\[
\Omega_{\nu}|_{\hat\Delta} = (\pi^*\omega_H + r.dr.d\alpha + \frac{1}{2} r^2 d\alpha)|_{\hat\Delta} = \pi^*(1 - t r^2 / 2)\omega_H
\]
is pulled back from the base.  Now suppose we have a pencil of hyperplane sections $\{H_t\}_{t\in\bP^1}$ of a projective variety $P$ with base locus $H_0\cap H_{\infty} = M \subset P$.  Strengthening hypothesis \eqref{Equation:Hypothesis}, we suppose the elements of the pencil are ample hypersurfaces Poincar\'e dual to a multiple of the symplectic form, which is the condition $c_1(\nu) = t [\omega_H]$ above. The blow-up $Bl_MP$ is obtained by removing a symplectic tubular neighbourhood of $M$, with boundary an $S^3$-bundle over $M$, and then projectivising the fibres of that unit normal bundle via the Hopf map $S^3 \rightarrow S^2$. Suppose we have a symplectic form $\Omega_{\lambda}$ on $Bl_M(P)$ for which the ruling curves have area $\lambda > 0$.

\begin{Lemma} \label{lem:spin}
Let $S^n\cong L\subset M$ be a Lagrangian sphere.  Fix a Lagrangian disk $\Delta^{n+1}\subset H$ with $\partial \Delta = L$ in a smooth element $H\subset P$ of the pencil, and with $\Delta \cap M = L$.  There is a Lagrangian sphere $S^{n+2}\cong\hat{L} \subset (Bl_M(P),\Omega_{\lambda})$. The Hamiltonian isotopy class of $\hat{L}$ is determined by the Lagrangian isotopy class of $\Delta$.
\end{Lemma}

\begin{proof}
Locally near $H$, the symplectic manifold $P$ is symplectomorphic to the normal bundle $\nu\rightarrow H$. Fix a compatible metric on this normal bundle.  Form a Lagrangian tube in $P$ from the radius $\lambda$ circles over $\Delta$. After an isotopy of $\Delta$  if necessary, this Lagrangian tube will meet the boundary of the $\lambda$-disk bundle over $M$ in a subfibration which is a circle bundle over $L$ composed entirely of Hopf circles in $S^3$-fibres of the normal bundle. Blowing up (i.e. symplectic cutting) will collapse these circles to points. Differentiably, the lift of the tube over $\Delta$ is therefore a circle bundle over an $(n+1)$-disk with the circles collapsed  over the boundary; this is topologically a sphere, and near any collapsed circle it is locally modelled on Example \ref{Ex:SpinModel}, hence is globally smooth. It is easy to see that the choices, for fixed $\lambda$, lead to Lagrangian isotopic lifts to the total space, which are therefore Hamiltonian isotopic.  
\end{proof}

\begin{Definition}\label{Definition:SpunSphere}
In the situation of Lemma \ref{lem:spin}, we will say that $\hat{L}$ is obtained from $L \subset M$ by \emph{spinning} the thimble $\Delta$.
\end{Definition}

In general, for such an $S^n \cong L \subset M$, we now have two constructions of spherical objects in $Tw^{\pi} Bl_M(P)$, namely $\hat{L}$ and the image $\Phi_{\Gamma}^+(L)$.  At least in the special case in which $\scrF(M)$ is split-generated by a chain of Lagrangian spheres, we are able to relate these two objects. 
 To make sense of the next result, note that a symplectomorphism $\phi: M \rightarrow M$ which is the monodromy of a family of embeddings $\{M_t\}_{t\in S^1} \rightarrow P$ induces a symplectomorphism $\hat{\phi}: Bl_M(P) \rightarrow Bl_M(P)$, well-defined up to isotopy, by  taking parallel transport in a fibrewise blow-up. We again refer to $\hat{\phi}$ as the spin of $\phi$.

\begin{Lemma} \label{lem:SpinMany}
In the situation of Lemma \ref{lem:spin}, let $\{V_1,\ldots, V_n\}$ be Lagrangian spheres in $M$ which are vanishing cycles for the pencil $\{H_t\}$ on $P$. There are spinnings $\{\hat{V}_j\} \subset (Bl_M(P), \Omega_{\lambda})$ with the property that $\hat{V}_i \cap \hat{V}_j = V_i\cap V_j$ for all $i \neq j$.  Moreover,  for any choice of vanishing thimble $\Delta$ for $V_j$, the Dehn twist $\tau_{V_j}$ on $M$ spins to the Dehn twist $\tau_{\hat{V}_j}$ in $\hat{V}_j \subset Bl_M(P)$.
\end{Lemma}

\begin{proof}
We choose the bounding thimbles $\Delta_j$ with $\partial \Delta_j = V_j$ to lie in one fixed smooth element $H$ of the pencil of hypersurfaces on $P$, and to meet only at their boundaries (this is always possible by the definition of a Lefschetz pencil).  The associated Lagrangian tubes   meet only in the circle fibrations which are collapsed by the blowing-up process, which gives the co-incidence of intersections $\hat{V}_i \cap \hat{V}_j = V_i\cap V_j$.  The last statement follows by considering parallel transport around the loop which encircles the path in $\bC$ over which a chosen Lefschetz thimble fibres.
\end{proof}

We now return to the specific situation of Section \ref{Subsec:Pencil}.  Take $\gamma \subset \Sigma_g$ to be a simple closed curve associated to a  matching path of $\Sigma \rightarrow \bP^1$, so $\gamma$ is invariant under the hyperelliptic involution on $\Sigma_g$. From Wall's Lemma \ref{Ex:genus2} there is a vanishing cycle $V_{\gamma} \subset Q_0 \cap Q_1$ associated to the degeneration of $\Sigma$ which collapses $\gamma$.  Thus, $V_{\gamma}$ is a vanishing cycle for the family of $(2,2)$-intersections defined by a hyperelliptic Lefschetz fibration over the disk with fibre $\Sigma$ and vanishing cycle $\gamma$.  We know from Section \ref{Subsec:Pencil}  that $V_{\gamma}$ arises as a vanishing cycle for a Lefschetz pencil on the quadric $2g$-fold $Q\subset \bP^{2g+1}$.   $V_{\gamma}$ therefore bounds a Lagrangian disk $\Delta_{\gamma} \subset Q$, and by fixing a choice of such a Lagrangian disk one can spin the sphere $V_{\gamma}$ to obtain a Lagrangian sphere $\hat{V}_{\gamma} \subset Z$, as in Definition \ref{Definition:SpunSphere}.
 Write $\scrP \subset \scrF(Z)$ for the subcategory generated by spun spheres $\hat{V}_{\gamma}$, for hyperelliptic-invariant curves $\gamma \subset \Sigma$.  
 
  According to Section \ref{Section:Curve}, the Fukaya category of the curve $\Sigma$, and hence the subcategory $\scrP$, is split-generated by an $A_{2g+1}$-chain of Lagrangian spheres (when $g=2$ these are the spheres associated to the curves $\zeta_j$, $1\leq j\leq 5$, of Figure \ref{Fig:genus2pencil}).  Lemma \ref{Lem:HolonomySpheresMeetRight} shows that the associated $V_{\gamma}$-spheres also define an $A_{2g+1}$-chain in $Q_{2,2}$, which moreover split-generate the nilpotent summand of its Fukaya category.  Label the spheres of this chain by $V_{\gamma_1},\ldots, V_{\gamma_{2g+1}}$, and their spinnings by $\hat{V}_{\gamma_j}$.   Recall the functor $\Phi^+_{\Gamma}: D^{\pi}\scrF(Q_0 \cap Q_1;0) \rightarrow D^{\pi}\scrF(Z)$ of Proposition \ref{Prop:BlowUpEmbeds}  takes any simply-connected Lagrangian submanifold $L$ to an idempotent summand of $L\times T^2 \subset Z$.  

 \begin{Lemma} \label{Lem:TriangleCount}
 The cup-product map
\begin{equation} \label{Eqn:ThisWillDo}
HF(\hat{V}_{\gamma_{j}}, \hat{V}_{\gamma_{j+1}}) \otimes HF(V_{\gamma_j} \times T^2, \hat{V}_{\gamma_{j}}) \longrightarrow HF(V_{\gamma_j} \times T^2, \hat{V}_{\gamma_{j+1}})
\end{equation}
does not vanish.
 \end{Lemma}
 
 \begin{proof}
  Let  $Q_{ref}$ denote a ``reference" quadric hypersurface, taking the role of the abstract $H_{ref}$ from Section \ref{Section:Idempotent} and extending the given pencil $\langle Q_0, Q_1\rangle$ of quadrics to a net. The Lefschetz thimbles $\Delta_{\gamma_i} \subset Q$ lie in the fixed hypersurface $Q$ from the original pencil, and we may assume they lie in the complement of $Q\cap Q_{ref}$.  Moreover, the base locus $B \subset Q$ meets $Q_{ref}$ in a hypersurface Poincar\'e dual to its normal bundle, hence the complement $B\backslash (B \cap Q_{ref}) \subset Q\backslash (Q\cap Q_{ref})$ has trivial normal bundle.  The Lagrangian submanifolds appearing in Equation \ref{Eqn:ThisWillDo} are, locally near the intersection circle of $V_{\gamma_j} \times T^2$ and $\hat{V}_{\gamma_{j+1}}$, products of two Lagrangian disks in $B\backslash (B \cap Q_{ref})$ meeting transversely once at a point $p_0$, with the Lagrangian submanifolds occuring in Equation \ref{Eqn:ProductOnSpunSpheres}.  The discussion there shows the map of Equation \ref{Eqn:ThisWillDo} does not vanish provided all holomorphic triangles remain in an open neighbourhood of  $\{p_0\}\times T^2 \subset (B\backslash (B \cap Q_{ref})) \times Bl_0(\bC^2)$ in which this product description holds valid.  

The argument is now a variant of that of Lemma \ref{Lem:NonZero}.  Fix a volume form $\Omega_{perp}$ on $Bl_0(\bC^2)$ with simple poles along the transform of a co-ordinate axis in $\bC^2$.  We fix a holomorphic volume form $\Omega'$ on $Z$ with the properties
\begin{itemize}
\item the polar divisor $D_{\Omega'} \ = \ [\pi^{-1}(\infty)] + g [Q_{ref}] \ \simeq \ c_1(Z)$; 
\item locally near  the affine open set $B\backslash (B \cap Q_{ref})$, $\Omega' = \Omega_{base} \wedge \Omega_{perp}$.
\end{itemize}
Using $\Omega_{base}$ we grade the Lagrangian spheres $V_{\gamma_j}$ in $B \backslash (B \cap Q_{ref})$ such that the intersection point at $p_0$ has degree $0$.  The non-trivial Floer gradings amongst the Lagrangians of Equation \ref{Eqn:ThisWillDo} come from fibre directions, and the model of Equation \ref{Eqn:ProductOnSpunSpheres}. Here all Floer gradings are concentrated in degrees $\{0,1\}$, since in that Lefschetz fibration over $\bC$ the Lagrangians were given by two isomorphic Lefschetz thimbles and a torus meeting each thimble cleanly in a circle.  It follows that all contributing holomorphic triangles have Maslov index $0$; by Lemma \ref{Lem:DeformationDegree} they are disjoint from $D_{\Omega'}$ (cf. the corresponding argument in Lemma \ref{Lem:MaslovViaIntersection}, noting that $\dim_{\bC}(B)>1$ in our case).  

Disjointness from $D_{\Omega'}$ means that all relevant holomorphic triangles live over the complement of the boundary divisors $Q_{ref}$ and $Q_{\infty}$, denoted $H_{ref}$ and $H_{\infty}$ in Figure \ref{Fig:Hirz}.  As in Lemma \ref{Lem:NonZero}, we perform a symplectic cut at the boundary of a tubular neighbourhood of $E$ large enough to contain the $\hat{V}_{\gamma}$ and $V_{\gamma} \times T^2$.  Via Gromov compactness, we consider the holomorphic triangles in the reducible zero-fibre of the total space of the deformation to the normal cone, and observe they cannot reach the component which is not the $\bF_1$-bundle for Maslov index reasons (all spheres in the other component have positive Chern number).   Indeed, disjointness from $Q_{ref}$ implies that all the Maslov zero holomorphic triangles are localised in the complement inside the $\bF_1$-bundle over $B$ of the subvariety defined by $Q_{\infty}$, hence  live inside a bundle with fibre $Bl_0(\bC^2)$.  Maximum principle under projection to the affine part $B\backslash (B\cap Q_{ref})$ and in the fibres $Bl_0(\bC^2)$  shows that all the relevant holomorphic triangles are those described in the discussion after Equation \ref{Eqn:ProductOnSpunSpheres}.
 \end{proof}

\begin{Lemma} \label{Lem:Step2}
$\Phi^+_{\Gamma}(V_{\gamma}) \simeq \hat{V}_{\gamma}$ are quasi-isomorphic objects of mod-$\scrP$.
\end{Lemma}

\begin{proof}
It suffices to show that, for $\gamma=\gamma_j$,  there is an element 
\[
q \in CF^{ev}(\Phi_{\Gamma}^+(V_{\gamma}), \hat{V}_{\gamma})
\]
for which
\begin{equation} \label{Eqn:WantIso}
CF(\hat{V}_{\gamma}, \hat{V}_{\gamma_i}) \xrightarrow{\mu^2(\cdot,q)} CF(\Phi_{\Gamma}^+(V_{\gamma}), \hat{V}_{\gamma_i})
\end{equation}
is an isomorphism for each $1\leq i \leq 2g+1$.  A fibred version of Lemma \ref{Lem:Clean} shows that $\hat{V}_{\gamma}$ and $V_{\gamma} \times T^2$ have non-trivial intersection, meeting cleanly along a copy of $V_{\gamma} \times S^1$, which moreover can be perturbed by a Morse function to give intersection locus two copies of $V_{\gamma}$, of differing $\bZ_2$-grading.  The choice of idempotent in $HF(V_{\gamma} \times T^2, V_{\gamma} \times T^2)$ picks out one of these two graded pieces, and there is  a quasi-isomorphism
\[
 CF(\Phi_{\Gamma}^+(V_{\gamma}), \hat{V}_{\gamma}) \ \simeq \ CF(V_{\gamma}, V_{\gamma}).
\]
Take $q$ to be  a chain-level representative for the cohomological unit of $HF(V_{\gamma}, V_{\gamma})$.  The chain groups appearing in Equation \ref{Eqn:WantIso} vanish, so the fact that multiplication by $q$ is an isomorphism is trivial, unless $|i-j| < 2$. Hence there are really two cases to consider: $i=j$ and $|i-j| = 1$, corresponding to taking the same curve twice, respectively adjacent curves, in the $A_{2g+1}$-chain. The two arguments are similar: we give that for $|i-j| = 1$.  In this case, the two Floer complexes in Equation \ref{Eqn:WantIso} have rank one cohomology, so the map is either an isomorphism or vanishes.   To exclude the latter case, since the idempotent summands in $V_{\gamma}\times T^2$ are quasi-isomorphic up to shift, it is enough to show that the map
\begin{equation} 
HF(\hat{V}_{\gamma_{j}}, \hat{V}_{\gamma_{j+1}}) \otimes HF(V_{\gamma_j} \times T^2, \hat{V}_{\gamma_{j}}) \ \xrightarrow{\mu^2(\cdot,q)} HF(V_{\gamma_j} \times T^2, \hat{V}_{\gamma_{j+1}})
\end{equation}
does not vanish. This is the conclusion of Lemma \ref{Lem:TriangleCount}.
\end{proof}

\begin{Corollary} \label{Cor:MappingClassCompatible-1}
The triangulated embedding $\Phi^+: D^{\pi}\scrF(Q_0 \cap Q_1) \hookrightarrow D^{\pi}\scrF(Z)$ is compatible with the natural weak actions of the hyperelliptic pointed mapping class group, in the sense that  for any element $f \in \Gamma_{g,1}^{hyp}$ there is some quasi-equivalence of functors
\[
\Phi^+ \circ f|_{D^{\pi}\scrF(Q_0 \cap Q_1)} \ \simeq \ f|_{D^{\pi} \scrF(Z)} \circ \Phi^+.
\]
\end{Corollary}

\begin{proof}
It is sufficient to prove compatibility with the Dehn twist defined by an essential simple closed curve $\gamma \subset \Sigma$ invariant under the hyperelliptic involution, since such twists generate the hyperelliptic mapping class group.   We actually restrict to twists in curves projecting to arcs in $\bC \subset \bP^1$, hence the group $\Gamma_{g,1}^{hyp}$, so that the relevant hyperelliptic curves live in families over configuration spaces of points in $\bC$ and the description of the associated pencil of quadrics in the co-ordinate terms of Equation \ref{Eqn:PencilCoord} is always valid.  The Dehn twist in $\gamma \subset \Sigma$  acts by the Dehn twist in $V_{\gamma} \subset Q_0 \cap Q_1$ by Section \ref{Subsec:Pencil} and Lemma \ref{Ex:genus2}, and by the Dehn twist in $\hat{V}_{\gamma} \subset Z$ by Lemma \ref{lem:SpinMany}.  The relation between algebraic and geometric twists, Proposition \ref{Prop:twists} and more specifically Equation \ref{Eqn:TwistFunctor}, implies
\[
\scrT_{\Phi^+(V_{\gamma})} \circ \Phi^+ \ = \ \Phi^+ \circ \scrT_{V_{\gamma}}.
\]
Using Lemma \ref{Lem:Step2} to identify $\Phi^+(V_{\gamma}) \simeq \hat{V}_{\gamma}$, we see $\Phi^+$ entwines the action of a Dehn twist. This implies the Lemma (note we do not claim any compatibility between the quasi-isomorphisms for different mapping classes).
\end{proof}

%%%%%%%%%%%%%%%%%%%%%%%%%%%%%%%%%%%%%%%%%%%%
\section{From the curve to the relative quadric\label{Sec:CurveToQuadric}}

\begin{Notation} For a space $M$ equipped with a map $p:M \rightarrow \bP^1$, write $M_{\neq t}$ for the complement of the fibre $p^{-1}(t) \subset M$; similarly $M_{\neq T}$ denotes the complement of the union of the fibres for $t\in T$.
\end{Notation}

\subsection{Matching Spheres\label{Section:Match}}

Let $Z$ continue to denote the blow-up of $\bP^{2g+1}$ along $Q_0\cap Q_1$. 
A matching cycle is by definition a path $\chi: [-1,1] \rightarrow \bP^1$ in the base of a Lefschetz fibration which ends at distinct critical values (and is otherwise  disjoint from the set of critical values) and for which the associated vanishing cycles are Hamiltonian isotopic in the fibre over $\chi(0)$.  Any path between critical points for $Z \rightarrow \bP^1$ is a matching path, by Lemma \ref{Lem:uniquesphere}.

\begin{Lemma} \label{Lem:5sphere}
Any matching path $\chi:[-1,1] \rightarrow \bP^1$ for $Z\rightarrow \bP^1$ defines a Lagrangian sphere $L_{\chi} \subset Z$, unique up to Hamiltonian isotopy.  
\end{Lemma}

\begin{proof}
The monotone symplectic form on $Z$ is restricted from a global K\"ahler form on $\bP^{2g+1} \times \bP^1$. We define the sphere $L_{\chi}$ to be the vanishing cycle of the nodal algebraic degeneration of $Z$ in which two critical values coalesce along $\chi$ (such a degeneration is easily obtained by moving $Z$ in a  Lefschetz pencil on $\bP^{2g+1} \times \bP^1$). 
\end{proof}

\begin{Remark}
 If one restricts to $Z_{\neq t}$ for some $t\not \in \chi$, there is another description of this sphere.
By hypothesis the vanishing cycles associated to $\chi|_{[-1,0]}$ and $\chi|_{[0,1]}$ are Hamiltonian isotopic in the fibre $Q_0$, by a Hamiltonian isotopy $H_s$.  One can deform the symplectic connection of the fibration $Z_{\neq t} \rightarrow \bC$ so that the parallel transport along $\chi|_{[-\delta,\delta]}$ incorporates this isotopy, and exactly matches up the boundaries of the two vanishing thimbles, cf. Lemma 8.4 of \cite{AMP}.   It is standard that the two constructions give rise to the same Lagrangian sphere up to isotopy.  Note that the matching cycle construction, requiring a deformation of symplectic connection, \emph{a priori} only yields a Lagrangian sphere for a cohomologically perturbed symplectic form on the total space of $Z$, rather than of $Z_{\neq t}$, which is why the definition as a vanishing cycle is preferable.
\end{Remark}

A matching sphere $L_{\chi}$ can be oriented by choosing an orientation of the matching path $\chi$ and of the vanishing cycle in the fibre.  (Since Dehn twists act on even-dimensional spheres by reversing orientation, the latter choice will not generally be preserved by monodromy.) 
The spheres $L_{\chi}$ are homologically essential, which implies that they are non-zero objects of $\scrF(Z)$, with Floer homology additively given by $H^*(S^{2g+1})$.  

\begin{Lemma} \label{Lem:Unchanged}
For a matching path $\gamma \subset \bC$ defining  $L_{\gamma}\subset Z$, $HF(L_{\gamma}, L_{\gamma}) \cong H^*(S^{2g+1})$ as a ring.
\end{Lemma}

\begin{proof}
From Lemma \ref{Lem:Composition}, there are natural maps
\begin{equation} \label{Eqn:ClosedOpenClosed}
QH^*(Z) \rightarrow HF^*(L_{\gamma}, L_{\gamma}) \rightarrow QH^*(Z)
\end{equation}
whose composite is given by quantum product
\[
QH^*(Z) \longrightarrow QH^*(Z), \quad \alpha \mapsto [L_{\gamma}] \ast \alpha
\]
with the fundamental class $[L_{\gamma}] \in QH^*(Z)$.   Let $L_{\gamma}^{\dagger}$ denote the matching cycle associated to a path meeting $\gamma$ transversely once, so $[L_{\gamma}^{\dagger}]$ and $[L_{\gamma}]$ represent Poincar\'e dual cycles in $Z$.  Equation \ref{Eqn:ClosedOpenClosed} takes $[L_{\gamma}^{\dagger}] \mapsto [L_{\gamma}] \ast [L_{\gamma}^{\dagger}].$
The image is non-zero, since the classical cup-product of the classes is $\pm 1$ and deformation to the quantum product does not change the coefficient in the maximal cohomological degree.  It follows that 
\[
QH^*(Z) \rightarrow HF^*(L_{\gamma}, L_{\gamma})
\]
is surjective, with $1_Z$ and $[L_{\gamma}^{\dagger}]$ mapping to generators (the image of the first is non-trivial by unitality; the images of the two classes are distinct since they have different  mod 2 degrees). Since $[L_{\gamma}^{\dagger}]\ast [L_{\gamma}^{\dagger}] = 0 \in QH^{ev}(Z)$ by graded commutativity, the Floer product in $HF^*(L_{\gamma}, L_{\gamma})$ is undeformed.
\end{proof}

Seidel \cite{Seidel:graded} proved there are no Lagrangian spheres in $\bP^{2g+1}$, so these matching spheres necessarily intersect the exceptional divisor $E \subset Z$.  Any such intersection point lies on a ruling curve $R\subset E$, which can then be viewed as a Maslov index 2 disk with (collapsed) boundary on $L_{\gamma}$. 

\begin{Lemma} \label{Lem:NotNilpotent}
For any matching path $\gamma$, the sphere $L_{\gamma} \in \scrF(Z;1)$.
\end{Lemma}

\begin{proof}
Let $L=L_{\gamma}$, and fix a Poincar\'e dual sphere $L^{\dagger}$ as before.   By Lemma \ref{Lem:Unchanged}, it suffices to compute the coefficient of the dual cycle $[L]$ in the quantum cup-product $c_1(Z) \ast [L^{\dagger}] \in QH^{odd}(Z)$.  A dimension count shows that only Chern class 2 spheres contribute, and the only such sphere with a holomorphic representative is the ruling curve $R$, so we are interested in the algebraic number of $R$-curves passing through $L$, $L^{\dagger}$ and $E$.  The spinning construction of $L$ and $L^{\dagger}$ given in Section \ref{Section:Spinning} shows we can take them to come from $(2g-1)$-spheres inside $E$ meeting transversely once. The $R$-curve through that intersection point is then the unique curve which contributes to the product. 
\end{proof}

\begin{Lemma} \label{Lem:Step1}
For a matching path $\gamma \subset \bC$, the spun sphere $\hat{V}_{\gamma}$ and the matching sphere $L_{\gamma}$ are quasi-isomorphic objects of $\scrF(Z)$.
\end{Lemma}

\begin{proof}
Take a loop $Y \rightarrow S^1$ of genus g curves $\{\Sigma_g^t\}_{t\in S^1}$ with monodromy the Dehn twist $t_{\gamma}$.  $Y$ is obtained as the double cover of $\bP^1 \times S^1$ over a multi-section comprising $2g$ constant sections and a bi-section in which the branch points undergo a half-twist.  There is an obvious inclusion $Y \hookrightarrow \mathcal{Y} \rightarrow \bC$ to a Lefschetz fibration in which the curve defined by $\gamma$ has collapsed in the central fibre.  $\mathcal{Y}$ defines  algebraic Lefschetz fibrations with generic fibre $Z$ in two ways. First, one can construct a family of relative quadrics acquiring a node, with vanishing cycle $L_{\gamma}$, as in Lemma \ref{Lem:5sphere}.  Second, one can consider the Lefschetz fibration of $(2,2)$-intersections $\{Q_{2,2}\}_{t\in \bC}$ with vanishing cycle $V_{\gamma}$,  define a  family of relative quadrics by blowing up $\bC^*\times \bP^{2g+1}$ along $\bC^*\times (Q_{2,2})_t$, and noting that the total space again extends to an algebraic Lefschetz fibration over $\bC$.   The induced monodromy around the circle is then a Dehn twist in $\hat{V}_{\gamma}$, by Lemma \ref{lem:SpinMany}. The constructions exhibit both $L_{\gamma}$ and $\hat{V}_{\gamma}$ as vanishing cycles for the algebraic degeneration of $Z$ with monodromy induced by $t_{\gamma}$, which implies they are Hamiltonian isotopic.
\end{proof}

\begin{Corollary}\label{Cor:QHRelativeQuadric}
The subcategory  $\scrP\subset \scrF(Z;1)$ generated by the matching spheres $L_{\gamma}$ for arbitrary paths $\gamma$ has Hochschild cohomology $HH^*(\scrP,\scrP)\cong H^*(\Sigma_g)$.
\end{Corollary}

\begin{proof}
Lemma \ref{Lem:Step1} shows that $\scrP$ is the same as the subcategory defined by spun spheres $\hat{V}_{\gamma}$.  According to Lemma \ref{Lem:Step2} this is in turn the subcategory which is the image of the functor $\Phi^+: Tw^{\pi}\scrF(Q_0 \cap Q_1) \rightarrow Tw^{\pi}\scrF(Z)$.  On the other hand, this summand of the functor $\Phi$ is fully faithful, by Lemma \ref{Lem:FullyFaithful}; and finally, the Hochschild cohomology of $\scrF(Q_0 \cap Q_1)$ (and hence its split-closure) was calculated in Corollary \ref{Cor:QHIntersectQuadrics}.
\end{proof}

\begin{Remark}
The relative quadric $Z \subset \bP^{2g+1} \times \bP^1$ is a divisor of bidegree $(2,1)$, from which perspective it is straightforward to construct  a Lefschetz pencil with fibre $Z$ and with vanishing cycles matching spheres in the sense of Lemma \ref{Lem:5sphere}.  However, the cycle class $\mathcal{C}(w)$ for the corresponding fibration $w: Bl_{Z \cap Z'}(\bP^{2g+1} \times \bP^1) \rightarrow \bP^1$ does not act nilpotently on $HF(L_{\chi}, L_{\chi})$, which means we can't compute the Hochschild cohomology of $\scrP$ by direct appeal to Corollary \ref{Cor:QuiltsQH}.  This accounts for the somewhat roundabout proof of Corollary \ref{Cor:QHRelativeQuadric}, and its appeal to Lemma \ref{Lem:Step2}.  \end{Remark}

We take the matching paths of the pentagram, Figure \ref{Figure:Pentagram}, and write $\scrA$ for the subcategory of $\scrF(Z)$ generated by the associated matching spheres.

\subsection{Thimbles\label{Section:Thimbles}}

For $\pi: W \rightarrow \bC$ a Lefschetz fibration, with exact or monotone fibres, we denote by $\scrV(\pi)$ the \emph{Fukaya category of the Lefschetz fibration} in the sense of \cite[Section 18]{FCPLT}; we refer to \emph{op. \!cit.} for foundational material on such categories.     $\scrV(\pi)$ has objects  either closed Lagrangian submanifolds or Lefschetz thimbles, equipped with the usual brane and perturbation data;  we recall that one can achieve transversality for curves without components contained in fibres within the class of almost complex structures making $\pi$ pseudo-holomorphic.  In general $W$ will not be convex at infinity.   However, when we consider compact Lagrangian submanifolds, they project to a compact set in $\bC$, and (using almost complex structures compatible with the fibration structure) the maximum principle in the base ensures spaces of holomorphic curves are compact. When we consider $A_{\infty}$-operations between Lefschetz thimbles, these thimbles will always be perturbed so all the intersection points lie within a compact set in the base, which is sufficient for well-definition of the Fukaya category.  As a particular instance of this, the Floer cohomology of a Lefschetz thimble with itself is defined by small Hamiltonian perturbation at infinity.  

\begin{Proposition}[Seidel] \label{Prop:generate}
Any compact Lagrangian submanifold in $\scrV(\pi)$ is generated by a distinguished basis of Lefschetz thimbles.
\end{Proposition}

\begin{proof}[Sketch]
In the exact case, this is Proposition 18.17 of \cite{FCPLT}.  The argument relies on a beautiful double-covering trick.  Namely, one considers the double cover $\widehat{W}$ of $W$ over a smooth fibre $\pi^{-1}(t)$, viewed as near infinity.  Lefschetz thimbles $\Delta_i$ with boundary lying in $\pi^{-1}(t)$ lift to closed Lagrangian spheres $S_i$ in $\widehat{W}$, and one actually \emph{defines} $\scrV(\pi)$ as a $\bZ_2$-equivariant version of the usual Fukaya category $\scrF(\widehat{W})$; this requires the underlying coefficient field $\bK$ not to be of characteristic 2.  Any \emph{compact} $L\subset W$ lifts to a pair of disjoint compact Lagrangians $L^{\pm} \subset \widehat{W}$, which are exchanged by the covering involution $\iota$. This involution can be expressed as a product of Dehn twists $\iota = \prod_j \tau_{S_{i_j}}$ in the spheres $S_i$ associated to a basis of thimbles $\{\Delta_i\}$.  We now use  the monotone version of Proposition \ref{Prop:twists}.  (Recall that Wehrheim and Woodward have extended Seidel's exact triangle for a Dehn twist to the monotone case, Theorem \ref{Thm:WWtriangle}, see also Oh's  \cite{Oh:triangle}.  The exact triangle, together with general aspects of quilt theory, implies that the geometric Dehn twist and algebraic twist are quasi-isomorphic functors \cite{WW:triangle}.)  The disjointness 
\[
\prod_j \tau_{S_{i_j}}(L^+) \cap L^+ \ = \ L^- \cap L^+ \ = \ \emptyset
\]
therefore yields a vanishing arrow in an exact triangle in $\scrF(\widehat{W})$ which implies split-generation of the summand $L^+$ by the spheres $S_j$.  This is then interpreted downstairs in $W$, i.e. in $\scrV(\pi)$, as generation by the thimbles $\Delta_j$. 
\end{proof}

  Let  $C_{\neq\infty}$ and $Z_{\neq\infty}$ be the Lefschetz fibrations over $\bC$ with fibre a pair of points, respectively a $2g$-dimensional quadric, given by removing the $\infty$-fibres from the closed manifolds $C$ and $Z$.  We can alternatively remove the 0-fibres to yield $C_{\neq 0}$ and $Z_{\neq 0}$, which are also Lefschetz fibrations over $\bC = \bP^1\backslash \{\infty\}$.

\begin{Lemma}\label{Lem:thimblesagree}
The categories $Tw\,\scrV(C_{\neq\infty})$ and $Tw\,\scrV(Z_{\neq\infty})$ are quasi-isomorphic as $\bZ_2$-graded $A_{\infty}$-categories.  Similarly, $Tw\,\scrV(C_{\neq 0}) \simeq Tw\,\scrV(Z_{\neq 0})$.
\end{Lemma}

\begin{proof}
The  category generated by a distinguished basis of vanishing thimbles is just the directed category amongst the associated vanishing cycles $\{\partial \Delta_i\}$.  This category is computed entirely inside the fibre of the Lefschetz fibration. For $Z_{\neq\infty}$ the vanishing cycles are all the particular Lagrangian sphere constructed in Lemma \ref{Lem:uniquesphere} (whilst the fibre $S^0$ of $C_{\neq\infty}\rightarrow \bC$ obviously contains a unique Lagrangian sphere).  Thus, the relevant summands of the Fukaya categories of the fibres are quasi-isomorphic by an equivalence which identifies the ordered collection of vanishing cycles;  the result follows. The argument for the categories given by removing the 0-fibres is identical, since Proposition \ref{Prop:generate} does not require that the Lefschetz fibration have trivial monodromy at infinity.
\end{proof}

Let's call a \emph{basic path} a linear path in $\bC$ between two $(2g+1)$-st roots of unity adjacent in argument; one is dotted in Figure \ref{Figure:Pentagram}.  Thus, the pentagram spheres are matching spheres for paths which are given by completing a pair of consecutive basic paths to a triangle. Similarly, if a \emph{radial path} is one from the origin to a root of unity, the pentagram spheres are alternatively obtained from pairs of radial paths. We record this for later:

\begin{Lemma} \label{Lem:BasictoPent}
Each pentagram matching path is obtained from some basic, respectively radial, matching path by applying a half twist  along another basic, respectively radial, matching path.
\end{Lemma}

\begin{Lemma} \label{Lem:pentagramsagree}
The equivalence of Lemma \ref{Lem:thimblesagree} yields a quasi-isomorphism between the subcategories $Tw \,\scrA(C_{\neq\infty})$ and $Tw\,\scrA(Z_{\neq\infty})$ generated by the finite set of Lagrangian matching spheres associated to the arcs of the generalised pentagram.
\end{Lemma}

\begin{proof}
The matching sphere for a basic path  is the cone on a distinguished degree zero morphism between the thimbles associated to the end-points, by \cite[Proposition 18.21 \& Remark 20.5]{FCPLT} (note that our fibres are semi-simple, so it is the special case covered by Remark 20.5 which is relevant).  These cones amongst thimbles involve identical data in the two spaces.  The same argument shows that the twist functors associated to the spheres for basic matching paths are entwined by the equivalence of Lemma \ref{Lem:thimblesagree}.  The result now follows from Lemma \ref{Lem:BasictoPent}. Note that the quasi-equivalence acts cohomologically by the obvious isomorphism of the underlying vector spaces of cohomology groups. \end{proof}

Lemma \ref{Lem:pentagramsagree} and the description of $\scrF(\Sigma_g)$ given in Section \ref{Section:Curve} imply that the Floer cohomology algebra $A=H(\scrA)$ defined by the pentagram spheres in $Z_{\neq\infty}$ is the $\bZ_2$-graded semi-direct product $\Lambda(V)\rtimes \bZ_{2g+1}$, for a 3-dimensional vector space $V$.   To pin down higher products explicitly, it is computionally most efficient to avoid Hamiltonian perturbations, and to pass to a pearly model for the Fukaya category as in 
Remark \ref{Rem:Pearls}. This involves fixing metrics and Morse functions on the individual pentagram Lagrangians. We take these to be equivariant under the obvious finite rotation symmetry group. 

\begin{Lemma} \label{Lem:even}
One can define $\scrA(Z)$ with respect to almost complex structures with the following properties:
(i)  the map $\pi: Z \rightarrow \bP^1$ is pseudoholomorphic; (ii) the structure is integrable near the fibres lying over $0$ and $\infty$;  (ii) the rotation group $\bZ_{2g+1}$ pulled back from $\bP^1$ acts holomorphically.
\end{Lemma}

\begin{proof}[Sketch]
The first statement is standard. The fibres over $0$ and $\infty$ are disjoint from the Lagrangians, and holomorphic pearls contributing to the $A_{\infty}$-structure contain no sphere bubbles by monotonicity, which implies the second statement.  Since no holomorphic curve components of any pearl lie inside the fixed locus of the rotation group, one can achieve equivariant transversality by perturbations on the domain, which gives the last statement, see \cite[Remark 9.1]{Seidel:HMSgenus2}. \end{proof}

\begin{Remark} \label{Rem:EquivtTransverse}
The usual obstruction to equivariant transversality  is the existence of holomorphic curves entirely contained in the fixed point locus of the finite group action, cf. \cite[Lemma 5.12]{KhovanovSeidel}, which cannot exist if the finite group action is free on the Lagrangian submanifolds involved in the relevant category $\scrA$.  More generally, working in characteristic zero one can always take the $A_{\infty}$-structure on $\scrA$  to be strictly invariant under a given action of a finite group $\Gamma$.  In the approach pioneered in \cite{FO3},  the $A_{\infty}$-operations are constructed via virtual perturbations and Kuranishi chains, and equivariance with respect to arbitrary finite group actions is built in from the start. An alternative argument, due to Paul Seidel (private communication),  introduces an artificially enlarged category $\scrA^+$ which contains many copies of each object, one for each group element.  The $\Gamma$-action is now free on objects, hence can be made strict on $\scrA^+$; on the other hand, starting from the equivalence in characteristic zero   $H^*(CC^*(\scrA^+,\scrA^+)^\Gamma) = HH^*(\scrA^+, \scrA^+)^\Gamma$, a deformation theory argument  shows that there must be some $\Gamma$-invariant structure on $\scrA$ itself.
\end{Remark}

%%%%%%%%%%%%%%

\subsection{Exterior Algebra\label{Section:Exterior}}

We continue to study the algebra $A = H(\scrA)$ defined by the $(2g+1)$ pentagram spheres of Figure \ref{Figure:Pentagram}.  We orient the spheres equivariantly for the cyclic symmetry group which permutes them.  Any two pentagram spheres are either disjoint, meet transversely in a point, or meet cleanly in a copy of $S^{2g}$. The two generators of the Floer complex in the last case have degrees of the same parity, so there is no differential. Therefore as $\bZ_2$-graded groups, $HF(L_i, L_j)$ can be computed equivalently in $Z_{\neq\infty}$ or in $Z$ (in either case one just recovers the cohomology of the clean intersection).   It follows that $H(\scrA) = \oplus_{i,j} HF(L_i, L_j)$ is isomorphic, as a $\bZ_2$-graded complex vector space, to the space $\Lambda(V) \rtimes \bZ_{2g+1}$  obtained in Lemma \ref{Lem:pentagramsagree}, and hence to the graded vector space describing $\scrF(\Sigma_g)$ from Section \ref{Section:Curve}.   Our strategy now has three components:
\begin{enumerate}
\item Prove the exterior algebra structure on $H(\scrA)$ survives compactification from $Z_{\neq\infty}$ to $Z$;
\item Prove that the ``first order term" (cubic term) of the $A_{\infty}$-structure is non-trivial;
\item Infer that the $A_{\infty}$-structure on the exterior algebra is determined up to quasi-isomorphism by knowledge of this first order term and of the Hochschild cohomology.
\end{enumerate}
Whilst the final step is pure algebra, and bypasses having to make explicit computations with holomorphic curves, the first two require some control on holomorphic polygons.  Recall from Lemma \ref{Lem:Unchanged} that for each of the individual pentagram spheres $L$, the Floer cohomology ring $HF(L,L) \cong H^*(S^{2g+1})$ is undeformed in $Z$. 

\begin{Lemma} \label{Lem:TwoSpheresOK}
For a pair of pentagram spheres $L, L'$, the Floer products
\begin{equation}
\begin{aligned}
HF(L', L') \otimes HF(L,L') & \longrightarrow HF(L,L') \\
HF(L', L) \otimes HF(L,L') & \longrightarrow HF(L,L)
\end{aligned}
\end{equation}
are the same in $Z$ as in $Z_{\neq\infty}$.
\end{Lemma}

\begin{proof}
In the first case, the input element of $HF(L', L')$ is either the identity, in which case the product is obviously determined, or has odd mod 2 degree.  However, $HF(L,L')$ is concentrated in a single mod 2 degree. Therefore, only $\mu^2(1_{L'}, \cdot)$ can be non-trivial.  The second case is analogous.  By Poincar\'e duality in Floer cohomology, the groups $HF(L,L')$ and $HF(L',L)$ are concentrated in distinct mod 2 degrees.  Therefore, the product can only be non-trivial into the odd degree component of $HF(L,L)$, which is spanned by the fundamental class.  This component is determined by the non-degeneracy of Poincar\'e duality, via its trace, hence will be  independent of the compactification to $Z$. \end{proof}

To compute the algebra structure on $H(\scrA)$ (taken inside $Z$) completely, we must therefore understand Floer products amongst 3 different Lagrangians.  Since 3 distinct pentagram spheres are pairwise transverse or cleanly intersecting, holomorphic pearls reduce to Morse-Bott Floer trajectories, in which we count holomorphic triangles which may have non-trivial cycle constraints (or attached flowlines) at a corner mapping to a clean intersection.  Lemma \ref{Lem:even} implies that such a Floer disk projects to a holomorphic triangle in $\bP^1$ with boundary on the pentagram (but in principle holomorphic triangles in $Z$ might project \emph{onto} $\bP^1$). There are only two distinct possible boundary patterns for the boundary of the triangle in $\bP^1$, modulo the rotation action by $\bZ_{2g+1}$ (in either picture any of the three vertices might be the output of the Floer product). These are depicted in Figure \ref{Fig:Triangle} in the case $g=2$.
\begin{center}
\begin{figure}[ht]
\includegraphics[scale=0.5]{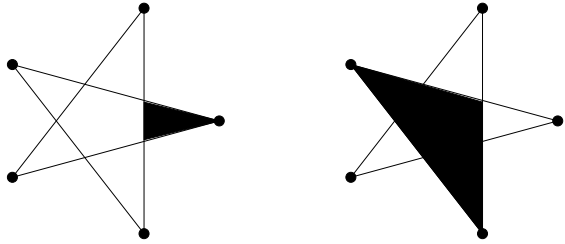} 
\caption{Possible holomorphic triangles\label{Fig:Triangle}}
\end{figure}
\end{center}
In $Z_{\neq\infty}$, the generators of the exterior algebra are given by ($\bZ_{2g+1}$ orbits of) the idempotent summands of the Floer cohomology of the clean intersection spheres (lying over interval vertices) and the isolated external vertices.  Lemma \ref{Lem:pentagramsagree} implies that we can choose $\bZ_2$-gradings which reproduce those of $\scrF(\Sigma_g)$.  From \cite{Seidel:HMSgenus2} and \cite{Efimov}, or by direct computation, these are as follows:
\begin{enumerate}
\item $HF^*(L_\gamma, L_{e^{-2i\pi/(2g+1)}\gamma}) \cong H^*(S^{2g})[1]$ is concentrated in odd degree;
\item $HF^*(L_{\gamma}, L_{e^{-4i\pi/(2g+1}\gamma}) \cong \bC$ is concentrated in even degree.
\end{enumerate}

To clarify, for the small holomorphic triangle on the left of Figure \ref{Fig:Triangle} with the anticlockwise boundary orientation that it inherits as a holomorphic disk in $\bC$, the clean intersections at the internal vertices will be concentrated in odd degree when viewed as inputs, and will have output the even degree external vertex. (Thus, the output is even degree read anticlockwise, meaning viewed as a morphism going from the upper edge to the lower.) In $Z_{\neq\infty}$ (and in the affine hyperelliptic curve) there \emph{are} rigid holomorphic triangles projecting to these small triangles, and their cyclically rotated versions, which define the exterior algebra structure. The existence of these triangles determines the mod 2 degrees of the Floer generators of the algebra $H(\scrA)$.

\begin{Lemma} \label{Lem:NotOne}
No rigid holomorphic triangle in $Z$  has boundary which projects to the second triangle of Figure \ref{Fig:Triangle}.
\end{Lemma}

\begin{proof}
The mod 2 gradings of the vertices preclude existence of any rigid triangle, recalling that $\mu^2$ has degree zero. For instance, on the right of Figure \ref{Fig:Triangle}, if the external vertices are the inputs, they each have odd degree, but the internal vertex as an output also has odd degree.
\end{proof}

Return to the small triangle of Figure \ref{Fig:Triangle}.  Label the Lagrangians on that triangle: $L''$ is the vertical side, $L$ the lower side and $L'$ the upper side.  Let $p \in HF^{odd}(L,L')$ be the Floer generator at the external vertex, viewed as an input.

\begin{Lemma} \label{Lem:NoTriangle}
The product $\mu^2_{\scrF(Z)}(p,\cdot): HF^{odd}(L'',L) \rightarrow HF^{ev}(L'', L')$ is an isomorphism.
\end{Lemma}

\begin{proof}
Consider the Lagrange surgery $L \# L'$ of $L$ and $L'$ at $p$, as depicted by the curved path in Figure \ref{Fig:Surgery}. This surgery is equivalently  given by taking the Dehn twist of $L'$ about $L$.  The surgery is Hamiltonian isotopic, via matching spheres, to a sphere $L\# L'$ which is disjoint from $L''$, namely that defined by the basic path between the leftmost two vertices of Figure \ref{Fig:Triangle}. 
\begin{center}
\begin{figure}[ht]
\includegraphics[scale=0.5]{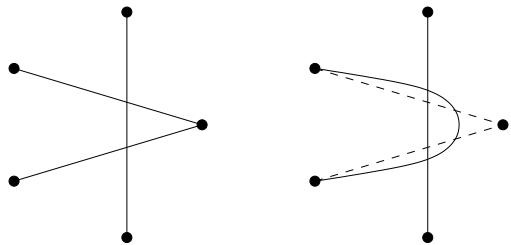} 
\caption{Lagrange surgery at a vertex\label{Fig:Surgery}}
\end{figure}
\end{center}
The exact triangle for monotone Dehn twists implies there is an exact triangle in $\scrF(Z)$
\begin{equation}
\xymatrix{
L \ar[rr] & & L' \ar[dl] \\ & L\#L' \ar[ul]^{[1]} &
}
\end{equation}
This gives a long exact sequence of Floer cohomology groups in $Z$
\[
\cdots \rightarrow HF(L'',L) \rightarrow HF(L'', L') \rightarrow HF(L'', L\#L') \rightarrow \cdots 
\]
Since the terms $HF(L'', L\#L') = 0$, we deduce $\mu^2(p,\cdot): HF(L'',L) \stackrel{\sim}{\longrightarrow} HF(L'', L')$ is an isomorphism.
 \end{proof}

Non-vanishing of $\mu^2(p,\cdot)$ \emph{a posteriori} implies that the algebraic number of triangles in $Z$ with the boundary conditions on the left of Figure \ref{Fig:Triangle} is non-zero. The same argument proves that right multiplication by $p$ also defines an isomorphism
\begin{equation} \label{Eqn:NoTriangleTwo}
HF^{odd}(L', L'') \xrightarrow{\mu^2_{\scrF(Z)}(\cdot,p)} HF^{ev}(L, L'').
\end{equation}
  Poincar\'e duality in Floer cohomology is graded commutative \cite[Section 12e]{FCPLT}. Together with associativity of Floer cup-product,  Lemma \ref{Lem:TwoSpheresOK} and Lemma \ref{Lem:NoTriangle}  together imply that $H(\scrA)$ is still given by an exterior algebra after compactification from $Z_{\neq\infty}$ to $Z$, compare to \cite[Section 10]{Seidel:HMSgenus2}. Explicitly,  in the notation of that paper, the generators of $\Lambda^1$ arise from $p=\xi_3$ and the two idempotents $\xi_1,\xi_2$ of $HF^{odd}(L'',L)$, or rather their $\bZ_{2g+1}$-orbits. The generators of $\Lambda^2$ are the Poincar\'e dual cycles $\bar{\xi}_3, \bar{\xi}_1, \bar{\xi}_2$, and the unit respectively fundamental class  of the pentagram spheres give the generators  $1 \in \Lambda^0$ and $q \in \Lambda^3$. Lemma \ref{Lem:NoTriangle} and its cousin \eqref{Eqn:NoTriangleTwo}  implies that $\xi_2 \cdot \xi_3 = \bar{\xi}_1$ and $\xi_3 \cdot \xi_1 = \bar{\xi}_2$, for suitably chosen signs of the generators, whilst Lemma \ref{Lem:TwoSpheresOK} implies that $\xi_i \cdot \bar{\xi}_i = q = -\bar{\xi}_i \cdot \xi_i$.  These relations define an exterior algebra.
 
%%%%%%%%%

\begin{Corollary} \label{Cor:PentsEqual}
If $g\geq 2$, the Floer cohomology algebra generated by the pentagram spheres in $Z$ is $\Lambda(\bC^3) \rtimes \bZ_{2g+1}$ (graded mod 2).
\end{Corollary} 

Both the curve $C$ and the relative quadric $Z$ define $A_{\infty}$-structures on the exterior algebra which can be viewed as deformations of the particular quasi-isomorphism type of Lemma \ref{Lem:pentagramsagree}, with the deformation parameter counting intersections of holomorphic curves with the fibre over $\infty \in \bP^1$.  Since the relevant deformation space is large, we appeal to more algebraic methods to pin down the $A_{\infty}$-structure.

\begin{Remark} \label{Rem:TwoVblesNeeded}
By comparing to \cite{Seidel:HMSgenus2} and \cite{Efimov}, one sees that two of the generators of the exterior algebra arise from the two idempotent summands of $HF(V,V) \cong H^*(S^0)$, where $V \subset Q$ is the Lagrangian vanishing cycle in a  quadric fibre of $Z$. These play an essentially symmetric role.  Indeed, returning to Remark \ref{Rem:BranchingInfinity}, there is a holomorphic involution $\tau$ of $Z$ which reverses the sign of the first projective homogeneous co-ordinate $x_0$, and acts fibrewise over $\bP^1$, in particular fixes the nodes in the singular fibres. $\tau$ is one lift of the hyperelliptic involution on $\Sigma$ to a symplectomorphism of $Z$.  The  monotone K\"ahler form, which is pulled back from $\bP^{2g+1} \times \bP^1$, is $\tau$-invariant, which means that the associated parallel transport maps along matching paths are $\tau$-invariant, so $\tau$ preserves the Lagrangian pentagram spheres.  In particular, $\tau$ preserves the vanishing cycle $V = L'' \cap L$ in the fibre $Q$ defined by the corresponding matching paths.  Being a lift of the hyperelliptic involution, $\tau$ reverses orientation on $V$ (and so on $L, L''$),  by compatibility of the symplectic group representations on the Riemann surface and the associated $(2,2)$-intersection and hence relative quadric, cf. \cite{Reid}. 

Working with equivariant almost complex structures as in Lemma \ref{Lem:even}, and appealing to Remark \ref{Rem:EquivtTransverse}, we can assume that \emph{the $A_{\infty}$-structure on $Z$ is equivariant with respect to $\tau$.}  Note that if $e$ and $f$ denote generators for $HF(Q,Q) \cong HF(L'', L)$ above, then the idempotent summands are $(e\pm f)/2$, and the fundamental class of $Q$ is the \emph{difference} between the idempotents.  Since the involution $\tau$ acts by an orientation-reversing automorphism of the Lagrangian vanishing cycle it sends $f\mapsto -f$ and exchanges the two idempotents.   Therefore $\tau$-equivariance amounts to saying that $\scrA(Z)$ is equivariant under changing $f\mapsto -f$.   
\end{Remark}

\subsection{The cubic term\label{Sec:Cubic}}

 Although the category $\scrA(Z)$ generated by the pentagram spheres is only $\bZ_2$-graded, as in Section \ref{Sec:Indices} one can impose a $\bZ$-grading by considering the filtrations on Floer complexes coming from intersections of holomorphic polygons with the polar divisor $\Delta$ of a meromorphic 1-form.  We will choose the divisor $\Delta$ disjoint from the pentagram Lagrangians, at which point one can achieve transversality with perturbations of the almost complex structure and Floer equation supported away from $\Delta$.  The reason this is useful is that the term $\Delta \cdot [\im(u)]$ in Equation \ref{Eqn:IndexFormula} is necessarily \emph{positive}, by positivity of intersections of holomorphic curves with holomorphic divisors. If one knows something about the absolute indices $\iota^{\eta}(x_j)$ of intersection points, this can be a non-trivial constraint on which polygons exist.  As an illustration, we will extract slightly more from the argument of Lemma \ref{Lem:NoTriangle}, to give at least partial information on the cubic term $\mu^3_{\scrA(Z)}$ of the $A_{\infty}$-structure.
 
Choose a fibrewise hyperplane section of $Z$ -- the pullback of a general hyperplane $H\subset \bP^{2g+1}$, which now plays the role of $H_{ref}$ from Section \ref{Section:Idempotent}  -- whose complement is an affine quadric bundle at least over a large open set in $\bC$ containing the finite set of pentagram matching paths.  (Globally, any choice of hyperplane $H\subset \bP^{2g+1}$ will fail to be transverse to a finite number of the quadrics in the pencil, and its complement will not globally be a $T^*S^{2g}$-bundle over $\bP^1$; but one can put the singular fibres near infinity, and they will make no difference to the computations undertaken below.)  Up to symplectomorphism, one can identify an open set in this $T^*S^{2g}$-bundle with an open set in a suitable family of Milnor fibres 
\[
W \ = \ \left\{(x_0,\ldots,x_{2g},y) \, \big | \, h(x_0) + \sum_{i>0} x_i^2 =y^2 \right\} \ \subset \ \bC^{2g+2}
\]
as studied by Khovanov and Seidel in \cite{KhovanovSeidel}, taking $h(z) = z(z^{2g+1}-1)$ to obtain a family  of affine quadrics with nodal singular fibres at the origin and the $(2g+1)$-st roots of unity. Since this open subset of $Z_{\neq\infty}$ is simply-connected, as are the Lagrangians, we know \emph{a priori} that gradings, which come from trivialisations of $K^{\otimes 2}$ and associated phase functions, are unique.

\begin{Lemma} \label{Lem:Bigrade}
The $\bZ$-gradings on the pentagram algebra inside $W$ are $\bZ_{2g+1}$-equivariant.
\end{Lemma}

\begin{proof}
For any $M=g^{-1}(0) \subset \bC^{n+1}$, there is a holomorphic $n$-form $\eta$ on $M$  given by 
\begin{equation} \label{Eqn:VolForm}
\textrm{det} _{\bC} (\overline{dg_x} \wedge \ldots ) |_{TM}
\end{equation}
for which the phase of an oriented Lagrangian $n$-plane with orthonormal basis $\langle e_1,\ldots, e_n\rangle$ is given by $\det_{\bC} (\overline{dg_x} \wedge e_1 \ldots \wedge e_n)^2$.  Starting from here, \cite[Equation 6.5]{KhovanovSeidel} shows that for a Lagrangian $L_{\gamma} \subset W$ associated to a matching path $\gamma(t)$, the phase function is
\begin{equation} \label{Eqn:Phase}
t\ \mapsto \ h(\gamma(t))^{2g-1}\gamma'(t)^2
\end{equation}
where $h(z) = z(z^{2g+1}-1)$ and the phase is $O(2g+1)$-invariant so depends only on the base co-ordinate.  \eqref{Eqn:Phase} extends continuously over the two points of the sphere $L_{\gamma} \subset W \subset Z_{\neq\infty}$ lying over the endpoints of $\gamma$, and  is invariant under rotation by $exp(2i\pi/(2g+1))$, so the phase functions for the different pentagram spheres $L_j$ are identical under natural parametrisations. Absolute gradings in Floer theory are determined by local geometry near the intersection point and the associated phase function; here, both are invariant under rotation.
\end{proof}

\begin{Remark}[Notation] Recall in Figure \ref{Fig:Surgery} the three Lagrangians are $L''$ (vertical), $L$ (lower) and $L'$ (upper).  We will write $e$, $f$ and $p$ for three Floer generators (or their $\bZ_{2g+1}$-orbits) arising as follows:  $p$ is the external vertex of the pentagram, graded in odd degree as the generator of $HF(L,L')$ (as in the proof of Lemma \ref{Lem:NoTriangle}); and $e, f$ are the generators of $H^*(S^{2g})[1] = HF(L'',L)$ (which are also, by equivariance, the odd degree generators of $HF(L', L'')$, cf. \emph{op. cit.}).  We will write $\bar{e}, \bar{f}$ for the Poincar\'e dual generators, and $\iota^{\eta}(\cdot) = |\cdot|$ for the absolute index of any generator. We declare $|f| = |e|+2g$, so $e$ arises from the cohomological identity of the $S^{2g}$-clean intersection and $f$ from the fundamental class.
\end{Remark}

We now take a holomorphic volume form $\eta_Z$ on $Z$ which has poles of order $g$ along $H_{ref}$ and simple poles along the fibre $Z_{\infty}$, recalling that $c_1(Z) = (2g+2)H-E = 2gH + [\textrm{Fibre}]$.  Up to homotopy, the induced volume form on the subset $W \subset Z_{\neq\infty}$ agrees with that from \eqref{Eqn:VolForm}, so the corresponding gradings are $\bZ_{2g+1}$-invariant.  Note that although the $\bZ_2$-graded $A_{\infty}$-structure on $\scrA(Z)$  can  be taken to be invariant under the involution of Remark \ref{Rem:TwoVblesNeeded}, the $\bZ$-gradings need not be, since $\Delta$ is not invariant.   Poincar\'e duality in Floer cohomology implies that if $x \in CF(L,L')$ is a transverse intersection, then 
\[
\iota^{\eta}(x) = (2g+1) - \iota^{\eta}(\bar{x})
\]
where $\bar{x} \in CF(L',L)$ is the morphism viewed in the opposite direction.

 The proof of Lemma \ref{Lem:NoTriangle} relied only on the fact that the Lagrange surgery $L\# L'$ was Hamiltonian isotopic to a Lagrangian sphere disjoint from $L''$.  This is true inside the affine open subset $W$, so the holomorphic triangles whose existence was implied by Lemma \ref{Lem:NoTriangle} actually exist in $W$, in particular have trivial intersection number with the divisor $\Delta$. This provides a constraint on the absolute gradings of the Floer generators, via Equation \ref{Eqn:IndexFormula}:
 \[
|e| + |p| \ \textrm{and} \  |f| + |p| \, \in \{ |\bar{e}|, |\bar{f}|\} \, = \, \{2g+1-|e|, 1-|e|\}
\]
where the two equations come from the fact that $\mu^2(p,\cdot)$ was an isomorphism in Lemma \ref{Lem:NoTriangle}, so of rank 2.  On the other hand, we also know
\[
|f| = |e|+2g
\]
The two previous equations imply $2|e|+|p| =1.$  Now we consider cubic terms:
\begin{equation} \label{Eqn:CubicDegrees}
\mu^3(e,e,p); \quad \mu^3(e,f,p); \quad \mu^3(f,f,p). 
\end{equation}
We are interested in the coefficient of the identity $1_L$ in these cubic terms, which are pieces of the  product
\[
\mu^3: CF^*(L'', L) \otimes CF^*(L', L'') \otimes CF^*(L,L') \longrightarrow CF(L,L)[-1]
\]
Consider first working in $Z_{\neq\infty}$.  One can import some information about the higher $A_{\infty}$-structure on $\scrA(Z_{\neq\infty})$, using Lemma \ref{Lem:pentagramsagree} and our knowledge of $\scrA(C_{\neq\infty})$ arising from the results of Seidel and Efimov described in Section \ref{Section:Curve}.  Recall from \eqref{Eqn:Maurer-Cartan-FirstTerms} that the cubic term of an $A_{\infty}$-structure extending a given product defines a Hochschild cocycle;  exponentiated infinitesimal gauge transformations
\[
\alpha \mapsto \alpha - \partial \gamma + [\gamma,\alpha] + \frac{1}{2} [\gamma, [\gamma,\alpha]-\partial\gamma] + \cdots
\]
by elements $\gamma \in CC^0$ with vanishing linear term $\gamma = (\gamma^i)_{i\geq 2}$, so preserving the product, preserve the cohomology class $\Phi^1(\alpha^3) \in HH^*$ of $\alpha^3$, which is accordingly well-defined, cf. \cite[Eqn. (5.6)]{Seidel:HMSgenus2}. In this vein,  we record one implication of Seidel and Efimov's work.

\begin{Lemma} \label{Lem:CubicOpen}
In $Z_{\neq\infty}$, the product $\mu^3(e+f, e-f, p)$ contains $1_L$ with non-zero coefficient, whilst $\mu^3(e+f,e+f,p)$ is trivial.
\end{Lemma}

This follows from Equation \ref{Eqn:AinftyCurve}; in that notation, the first statement of the Lemma reflects the non-trivial cubic term $v_1v_2v_3$ whilst the second reflects the  vanishing of the terms $v_i^2v_3$, $i=1,2$.   We now combine Lemma \ref{Lem:CubicOpen} with the $\bZ$-grading considerations discussed before \eqref{Eqn:CubicDegrees}.   Lemma \ref{Lem:DeformationDegree} implies that intersections with $H_{ref}$ respectively $Z_{\infty}$ contribute $4g$ respectively $2$ to the Maslov index of the corresponding holomorphic polygon.  Indeed, for a rigid (virtual dimension zero) polygon contributing to $\mu^k$, with inputs $x_j$ and output the identity element $1_L \in HF^0(L,L)$, Lemma \ref{Lem:DeformationDegree} implies
\begin{equation} \label{Eqn:Virdim0Degree}
\sum_j |x_j| + 2-k \, = \, \Delta \cdot [\im(u)] = 4g\, \im(u)\cdot H_{ref} + 2\, \im(u) \cdot Z_{\infty}.
\end{equation}
We now put together the following pieces of information:
\begin{itemize}
\item By Remark \ref{Rem:TwoVblesNeeded}, the $A_{\infty}$ structures on $Z_{\neq\infty}$ and $Z$ can be taken invariant under an involution which preserves $p$ and $e$ but sends $f\mapsto -f$, so $\mu^3(e,f,p) = 0$ in both cases;
\item Lemma \ref{Lem:CubicOpen} then implies $\mu^3_{\neq\infty}(e,e,p) \neq \mu^3_{\neq\infty}(f,f,p)$ when computed in $Z_{\neq\infty}$.  By Equation \ref{Eqn:Virdim0Degree}, and the discussion before \eqref{Eqn:CubicDegrees}, the holomorphic polygons which contribute to  these expressions are necessarily disjoint from $\Delta$, for $\mu^3_{\neq\infty}(e,e,p)$, respectively meet $\Delta$ exactly once, for $\mu^3_{\neq\infty}(f,f,p)$.
\end{itemize}

Now consider the term $\mu^3_{\scrA(Z)}(e,e,p)$.  Equation \ref{Eqn:Virdim0Degree} and the relation $2|e|+|p|=1$ implies that no holomorphic polygon contributing to this product can meet either $H_{ref}$ or $Z_{\infty}$, in particular its value does not change under compactification from $Z_{\neq\infty}$ to $Z$.  There are in principle new holomorphic polygons which contribute to $\mu^3_{\scrA(Z)}(f,f,p)$, given by adding a Chern number zero sphere in homology class $2gR-L$ to a disk living in $Z_{\neq\infty}$.  (Since our total space is monotone, such a Chern number zero sphere cannot itself be represented by a holomorphic curve, so these additional polygons do not arise from broken configurations by gluing on  sphere bubble components.)  However, we can deduce:

\begin{Corollary} \label{Cor:CubicClosed}
The coefficient of $1_L$ in exactly one of  
\[
\mu^3_{\scrA(Z)}(e+f, e-f, p) \quad \textrm{and} \quad \mu^3_{\scrA(Z)}(e+f,e+f,p)
\]
 is non-zero.
\end{Corollary}

\begin{proof}
This is immediate from the preceding discussion. Either the coefficient of $1_L$ in $\mu^3(f,f,p)$ does not change on compactification, in which case we inherit the non-vanishing of Lemma \ref{Lem:CubicOpen}, or it does change, but since the analogous coefficient didn't change in $\mu^3(e,e,p)$, that would imply the second outcome. 
\end{proof}

\begin{Remark} It is possible to compute the absolute grading of $\scrA(Z)$ using the theory of ``bigraded curves" developed in \cite{KhovanovSeidel}, given a computer to plot phase functions along matching paths (and a blackboard to extract the gradings from those plots).  Knowledge of $\scrA(Z_{\neq\infty})$ gives many further \emph{a priori} constraints on the bigrading; for instance, \eqref{Eqn:AinftyCurve} implies $\mu^{2g+1}_{\neq\infty}(p,\ldots,p) \neq 0$, which together with Lemma \ref{Lem:DeformationDegree} shows that $|p| = 4gk-1$ for some $k\geq 1$, etc.
\end{Remark}

\subsection{Finite determinacy\label{Section:FiniteDeterminacy}}

We return to the situation of Section \ref{Section:Formality}, and an $A_{\infty}$-structure on an exterior algebra $A=\Lambda^*(V)$.  Suppose now $V \cong \bC^3$ is 3-dimensional, so $A$ has graded summands
\[
A^0 = \bC \oplus \Lambda^2(V); \quad A^1 = V \oplus \Lambda^3(V).
\]  
According to \cite[Theorem 3.3]{Seidel:HMSgenus2}, the $L_{\infty}$-quasi-isomorphism $\Phi$ of Theorem \ref{Thm:KontsevichFormality} has linear term $\Phi^1$ which is given by restricting elements of $Hom(\Lambda(V)^{\otimes i}, \Lambda(V))$ to the subspace of symmetric elements of $Hom(V^{\otimes i}, \Lambda(V))$ (this is the Hochschild-Kostant-Rosenberg map).  Since $V$ is in odd degree, $\alpha^i(v_1,\ldots, v_i) \in \Lambda^0 \oplus \Lambda^2$ for all $i$ and any $v_j\in V$.
In other words, taking the parities of the $\alpha^i$ into consideration, a Maurer-Cartan element comprises a pair
\[
(W, \eta) \in \bC[[V^{\vee}]] \oplus \bC[[V^{\vee}]]\otimes \Lambda^2(V)
\]
of a formal function and a formal 2-form, satisfying
\[
[W, W] = 0; \quad [\eta, \eta]=0; \qquad [W, \eta]=0.
\]  
Since the given $A_{\infty}$-structure does not deform the algebra structure on $\Lambda(V)$, we may suppose that $W$ and $\eta$ have no coefficients of polynomial degree $<3$.
Note that the equation $[W, W]=0$ always holds, since the Schouten bracket on polyvector fields involves contraction of the differential form, hence vanishes identically on $0$-forms.  In particular, any formal function $W$ together with the trivial 2-form $(W,0)$ defines a $\bZ_2$-graded $A_{\infty}$-structure on $\Lambda(\bC^3)$.  Gauge transformations are now by pairs
\[
(g^1,g^3) \in \left(\bC[[V^{\vee}]]\otimes \Lambda^1(V)\right) \oplus \left(\bC[[V^{\vee}]]\otimes \Lambda^3(V)\right).
\]
For vector fields vanishing at the origin the first term acts on both components $(W, \eta)$, by pullback by the diffeomorphism obtained from exponentiating the vector field.  Explicitly, a vector field $g^1$ acts on a formal function through the adjoint action
\begin{equation} \label{Eqn:ExpAdjoint}
exp(g^1)\cdot W \ = \ W + \sum_{n \geq 0} \frac{ad(g^1)^{n}}{(n+1)!} \, [g^1,W]
\end{equation}
 The second term acts by interior contraction
\begin{equation} \label{Eqn:gauge}
(0,g^3) \cdot (W, \eta) \ = \ (W, \eta + \iota_{dW}(g^3)).
\end{equation}
In this language, we can re-express the outcome of the previous section.  Pick co-ordinates $\xi_1, \xi_2, \xi_3$ on $V$ and dual co-ordinates $(v_1,v_2,v_3)$ on $V^{\vee}\cong \bC^3$.  Write $O(k)$ for power series all of whose terms have degree $\geq k$.   The action of $G=\bZ_{2g+1}$ on $V$ comes geometrically from the rotation action on $\bP^1$ permuting Lagrangian pentagram spheres, so the generator of $G$ acts on the hyperelliptic curve 
\[
y^2 = z(z^{2g+1}-1) \quad  \textrm{via} \quad (y,z) \mapsto (\xi^{g+1}y, \xi z), \ \textrm{with} \ \xi=e^{\frac{2i\pi}{2g+1}}.
\]
Via Lemma \ref{Lem:pentagramsagree}, there is a basis with respect to which $G$ acts on $V$  by diagonal matrices with entries which are non-trivial roots of unity; explicitly, there is a basis for which $G$ acts via the diagonal matrix with entries $(\xi,\xi,\xi^{2g-1})$.  

\begin{Lemma} \label{Lem:Gotyou}
The formal function $W$ defining the $A_{\infty}$-structure on $\scrA(Z)$ has the form
\[
W = \lambda v_1 v_2 v_3 + O(4) \quad \textrm{or} \quad W = \lambda (v_1^2 +v_2^2)v_3 + O(4).
\]
for some non-zero $\lambda \in \bC^*$.
\end{Lemma}

\begin{proof}
The generators $v_i$ of the exterior algebra correspond, up to scale, to the Floer generators $e+f, e-f, p$ of Section \ref{Sec:Cubic}.  Hyperelliptic invariance, exchanging $f$ and $-f$, corresponds to invariance under $v_1 \leftrightarrow v_2$.  Corollary \ref{Cor:CubicClosed} implies that either the term $v_1v_2v_3$ in $W$ has non-zero coefficient, or the term $v_1^2v_3$ (and hence by symmetry $v_2^2v_3$) has non-zero coefficient. Moreover, these two cases are mutually exclusive.  The two cases cover the only possible cubic terms in $W$, by $\bZ_{2g+1}$-invariance, which gives the result.
\end{proof}

We now  explore more systematically the constraints imposed by $\bZ_{2g+1}$-equivariance.  One can consider only $G$-equivariant $\bZ_2$-graded $A_{\infty}$-structures on $\Lambda(V)$ by imposing $G$-equivariance throughout the preceding discussion; in particular, the inverse $\Phi$ to Kontsevich's map $\Psi$ gives a quasi-isomorphism between equivariant Hochschild cochains and equivariant polyvector fields, and $G$-equivariant gauge transformations act on the set of such $G$-equivariant Maurer-Cartan solutions.

As in Section \ref{Section:Formality}, Lemma \ref{Lem:DeformedHH} and the Formality Theorem \ref{Thm:KontsevichFormality} imply that the Hochschild cohomology of a $\bZ_2$-graded $A_{\infty}$-structure on $\Lambda(V)$ defined by a pair $(W,\eta)$ is given by the cohomology of the $\bZ_2$-graded complex 
\begin{equation} \label{Eqn:HHPolyFields}
\bC[[V^{\vee}]] \oplus (\bC[[V^{\vee}]]\otimes \Lambda^2(V)) \xrightleftharpoons[{\ [\cdot, W+\eta] \oplus [\cdot,W]\ }]{{\ [\cdot, \eta] \oplus [\cdot,W+\eta]\ }} (\bC[[V^{\vee}]]\otimes \Lambda^1(V))  \oplus (\bC[[V^{\vee}]]\otimes \Lambda^3(V))
\end{equation}
noting that $\xi \mapsto [\xi, W] = \iota_{dW}(\xi)$ vanishes on $0$-forms, and similarly $\xi \mapsto [\xi,\eta]$ vanishes on $3$-forms, for degree reasons. The group $HH^*(\scrA, \scrA)^{\bZ_{2g+1}}$ is computed by the complex
\begin{equation} \label{Eqn:HHPolyFieldsInvt}
\left( \bC[[V^{\vee}]] \otimes \Lambda^{ev}(V)) \right)^{\bZ_{2g+1}} \  
\xrightleftharpoons[{\ [\cdot, W+\eta] \ }]{{\ [\cdot,W+\eta]\ }} \ \left( \bC[[V^{\vee}]]\otimes \Lambda^{odd}(V) \right)^{\bZ_{2g+1}}  
\end{equation}
Note that $1 \in \bC[[V^{\vee}]]$ always defines a non-trivial class in $HH^{ev}$, since $HH^*$ is a unital ring.  Therefore, Corollary \ref{Cor:QHRelativeQuadric} implies that in our case there is precisely one other non-trivial class.  The Maurer-Cartan equation implies that both $W$ and $\eta$ are themselves cocycles, so some some linear combination of them must be a coboundary.  The force of Lemma \ref{Lem:Gotyou} is that we can prove that $\eta$ is in fact already a coboundary in the Koszul complex associated to $W$.

The Koszul complex of a holomorphic function $W$ is acyclic in degrees greater than the dimension of the critical locus of $W$, see  \cite[Ch.~6, Prop.~2.21]{Dimca}. The underlying acyclicity result applies more generally to complexes of finitely generated free modules over a Noetherian local ring $R$, which is useful in the context of formal rather than holomorphic functions.  Indeed, \cite[Corollary 1]{BuchsbaumEisenbud} or \cite[Section 20.3]{Eisenbud} says that a complex of free $R$-modules
\[
0 \rightarrow F_n \xrightarrow{\phi_n} F_{n-1} \xrightarrow{\phi_{n-1}} \cdots \rightarrow F_2 \xrightarrow{\phi_2} F_1
\]
is exact if and only if
\begin{enumerate}
\item $\rk(F_k) = \rk(\phi_k) + \rk(\phi_{k+1})$ and 
\item the depth of the ideal $I(\phi_k) \geq k-1$
\end{enumerate}
 for $k=2,\ldots, n$. Here, the rank of $\phi_k$ is the size of the largest non-vanishing minor in a matrix representing $\phi_k$, and $I(\phi_k)$ is the ideal in $R$ generated by the determinants of all $r\times r$ minors, where $r=\rk(\phi_k)$.  $\C[[v_1,v_2,v_3]]$ is Noetherian \cite[Corollary 10.27]{AtiyahMacDonald} and Cohen-Macaulay.  In a Cohen-Macaulay ring, depth of an ideal (maximal length of a regular sequence) equals its codimension \cite[Theorem 18.2]{Eisenbud}, and the criterion above, applied  to the Koszul complex truncated in degree $> k$ of a holomorphic $W$ with critical set of dimension  $k$, recovers the familiar acyclicity result stated previously. Taking $W$ to be a formal function in 3 variables, we would like to know when the truncated complex
 \[
 0 \rightarrow \Lambda^3 \xrightarrow{[\cdot, dW]} \Lambda^2  \xrightarrow{[\cdot, dW]} \Lambda^1
 \]
 is acyclic, with $\Lambda^i$ the free module of rank ${3 \choose i}$ over $R=\bC[[v_1,v_2,v_3]]$. It is easy to check that whenever $W\neq 0$, the first condition on the ranks of the maps and modules is satisfied.  Since $R$ contains no zero-divisors, depth$\, I(\phi_2) \geq 1$ is trivial, so the key condition is 
 \[
 \textrm{depth} \, I(\phi_3) \geq 2.
 \]
 This certainly holds if one can order the partial derivatives $(\partial_i W, \partial_jW, \partial_kW)$ in such a way that some linear combination $u.\partial_i W+ v.\partial_jW$ is not a zero-divisor in $R/ \langle \partial_k W\rangle$. If $u\partial_iW+ v\partial_jW$ is such a zero-divisor, then
 \[
 \phi \cdot (u\partial_i W + v\partial_jW) = \psi \partial_k W
 \]
 for some formal functions $\phi, \psi$ with $\partial_k W \not| \, \phi$.   Since $R$ is moreover a unique factorisation domain, this can happen for all $u$, $v$  only if the partial derivatives share a common irreducible factor (which is not a unit, so has vanishing constant term). However, this is precluded by Lemma \ref{Lem:Gotyou}.  We deduce:
 
 \begin{Lemma}
$\scrA(Z)$ is defined by a Maurer-Cartan pair $(W,\eta)$ for which $\eta$ is a coboundary in the Koszul complex associated to $W$, so there is some formal 3-form $g^3$ for which $[g^3,dW] = \eta$.
 \end{Lemma}

We can therefore apply a gauge transformation as in Equation \ref{Eqn:gauge}  to kill $\eta$, so up to gauge equivalence, the $A_{\infty}$-structure on $Z$ is defined by a formal function $W$ satisfying Lemma \ref{Lem:Gotyou}, together with the trivial 2-form. Gauge transformations of $A_{\infty}$-structures induce the identity on cohomology. There are also $A_{\infty}$-equivalences which act non-trivially on cohomology; in particular,  $A_{\infty}$-structures on $\Lambda(V)\rtimes G$ are invariant under $G$-equivariant linear transformations in $GL(V)$.  Since the cubic polynomials 
\[
v_1v_2v_3 \quad \textrm{and} \quad (v_1^2+v_2^2)v_3
\]
are related by a linear transformation which commutes with $G$ (which acts by $\xi \cdot \id$ on the subspace $\langle v_1, v_2\rangle \subset V^{\vee}$), we can suppose up to quasi-isomorphism that 
\begin{equation} \label{Eqn:MoreInfo}
W = \lambda v_1v_2v_3 + O(4) \quad \textrm{and} \quad \eta \equiv 0.
\end{equation}
 We next recall some background on Hochschild cohomology of semi-direct products, see for instance \cite{Shepler-Witherspoon}.  The Hochschild cohomology of a semi-direct product of a finite-dimensional algebra with a finite group $\Gamma$ has an eigenspace-type decomposition over conjugacy classes of $\Gamma$.  In particular, $HH^*(\scrA\rtimes \bZ_n, \scrA\rtimes \bZ_n)$ splits into a piece corresponding to $id \in \bZ_n$ and $(n-1)$ other summands indexed by the other characters:
\begin{equation} \label{Eqn:Characters}
HH^*(\scrA\rtimes \bZ_n, \scrA\rtimes \bZ_n) = \bigoplus_{\gamma \in \bZ_n} Ext_{bimod-\scrA}^*(\scrA, Graph(\gamma))^{\bZ_n} 
\end{equation}  The summand for $\gamma = 1$ is canonically identified with the invariant part $HH^*(\scrA,\scrA)^{\bZ_n}$. In the special case of an $A_{\infty}$-structure on the exterior algebra, each of the other summands is computed by a cochain complex which actually defines a Koszul resolution of $\bC$, hence has one-dimensional cohomology \cite[Section 4b]{Seidel:HMSquartic}, whilst the invariant part of $HH^*(\scrA,\scrA)$ can be computed from the complex of Equation \ref{Eqn:HHPolyFields}.
As an illustration, we explicitly verify that the algebraic description of the Fukaya category $D^{\pi}\scrF(\Sigma_g) \simeq D(\scrA\rtimes \bZ_{2g+1})$  given in Section \ref{Section:Curve} is consistent with the prediction of Corollary \ref{Cor:QuiltsQH}. 

\begin{Lemma}  \label{Lem:ComputeonCurve}
$HH^*(D^{\pi}\scrF(\Sigma_g)) \simeq H^*(\Sigma_g)$ as $\bZ_{2g+1}$-representations.
\end{Lemma}

\begin{proof}
From the preceding discussion, we identify the invariant part $HH^*(\scrA,\scrA)^{\bZ_{2g+1}}$ with the invariant part of the \emph{Jacobian ring} 
\[
J(Q) = \bC[[x_1,x_2,x_3]] / \langle \partial_1 Q, \partial _2 Q, \partial_3 Q\rangle
\]
under the inherited $\bZ_{2g+1}$-action, where $Q = -x_1x_2x_3 + x_1^{2g+1}+x_2^{2g+1}+x_3^{2g+1}$ is the polynomial superpotential controlling the $A_{\infty}$-structure on $A = \Lambda(\bC^3)$.  In our case, the invariant part $J(Q)^{\bZ_{2g+1}}$ is generated by the polynomials $1$ and $x_1x_2x_3$, noting that modulo the ideal of relations $\langle \partial_i Q\rangle$ one has
\[
(x_1x_2x_3)^2 \sim (x_1x_2x_3)^{2g}
\]  
and that $1-(x_1x_2x_3)^{2g-2}$ is invertible since we work over formal power series, hence $x_1x_2x_3$ has order 2 in $J(Q)^{\bZ_{2g+1}}$.  The upshot is that the eigenspace decomposition of $HH^*$ under the $\bZ_{2g+1}$-action has a two-dimensional zero-eigenspace and $2g$ one-dimensional eigenspaces. This matches the cohomology (= quantum cohomology) of $\Sigma_g$, with $\bZ_{2g+1}$ acting on $H^1$ with minimal polynomial $1+\lambda + \lambda^2+\lambda^3+\cdots+\lambda^{2g}$ and acting trivially on $H^0 \oplus H^2$.
\end{proof}

The subcategory $\scrA(Z)$ is determined by a $\bZ_{2g+1}$-equivariant Maurer-Cartan pair $(W,0)$ as in \eqref{Eqn:MoreInfo}.  Corollary \ref{Cor:QHRelativeQuadric} and the discussion around Equation \ref{Eqn:Characters} imply that the Hochschild cohomology $HH^*(\scrA \rtimes \bZ_{2g+1},\scrA \rtimes \bZ_{2g+1})$ has $2g$ one-dimensional summands for the non-identity elements, living in odd degree, and two even-degree classes (corresponding to $H^0$ and $H^2$ of the genus $g$ curve).

\begin{Lemma} \label{Lem:HHnotsmall}
If the formal function $W = \lambda v_1v_2 v_3 + O(4)$, $\lambda \neq 0$, determines an $A_{\infty}$-structure $\scrA$ on $\Lambda(\bC^3)\rtimes \bZ_{2g+1}$ for which $HH^*(\scrA,\scrA)^{\bZ_{2g+1}}$ has rank 2, then the coefficients of the 3 monomials $v_i^{2g+1}$ in $W$ are all non-zero.
\end{Lemma}

\begin{proof}
 The Hochschild cohomology of the $A_{\infty}$-structure is the cohomology of the Koszul complex
 \begin{equation} \label{Eqn:Koszul}
0 \rightarrow \bC[[V^{\vee}]] \otimes \Lambda^3(V) \xrightarrow{\iota_{dW}}  \bC[[V^{\vee}]] \otimes \Lambda^2(V) \xrightarrow{\iota_{dW}} \bC[[V^{\vee}]]\otimes \Lambda^1(V) \xrightarrow{\iota_{dW}} \bC[[V^{\vee}]] \rightarrow 0.
\end{equation}
 This contains in degree zero a copy of the Jacobian ring $J(W) = \bC[[V^{\vee}]] / \langle \partial_iW \rangle $, which must therefore contain a unique non-identity element.  
 
 If the coefficient $\mu_i$ of the monomial $v_i^{2g+1}$ is zero, then $v_i^{2g+1}$ will survive into the Jacobian ring (since this is the lowest monomial term that can occur in $W$ by $\bZ_{2g+1}$-invariance). Therefore, at most one of the $\mu_{i>0}$ can be zero, in fact $\mu_3$ since the $A_{\infty}$-structure is invariant under permuting $v_1$ and $v_2$; and in this case $W$ must contain a monomial $v_3^{4g+2}$ (or the element $[v_3^{2g+1}] \in J(W)$ will not have square zero). However, if $W = \lambda v_1v_2v_3 + \mu(v_1^{2g+1} + v_2^{2g+1}) + \mu_3v_3^{4g+2} + p$ where all terms of $p$ of degree $< 2g+1$ involve at least two variables, then $[v_1v_2v_3]$ and $[v_3^{2g+1}]$ give distinct elements of $J(W)$, which again contradicts the invariant part of $HH^*$ having rank 2.
 \end{proof}

We recall Noether's classical result \cite{Noether} that the invariant subring $\bC[v_1, v_2, v_3]^G$ under a finite group $G$ is generated by monomials of degree at most the order of $G$.  Let $Q \in \bC[[V^{\vee}]]^G$ denote 
\[
Q = -v_1v_2v_3 + v_1^{2g+1} + v_2^{2g+1} + v_3^{2g+1}.
\]
The following is a variant of Seidel's \cite[Lemma 4.1]{Seidel:HMSgenus2}, whose proof we follow closely.  

\begin{Lemma} \label{Lem:FinitelyDetermined}
Pick constants $(\lambda,\mu_1,\mu_2,\mu_3) \in (\bC^*)^4$. Let $p(\cdot)$ be a polynomial of degree $4 \leq deg(p) \leq 2g+1$, with vanishing coefficients of the monomials $v_i^{2g+1}$.
The $A_{\infty}$-structure defined by any $\bZ_{2g+1}$-invariant Maurer-Cartan pair $(W,0)$ with
\[
W = \lambda v_1v_2v_3  + p(v_1, v_2, v_3) +  \mu_1 v_1^{2g+1} + \mu_2v_2^{2g+1} + \mu_3v_3^{2g+1} + O(2g+2) \] 
is $A_{\infty}$-isomorphic to the structure defined by $(Q,0)$.
\end{Lemma}

\begin{proof}
To be maximally explicit, we suppose first $g=2$, so $p(\cdot) = 0$, and explain the modifications for the general case at the end.   Let $W_{\mu}$ denote the degree $\leq 5=2g+1$ part of the given formal function $W$.   Letting $I$ denote the ideal generated by the partial derivatives $I = \langle \partial_1 W_{\mu}, \partial _2W_{\mu}, \partial_3W_{\mu}\rangle$, one checks that:
\begin{itemize}
\item $v_iv_j \in I + O(4)$ 
\item  $v_i^{6} \in I\cdot O(2) + O(8)$.
\end{itemize}
Explicitly, to check the second claim, one starts from $\partial_1 W_{\mu} \in I$ and writes
\[
5 \mu_1 v_1^4  \in I + \mu_0 v_2v_3 \ \Rightarrow \ v_1^6 \in I\cdot O(2) + \bC\cdot v_1^2v_2v_3
\]
and then substitutes 
\[
v_1^2v_2v_3 \ = \ (v_1v_2)\cdot (v_1v_3)  \, \in (I - \bC v_3^4)\cdot (I - \bC v_2^4).
\]
Given any $n$, direct manipulation now yields a co-ordinate transformation $exp(g^1)$ (which can be averaged to be $G$-equivariant) for which 
\[
exp(g^1)\cdot (W) = W_{\mu}+O(n).
\]
Explicitly, if $W = W_{\mu} + O(6)$ is $\bZ_5$-equivariant, it contains no monomials $v_i^t$ for $t=6,7,8$.  Then $W-W_{\mu}$ is a sum of terms $v_j.v_k.(\cdot)$ and terms of order $\geq 9$, hence $W-W_{\mu} \in I\cdot O(\geq 4) + O(\geq 9)$.  For a suitable change of variables
\[
v_j' = v_j + f_{4,j}(v); \qquad f_{4,j} \in O(4)
\]
one sees $W(v_1',v_2',v_3')$ is given by
\[
W(v_1,v_2,v_3) + \sum_j f_{4,j} \partial_j W + \cdots + O(9)
\]
which we write as:
\[
W + \sum_j f_{4,j} \partial_j W_{\mu} + \sum_j f_{4,j} (\partial_j W - \partial_j W_{\mu}) +O(9)
\]
Now using the bullet point above, one can choose the $f_{4,j}$ so that the first two terms give exactly the part of $W-W_{\mu}$ in $I\cdot O(4)$, and all the other Taylor co-efficients are $O(9)$, so we have improved $n$ from $6$ to $9$.  Iterating the procedure, one can increase $n$ arbitrarily.   Since $W_{\mu}$ has an isolated singularity at the origin, the general finite determinacy theorem from singularity theory  \cite{MatherYau} implies that, once $n$ is large enough, there is a formal change of variables $exp(h^1)$ for which
\[
exp(h^1)\cdot W \, = \, W_{\mu}.
\]
Recall that  $A_{\infty}$-structures on $\Lambda(V)\rtimes G$ are invariant under $G$-equivariant linear transformations in $GL(V)$. In our situation they are invariant under arbitrary invertible linear rescalings of the $v_i$. This brings any polynomial $W_{\mu}$ of the given form into the shape $tv_1v_2v_3 + v_1^5+v_2^5+v_3^5$. A further linear co-ordinate change will bring this into the form $-\epsilon v_1v_2v_3 + \epsilon^3(v_1^5+v_2^5+v_3^5)$, for some $\epsilon \in \bC^*$.  One now uses the existence of a canonical $\bC^*$-action (again not by gauge transformations, but nonetheless by $A_{\infty}$-isomorphisms) on the space of $A_{\infty}$-structures, where $\epsilon \in \bC^*$ acts by rescaling $\mu^j$ by $\epsilon^{j-2}$, to reduce to the desired polynomial $Q$.  This completes the argument when $g=2$.

The general case, when $g>2$, is similar. The polynomial $p$ actually lies in $O(d)$ for some $d>4$, for invariance reasons, and contains no monomial terms $v_i^j$ (the only invariant monomial terms of degree $<2g+2$ are $v_i^{2g+1}$, for which $p$  has vanishing coefficients by hypothesis).  Therefore, its lowest degree $d$ piece can be written (not necessarily uniquely) as
\[
p_d(v_1,v_2,v_3) = v_1v_2q_3 + v_2v_3q_1 + v_3v_1q_2
\]
for $G$-invariant homogeneous polynomials $q_i \in O(d-2)$ of degree $>2$.  There is an infinitesimal gauge transformation
\[
v_i \mapsto v_i' = v_i-q_i \ = \ v_i + [-q_i \xi_i, v_i]
\]
associated to the 1-form $-\sum q_i \xi_i$, and since $q_i \in O(2)$, this exponentiates to a gauge transformation which takes the given $W$ to a polynomial $W'$ of the same shape, but for which the polynomial $p$ has larger degree: 
\[
W' =  \lambda v_1'v_2'v_3'  + \tilde{p}(v_1',v_2', v_3') +  \mu_1 (v_1')^{2g+1} + \mu_2(v_2')^{2g+1} + \mu_3(v_3')^{2g+1} + O(2g+2) \] 
with  $deg(\tilde{p}) \geq d+1$, because all the higher order terms in the exponentiated adjoint action \eqref{Eqn:ExpAdjoint} have larger polynomial degree. Iterating (finitely often, so there are no convergence issues), one reduces to the case when $p \in O(2g+2)$, so without loss of generality we can take $p\equiv 0$ in the statement of the Lemma. Now $v_iv_j \in I + O(2g)$, and explicit formal changes of variables as in the $g=2$ case  iteratively kill the $O(2g+2)$-part of $W$, cf. \cite[Lemma 3.1]{Efimov}.  Finally, the coefficients $\mu_j$ are adjusted using $GL(V)$-invariance and the canonical $\bC^*$-action on $A_{\infty}$-structures.
\end{proof} 

Note that Lemma \ref{Lem:ComputeonCurve} verifies that in the only non-excluded case, the Hochschild cohomology does have the anticipated rank.   Combining our knowledge of quantum cohomology, Corollary \ref{Cor:QHRelativeQuadric}, with Lemma \ref{Lem:HHnotsmall}, one finds that the $A_{\infty}$-structure on $\scrF(Z)$ satisfies the conditions of Lemma \ref{Lem:FinitelyDetermined}.  Since this singles out a unique quasi-isomorphism class, one concludes that the generalised pentagram spheres in $Z$ define an equivalent structure to the corresponding curves in the genus $g$ surface.  Since on the genus $g$ curve these Lagrangians split-generate, this proves that
\[
D^{\pi}\scrF(\Sigma_g) \hookrightarrow D^{\pi}\scrF(Z).
\]
Lemma \ref{Lem:Step1} and Lemma \ref{Lem:Step2} together show that this embedding has the same image, up to quasi-isomorphism, as the embedding $D^{\pi}\scrF(Q_0 \cap Q_1) \hookrightarrow D^{\pi}\scrF(Z)$ of Lemma \ref{Lem:FullyFaithful}. 

\begin{Corollary} \label{Cor:Compatible}
The embedding of categories $D^{\pi}\scrF(\Sigma_g) \hookrightarrow D^{\pi}\scrF(Z)$ is compatible with the natural weak actions of the pointed hyperelliptic mapping class group.  \end{Corollary}

Corollary \ref{Cor:Compatible} has precisely the same proof as Corollary \ref{Cor:MappingClassCompatible-1}, but invoking Lemma \ref{Lem:Step1} in place of Lemma \ref{Lem:Step2}.   Together, these results imply Addendum \ref{Thm:Action} from the Introduction.

\begin{Remark} \label{Remark:FiniteGroupEntwine}
The finite group $H^1(\Sigma_g;\bZ_2)\ni \xi$ has a natural weak action on $\scrF(\Sigma_g)$, tensoring by flat line bundles $\xi$, and on $\scrF(Q_0 \cap Q_1;0)$, by symplectic involutions $\iota_{\xi}$ of Remark \ref{Rem:BranchingInfinity}, and one can ask if the equivalence $\Upsilon: D^{\pi}\scrF(\Sigma_g) \simeq D^{\pi}\scrF(Q_0 \cap Q_1;0)$ we have constructed also entwines that action. We shall not prove this, but the relevant cohomological evidence is provided by Remark \ref{Remark:2objects-1} and Remark \ref{Rem:TwoHolChoices}.  Presumably, just as $\Upsilon(\gamma) \simeq V_{\bar\gamma}$, also $\Upsilon(\xi \rightarrow \gamma) \simeq \iota_{\xi} V_{\bar\gamma}$.  If true, one could compute the Hochschild cohomology of the autoequivalence $\otimes \,\xi \in Auteq(\scrF(\Sigma))$ in terms of the fixed point Floer cohomology $HF(\iota_{\xi})$ on $Q_0 \cap Q_1$.
\end{Remark}

%%%%%%%%%%%%%%%%%%%%%%%%%%%%%%%%%%%%%%%%%%%%%
%%%%%%%%%%%%%%%%%%%%%%%%%%%%%%%%%%%%%%%%%%%%%%
\bibliographystyle{amsplain}

\end{document}